\documentclass[12pt,reqno,a4paper]{amsart}            

\usepackage{tikz}
\usetikzlibrary{arrows}

\usepackage{hyperref}

\usepackage{scrpage2}    
\usepackage{scrtime}     

\usepackage{amsmath, amssymb}
\usepackage{amsthm}

\usepackage{accents}     
\usepackage{mathtools}  

\usepackage{bbm}         

\usepackage{mathrsfs}    

\usepackage{graphicx}                

\usepackage[totalwidth=485pt, totalheight=761pt]{geometry}

\usepackage{scrpage2}    
\usepackage{scrtime}     

{ 
  \lohead{\textsf{\shorttitle}}
  \rehead{\textsf{\authors}}
  \automark[subsection]{section}
  \automark[subsection]{section}

  \newcommand{\look}[1]{}%
  \newcommand{\lookO}[1]{}%
  \newcommand{\lookS}[1]{}%
}

\usepackage{graphicx}                

\usepackage{amsmath, amssymb}
\usepackage{amsthm}

\usepackage{amscd}

\usepackage{accents}     
\usepackage{mathtools}  

\usepackage{bbm}         

\usepackage{mathrsfs}    

\usepackage{xspace}  

\numberwithin{equation}{section}

\newcounter{myenumi}

\newcommand{\itemref}[1]{\eqref{#1}}

\newcommand{\myfont}{\sffamily}

\newcommand{\myparagraph}[1]{\noindent\textbf{\myfont{#1}}}

\newtheoremstyle{mythmstyle}
  {\topsep}
  {\topsep}
  {\itshape}
  {}
  {\bfseries \myfont}
  {.}
  {.5em}
  {}

\newtheoremstyle{mydefstyle}
  {\topsep}
  {\topsep}
  {\normalfont}
  {}
  {\bfseries \myfont}
  {.}
  {.5em}
  {}

\swapnumbers                    

\theoremstyle{mythmstyle}       

\newtheorem{theorem}{Theorem}[section]

\newtheorem{proposition}[theorem]{Proposition}
\newtheorem{lemma}[theorem]{Lemma}
\newtheorem{corollary}[theorem]{Corollary}

\newcounter{intro}

\theoremstyle{mydefstyle}        
\newtheorem{definition}[theorem]{Definition}

\newtheorem{example}[theorem]{Example}
\newtheorem{examples}[theorem]{Examples}

\newtheorem{remark}[theorem]{Remark}
\newtheorem{remarks}[theorem]{Remarks}
\newtheorem*{remark*}{Remark}

\expandafter\let\expandafter\oldproof\csname\string\proof\endcsname
\let\oldendproof\endproof
\renewenvironment{proof}[1][\bfseries\myfont\proofname]{%
  \oldproof[\bfseries \myfont #1]%
}{\oldendproof}

\makeatletter
\renewcommand\section{\@startsection{section}{1}%
  \z@{.7\linespacing\@plus\linespacing}{.5\linespacing}%
  {\Large\myfont\bfseries}}
\renewcommand\subsection{\@startsection{subsection}{2}%
  \z@{-.5\linespacing\@plus-.7\linespacing}{.5\linespacing}%
  {\large\myfont\bfseries}}
\renewcommand\subsubsection{\@startsection{subsubsection}{3}%
  \z@{.5\linespacing\@plus.7\linespacing}{-.5em}%
  {\myfont\bfseries}}
\renewenvironment{abstract}{%
  \ifx\maketitle\relax
    \ClassWarning{\@classname}{Abstract should precede
      \protect\maketitle\space in AMS document classes; reported}%
  \fi
  \global\setbox\abstractbox=\vtop \bgroup
    \normalfont\Small
    \list{}{\labelwidth\z@
      \leftmargin3pc \rightmargin\leftmargin
      \listparindent\normalparindent \itemindent\z@
      \parsep\z@ \@plus\p@
      
    }%
    \item[\hskip\labelsep
      \myfont\bfseries
    \abstractname.]%
}{%
  \endlist\egroup
  \ifx\@setabstract\relax \@setabstracta \fi
}
\renewcommand\contentsnamefont{\myfont\bfseries}
\renewcommand\@starttoc[2]{\begingroup
  \setTrue{#1}%
  \par\removelastskip\vskip\z@skip
  \@startsection{}\@M\z@{\linespacing\@plus\linespacing}%
    {.5\linespacing}{
      \contentsnamefont}{#2}%
  \ifx\contentsname#2%
  \else \addcontentsline{toc}{section}{#2}\fi
  \makeatletter
  \@input{\jobname.#1}%
  \if@filesw
    \@xp\newwrite\csname tf@#1\endcsname
    \immediate\@xp\openout\csname tf@#1\endcsname \jobname.#1\relax
  \fi
  \global\@nobreakfalse \endgroup
  \addvspace{32\p@\@plus14\p@}%
  \let\tableofcontents\relax
}
\renewcommand\@settitle{\begin{center}%
  \baselineskip14\p@\relax
    \LARGE
    \bfseries
    \myfont
  \@title
  \end{center}%
}
\renewcommand\@setauthors{%
  \begingroup
  \def\thanks{\protect\thanks@warning}%
  \trivlist
  \centering\footnotesize \@topsep30\p@\relax
  \advance\@topsep by -\baselineskip
  \item\relax
  \author@andify\authors
  \def\\{\protect\linebreak}%
  \large
  \myfont\bfseries\authors
  \ifx\@empty\contribs
  \else
    ,\penalty-3 \space \@setcontribs
    \@closetoccontribs
  \fi
  \endtrivlist
  \normalfont\myfont\@setaddresses
  \endgroup
}
\renewcommand\@setaddresses{\par
  \nobreak \begingroup
\footnotesize
  \def\author##1{\nobreak\addvspace\bigskipamount}%
  \def\\{\unskip, \ignorespaces}%
  \interlinepenalty\@M
  \def\address##1##2{\begingroup
    \par\addvspace\bigskipamount\indent
    \@ifnotempty{##1}{(\ignorespaces##1\unskip) }%
    {
      \ignorespaces##2}\par\endgroup}%
  \def\curraddr##1##2{\begingroup
    \@ifnotempty{##2}{\nobreak\indent\curraddrname
      \@ifnotempty{##1}{, \ignorespaces##1\unskip}\/:\space
      ##2\par}\endgroup}%
  \def\email##1##2{\begingroup
    \@ifnotempty{##2}{\nobreak\indent\emailaddrname
      \@ifnotempty{##1}{, \ignorespaces##1\unskip}\/:\space
      \ttfamily##2\par}\endgroup}%
  \def\urladdr##1##2{\begingroup
    \def~{\char`\~}%
    \@ifnotempty{##2}{\nobreak\indent\urladdrname
      \@ifnotempty{##1}{, \ignorespaces##1\unskip}\/:\space
      \ttfamily##2\par}\endgroup}%
  \addresses
  \endgroup
}
\renewcommand\enddoc@text{\ifx\@empty\@translators \else\@settranslators\fi
}
\renewcommand\@secnumfont{\myfont\bfseries} 

\renewcommand\maketitle{\par
  \@topnum\z@ 
  \@setcopyright
  \thispagestyle{firstpage}
  \ifx\@empty\shortauthors \let\shortauthors\shorttitle
  \else \andify\shortauthors
  \fi
  \@maketitle@hook
  \begingroup
  \@maketitle
  \toks@\@xp{\shortauthors}\@temptokena\@xp{\shorttitle}%
  \toks4{\def\\{ \ignorespaces}}
  \edef\@tempa{%
    \@nx\markboth{\the\toks4
      \@nx\MakeUppercase{\the\toks@}}{\the\@temptokena}}%
  \@tempa
  \endgroup
  \c@footnote\z@
  \@cleartopmattertags
}
\makeatother

\newcommand{\Sec}[1]{Section~\ref{sec:#1}}

\newcommand{\Subsec}[1]{Subsection~\ref{ssec:#1}}
\newcommand{\Subsecs}[2]{Subsections~\ref{ssec:#1} and~\ref{ssec:#2}}
\newcommand{\SubsecS}[2]{Subsections~\ref{ssec:#1}--\ref{ssec:#2}}

\newcommand{\Fig}[1]{Figure~\ref{fig:#1}}

\newcommand{\Thm}[1]{Theorem~\ref{thm:#1}}
\newcommand{\Thms}[2]{Theorems~\ref{thm:#1} and~\ref{thm:#2}}

\newcommand{\Thmenum}[2]{Theorem~\ref{thm:#1}~(\ref{#2})}

\newcommand{\Ex}[1]{Example~\ref{ex:#1}}
\newcommand{\Exs}[2]{Examples~\ref{ex:#1} and~\ref{ex:#2}}

\newcommand{\Lem}[1]{Lemma~\ref{lem:#1}}

\newcommand{\Lemenum}[2]{Lemma~\ref{lem:#1}~(\ref{#2})}

\newcommand{\Cor}[1]{Corollary~\ref{cor:#1}}

\newcommand{\Prp}[1]{Proposition~\ref{prp:#1}}

\newcommand{\Rem}[1]{Remark~\ref{rem:#1}}
\newcommand{\Rems}[2]{Remarks~\ref{rem:#1} and~\ref{rem:#2}}

\newcommand{\Remenum}[2]{Remark~\ref{rem:#1}~(\ref{#2})}

\newcommand{\RemenumS}[3]{Remark~\ref{rem:#1}~(\ref{#2})--(\ref{#3})}

\newcommand{\Def}[1]{Definition~\ref{def:#1}}
\newcommand{\Defs}[2]{Definitions~\ref{def:#1} and~\ref{def:#2}}

\newcommand{\Defenum}[2]{Definition~\ref{def:#1}~(\ref{#2})}

\newcommand{\abs}[2][{}]{\lvert{#2}\rvert_{{#1}}}    
\newcommand{\abssqr}[2][{}]{\lvert{#2}\rvert^2_{#1}} 
\newcommand{\bigabssqr}[2][{}]{\bigl\lvert{#2}\bigr\rvert^2_{#1}}
\newcommand{\Bigabssqr}[2][{}]{\Bigl\lvert{#2}\Bigr\rvert^2_{#1}}

\newcommand{\normsymb}{\|}
\newcommand{\bignormsymb}[1]{#1\|}

\newcommand{\norm}[2][{}]{\normsymb{#2}\normsymb_{{#1}}}    
\newcommand{\normsqr}[2][{}]{\normsymb{#2}\normsymb^2_{#1}} 
\newcommand{\bignorm}[2][{}]{\bignormsymb{\bigl}{#2}\bignormsymb{\bigr}_{#1}}
\newcommand{\bignormsqr}[2][{}]{\bignormsymb{\bigl}{#2}%
                                \bignormsymb{\bigr}^2_{#1}}
\newcommand{\Bignorm}[2][{}]{\bignormsymb{\Bigl}{#2}\bignormsymb{\Bigr}_{#1}}

\newcommand{\iprod}[3][{}]{\langle{#2},{#3}\rangle_{#1}}  
\newcommand{\bigiprod}[3][{}]{\bigl\langle{#2},{#3}\bigr\rangle_{#1}}

\newcommand{\set}[2]{\{ \, #1 \, | \, #2 \, \} }      
\newcommand{\bigset}[2]{\bigl\{ \, #1 \, \bigl|\bigr. \, #2 \, \bigr\} }
\newcommand{\Bigset}[2]{\Bigl\{ \, #1 \, \Bigl|\Bigr. \, #2 \, \Bigr\} }

\newcommand{\map}[3]{ #1 \colon #2 \longrightarrow #3}    

\newcommand{\bd}  {\partial}          
\newcommand{\clo}[2][]{\overline{{#2}}^{#1}} 

\newcommand{\restr}[1]{{\restriction}_{#1}} 

\def\XXint#1#2#3{{\setbox0=\hbox{$#1{#2#3}{\int}$}
     \vcenter{\hbox{$#2#3$}}\kern-.5\wd0}}
\def\XXsum#1#2#3{{\setbox0=\hbox{$#1{#2#3}{\int}$}
     \vcenter{\hbox{$#2#3$}}\kern-.60\wd0}}

\newcommand{\card}[1]{\lvert#1\rvert}   

\newcommand{\dd}    {\, \mathrm d}    

\DeclareMathOperator{\dom}    {dom}
\DeclareMathOperator{\ran}    {ran}
\DeclareMathOperator{\id}     {id}   

\newcommand{\specsymb} {\sigma} 

\newcommand{\spec}[2][{}]   {\specsymb_{\mathrm{#1}}(#2)}

\renewcommand{\phi}{\varphi}   
\renewcommand{\rho}{\varrho}   

\newcommand{\conj}[1]{\overline {#1}}

\newcommand{\R}{\mathbb{R}} 
\newcommand{\C}{\mathbb{C}} 
\newcommand{\N}{\mathbb{N}} 

\newcommand{\1}{\mathbbm 1}                    

\newcommand{\im}{\mathrm i} 

\newcommand{\wt}{\widetilde}           
\newcommand {\qf}[1]{\mathfrak{#1}}    

\newcommand{\HS}{\mathscr H}           

\newcommand{\Sobsymb} {\mathsf H} 
\newcommand{\Sobnsymb} {\ring{\mathsf H}} 

\newcommand{\Lsymb}    {\mathsf L}     
\newcommand{\lsymb}    {\ell}          

\newcommand{\Sobspace}[1][1]{\Sobsymb^{#1}}
\newcommand{\Sobnspace}[1][1]{\Sobnsymb^{#1}}

\newcommand{\Lpspace}[1][p]    {\Lsymb_{#1}}     
\newcommand{\lpspace}[1][p]    {\lsymb_{#1}}     
\newcommand{\Lsqrspace}    {\Lpspace[2]}     
\newcommand{\lsqrspace}    {\lpspace[2]}          

\newcommand{\Lsqr}[2][{}]{\Lsqrspace^{#1}({#2})} 
\newcommand{\lsqr}[2][{}]{\lsqrspace^{#1}({#2})}   

\newcommand{\Linfty}[2][{}]{\Lpspace [\infty] ^{#1}({#2})} 

\newcommand{\Sob}[2][1]{\Sobspace [#1]({#2})}         
\newcommand{\Sobn}[2][1]{\Sobnspace [#1]({#2})}  

\newcommand{\Err}{\mathrm O}

\newcommand{\quadtext}[1]{\quad\text{#1}\quad}
\newcommand{\qquadtext}[1]{\qquad\text{#1}\qquad}

\newcommand{\Lin}[1]{\mathfrak{L}(#1)}

\providecommand{\myfont}{}
\providecommand{\look}[1]{}
\providecommand{\lookO}[1]{}
\providecommand{\lookS}[1]{}

\renewcommand{\itemref}[1]{\noindent \eqref{#1}}
\renewcommand{\implies}{\Rightarrow}
\renewcommand{\iff}{\Leftrightarrow}
\newcommand{\quadiff}{\;\iff\;}
\usepackage{hyperref}
\usepackage{nicefrac}

\newcommand{\HSmin}{\HS_{\min}} %

\newcommand{\D}{A}        
\newcommand{\HSgen}{\HS}  
\newcommand{\Nbar}{\overline \N} 

\newcommand{\gnrc}{\overset{\mathrm{gnrc}}%
    {\underset{\mathrm{QUE}}{\longrightarrow}}}

\newcommand{\gnrcW}{\overset{\mathrm{gnrc}}%
  {\underset{\mathrm{Weid}}{\longrightarrow}}}

\newcommand{\strongto}{\overset{\mathrm{s}}{\longrightarrow}}

\DeclareMathOperator{\Specsymb}{spec}
\renewcommand{\specsymb}{\Specsymb}

\DeclareMathOperator{\clolin}{\overline{lin}}

\begin{document}

\title[] {Generalised norm resolvent convergence: comparison of
  different concepts}

\author{Olaf Post}
\address{Fachbereich 4 -- Mathematik,
  Universit\"at Trier,
  54286 Trier, Germany}
\email{olaf.post@uni-trier.de}

\author{Sebastian Zimmer}%
\address{Fachbereich 4 -- Mathematik,
  Universit\"at Trier,
  54286 Trier, Germany}
\email{zimmerse@uni-trier.de}

\ifthenelse{\isundefined \finalVersion} %
{\date{\today, \thistime,  \emph{File:} \texttt{\jobname.tex}}}
{\date{\today}}  

\begin{abstract}
  In this paper, we show that the two concepts of generalised norm
  resolvent convergence introduced by Weidmann and the first author of
  this paper are equivalent.  We also focus on the convergence speed
  and provide conditions under which the convergence speed is the same
  for both concepts.  We illustrate the abstract results by a large
  number of examples.
\end{abstract}


\maketitle




%
\section{Main results}
\label{sec:intro}
%

\subsection{Introduction}

Convergence of operators is an important topic in many areas of
mathematics and applications.  For unbounded operators such as
Laplacians in Hilbert spaces, one usually studies convergence of their
\emph{resolvents}.  Often, not only the operators vary with the
sequence parameter, but also the spaces in which they are defined,
e.g., when one considers Laplace operators on varying domains.  We
refer to this fact as \emph{generalised} convergence (in contrast to
Kato~\cite[Sec.IV.2]{kato:67} where \emph{generalised} convergence
just means convergence of the resolvents).

Different types of convergence of operators are of interest:
convergence in \emph{operator norm} (uniform convergence),
\emph{strong} (i.e., pointwise) convergence, and weaker versions.  We
focus here on operator \emph{norm} convergence of resolvents of
(possibly) unbounded self-adjoint operators acting in Hilbert spaces;
for results on strong convergence we refer to
\Subsec{ex.lit}. Generalisations to non-self-adjoint operators or
operators acting in Banach spaces are possible, see \Subsec{outlook}
below.

Probably the first abstract approach of convergence of operators
acting on varying Banach spaces --- a priori not embedded in a common
space --- is given in~\cite{stummel:71} (see also the references
therein and~\cite{stummel:76} for a summary).  Stummel defines what he
calls ``discrete convergence of operators'' which is a generalisation
of \emph{strong} (pointwise) convergence of the resolvents to the case
where the underlying spaces vary.  Stummel's setting includes not only
discretisations of partial differential operators (see
e.g.~\cite{vainikko:77}, also for some older abstract results
before~\cite{stummel:71}), but also the case when the domains of the
partial differential operators vary, see \Subsec{ex.lit} for a more
detailed discussion.  For an overview on results on domain
perturbation up to 2008 we refer to the nice survey~\cite{daners:08}
and the references therein; for more recent results
see~\cite{arrieta-lamberti:17,anne-post:21} and references therein.

Weidmann defines \emph{generalised} norm/strong resolvent convergence
implicitly in~\cite{weidmann:84}, and explicitly
in~\cite[Sec.~9.3]{weidmann:00}, having domain perturbations of
Laplacians in mind.  Weidmann basically embeds all spaces into a large
one (called \emph{parent space} here) and considers convergence of the
``lifted'' resolvents in the parent space.  The lifted resolvents are
\emph{pseudo-resolvents}, cf.\ \Remenum{weidmann}{weidmann.b} below.
We introduce Weidmann's concept in \Subsec{gnrc.weid}.  Independently,
the first author of this paper defined generalised norm resolvent
convergence in~\cite{post:12}, based on the concept of
\emph{quasi-unitary equivalence} (\emph{QUE} for short), a
quantitative generalisation of unitary equivalence, first introduced
in~\cite{post:06}.  We present the latter \emph{generalised norm
  resolvent convergence} briefly in \Subsec{gnrc.que}.

The main result of this paper (\Thm{main}) is the equivalence of
generalised norm resolvent convergence as introduced by Weidmann and
the first author.  Let us mention that norm resolvent convergence of
operators $A_n$ towards $A_\infty$ as well as both generalisations
discussed in this paper imply the convergence of bounded operators of
the original operators such as spectral projections and the heat
operator in operator norm (in a suitable sense),
cf.~\cite[Satz~9.28]{weidmann:00} and \cite[Sec.~4.2]{post:12},
\cite[Sec.~1.3]{post-simmer:19} for precise statements.

An important feature of \emph{norm} resolvent convergence for
self-adjoint operators is the \emph{convergence of spectra} in the
sense that
\begin{equation}
  \label{eq:conv.spec}
  \lambda_\infty \in \spec{A_\infty}
  \qquadtext{if and only if}
  \forall n \in \N \; \exists \lambda_n \in \spec{A_n} \colon
  \lambda_n \to \lambda_\infty
\end{equation}
(cf.~\cite[Satz~9.24~(a)]{weidmann:00}).  For the weaker notion of
\emph{strong} (pointwise) resolvent convergence, this is no longer
automatically true, i.e., it may happen that there is an infinite
subset $I \subset \N$ such that for each $n \in I$ there exists an
element $\lambda_n \in \spec {A_n}$ with
$\lambda_n \to \lambda_\infty$ as $n \in I$ and $n \to \infty$, but
$\lambda_\infty \notin \spec {A_\infty}$.  Such points are called
\emph{spectral pollution}.

When considering non-self-adjoint operators, the spectral convergence
may fail (cf.~e.g.~\cite[Ex.~IV.3.8]{kato:67}) even for \emph{norm}
convergence.  The behaviour of the spectrum for general closed (and in
particular \emph{non}-self-adjoint) operators under generalised norm
and strong resolvent convergence in the sense of Weid\-mann is
discussed in~\cite{boegli:17,boegli:18} and references therein.  For
example, no spectral pollution occurs under generalised norm resolvent
convergence (\cite[Thm.~2.4~(i)]{boegli:17}).

We will not treat generalised \emph{strong} convergence here, as it is
a priori not obvious how to implement it in the setting of
quasi-unitary equivalence.  Nevertheless, Weidmann's concept allows a
straightforward definition, and the equivalence of both concepts in
the norm convergence case gives the chance to define generalised
strong resolvent convergence also in the QUE-setting.  We will treat
this in a forthcoming publication.

Many of the above-mentioned concepts are formulated in a variational
way, using sesquilinear forms instead of operators.  To simplify the
presentation, we only use operators here.  Note that the QUE-setting
was originally formulated for sesquilinear forms in~\cite{post:06}.

\subsection{Generalised norm resolvent convergence by Weidmann}
\label{ssec:gnrc.weid}
Weidmann's idea is to compare the operators acting in different
Hilbert spaces in a common so-called \emph{parent Hilbert space} where
the individual Hilbert spaces are subspaces of.  We use the following
generalisation here (cf. also
\RemenumS{weidmann}{weidmann.c}{weidmann.d}):
\begin{definition}[Weidmann's convergence]
  \label{def:gnrc-weid}
  Let $\D_n$ be a self-adjoint bounded or unbounded operator in a
  Hilbert space $\HS_n$ for
  $n \in \Nbar:=\N \cup \{\infty\}=\{1,2,3,\dots,\infty\}$.  We say
  that the sequence $(\D_n)_{n \in \N}$ \emph{converges to}
  $\D_\infty$ \emph{in generalised norm resolvent sense of Weidmann}
  (or shortly \emph{Weidmann-converges}), if the following conditions
  are true:
  \begin{enumerate}
  \item There exist a Hilbert space $\HSgen$, called \emph{parent
      (Hilbert) space} and for each $n \in \Nbar$ an isometry
    $\map {\iota_n}{\HS_n}{\HSgen}$.
  \item We have $\delta_n := \norm[\Lin{\HSgen}]{D_n} \to 0$ as
    $n \to \infty$, where
    \begin{equation}
      \label{eq:weidmanns_diff}
      D_n := \iota_n R_n \iota_n^* - \iota_\infty R_\infty \iota_\infty^*
    \end{equation}
    and where
    \begin{equation}
      \label{eq:def.res}
      R_n := (\D_n - z_0)^{-1}
    \end{equation}
    for $n \in \Nbar$ is the resolvent of $\D_n$ at some common
    resolvent element
    $z_0 \in \Gamma:=\bigcap_{n \in \Nbar} \rho(\D_n)$ (we will not
    stress the dependency of $R_n$ on $z_0$ in the notation).
  \end{enumerate}
  For short, we write $ \D_n \gnrcW \D_\infty $ and call
  $(\delta_n)_n$ the \emph{(convergence) speed}.
\end{definition}

\begin{example}[A motivating example]
  \label{ex:motivating.example1}
  We treat here a situation appearing often in applications and which
  is also the basis of Weidmann's consideration: domain perturbations.
  Assume that $X$ is a measure space with measure $\mu$, and that
  $\HS=\Lsqr X$.  We assume that $X_n \subset X$ are measurable
  subsets of $X$ for $n \in \Nbar$, and that the measure on $X_n$ is
  the restriction of the measure $\mu$ to subsets of $X_n$.  To avoid
  exceptional cases, we assume that $X_n \cap X_\infty$ has positive
  measure for all $n \in \N$.  We set $\HS_n = \Lsqr{X_n}$.  Denote
  the restriction of an equivalence class $f$ of functions from
  $\Lsqr X$ to $X_n$ by $f \restr {X_n}$.  Its adjoint is the
  embedding $\map {\iota_n} {\HS_n}\HS$ given by the extension of
  $f_n \in \HS_n$ by $0$, denoted as $f_n \oplus 0_{X \setminus X_n}$.
  We specify the operators $A_n$ acting on (a dense subspace of)
  $\HS_n$ in a moment, but assume here for simplicity that $A_n \ge 0$
  so we can choose $z_0=-1$ as common resolvent point and set
  $R_n := (A_n+1)^{-1}$.

  A natural candidate for a parent space and isometries are
  $\HS = \Lsqr X$ and
  \begin{equation*}
    \map{\iota_n}{\HS_n=\Lsqr{X_n}}{\HS=\Lsqr X},
    \qquad
    f_n \mapsto f_n \oplus 0_{X \setminus X_n}
  \end{equation*}
  for $n \in \Nbar$.  The operator norm estimate in Weidmann's
  convergence now is equivalent with
  \begin{align}
    \nonumber
    \normsqr[\HS]{(\iota_n R_n \iota_n^*
    -\iota_\infty R_\infty \iota_\infty^*)f}
    &= \int_X
      \bigabssqr{R_n (f \restr {X_n}) \oplus 0_{X \setminus X_n}
      - R_\infty (f \restr {X_\infty}) \oplus 0_{X \setminus X_\infty}}
      \dd \mu\\
    \nonumber
    &= \int_{X_n \cap X_\infty}
      \bigabssqr{R_n (f \restr {X_n})
      - R_\infty (f \restr {X_\infty})} \dd \mu\\
    \nonumber
    & \hspace*{0.1\textwidth}
      + \int_{X_n \setminus X_\infty}
      \bigabssqr{R_n (f \restr {X_n})}\dd \mu
      + \int_{X_\infty \setminus X_n}
      \bigabssqr{R_\infty (f \restr {X_\infty})}\dd \mu\\
    \label{eq:norm.weid.ex}
    &\le \delta_n^2 \normsqr[\Lsqr X] f
  \end{align}
  for all $f \in \HS=\Lsqr X$.  If $A_n$ is the operator
  multiplicating with the function $\map {a_n}{X_n}{[0,\infty)}$, then
  $R_n(f \restr {X_n})= (a_n+1)^{-1} f \restr {X_n}$.

  \myparagraph{Concrete example A: monotonely decreasing sequence.} %
  Let $X=[0,\infty)$ with Lebesgue measure,
  $X_n=[0,1] \cup [2^n,\infty)$ and $a_n(x)=x$ for $n \in \N$ and
  $X_\infty=[0,1]$, $a_\infty(x)=x$.  Here, the action of $A_n$ is the
  same for all $n \in \Nbar$, and $X_\infty \subset X_n$, hence the
  first and third integral in~\eqref{eq:norm.weid.ex} equal $0$ and
  for the second integral we have
  \begin{equation*}
    \int_{X_n \setminus X_\infty} \bigabssqr{R_n (f \restr {X_n})}\dd \mu
    = \int_{(2^n,\infty)} \Bigabssqr{\frac1{x+1} f(x)} \dd x
    \le \frac 1{(2^n)^2} \normsqr[\Lsqr X] f.
  \end{equation*}
  In particular, $A_n \gnrcW A_\infty$ with convergence speed
  $\delta_n=2^{-n}$.

  \myparagraph{Concrete example B: monotonely increasing sequence.} %
  Let $X=(0,1]$ with Lebesgue measure,
  $X_n=[2^{-n},1]$ and $a_n(x)=1/x$ for $n \in \N$ and
  $X_\infty=(0,1]$, $a_\infty(x)=1/x$.  Again, the action of $A_n$ is the
  same for all $n \in \Nbar$, and $X_n \subset X_\infty$, hence the
  first and second integral in~\eqref{eq:norm.weid.ex} equal $0$ and
  for the third integral we have
  \begin{equation*}
    \int_{X_\infty \setminus X_n}
    \bigabssqr{R_\infty (f \restr {X_\infty})}\dd \mu
    = \int_{(0,2^{-n})} \Bigabssqr{\frac1{1/x+1} f(x)} \dd x
    \le \frac 1{(2^n)^2} \normsqr[\Lsqr X] f.
  \end{equation*}
  In particular, $A_n \gnrcW A_\infty$ with convergence speed
  $\delta_n=2^{-n}$.
\end{example}

\begin{remark}[Weidmann's convergence: uniqueness, pseudo-resolvents,
  isometries versus subspaces]
  \label{rem:weidmann}
  \indent
  \begin{enumerate}
  \item The parent space in (our notion of) Weidmann's convergence is
    of course not unique.

  \item
    \label{weidmann.b}
    The lifted resolvents are also called \emph{pseudo-resolvents}
    (see e.g.~\cite[Sec.~VIII.4]{yosida:80}): A family
    $(R(z))_{z \in \Gamma}$ of bounded operators
    $\map{R(z)}\HSgen\HSgen$ with $z \in \Gamma \subset \C$ is called
    a family of \emph{pseudo-resolvents} if the (first) resolvent
    equation
    \begin{equation}
      \label{eq:pseudo.res}
      R(z)-R(w)=(z-w)R(z)R(w)
    \end{equation}
    holds for all $z,w \in \Gamma$.  One can see e.g.\ that
    $\ker R(z)$ is independent of $z \in \Gamma$.  Moreover,
    $R(z)=(A-z)^{-1}$ for some closed operator $A$ if and only if
    $\ker R(z)=\{0\}$. (cf.~\cite[Thm~VIII.4.1]{yosida:80}).  In
    particular, Weidmann's generalised resolvent convergence is a
    rather natural generalisation of the usual resolvent convergence.
  \item
    \label{weidmann.c}
    At this point it should be noted that the situation in the book of
    Weidmann~\cite[Sec.~9.3]{weidmann:00} (see also
    \cite{boegli:17,boegli:18}) is slightly different.  Weidmann
    assumes that $\HS_n$ and $\HS_\infty$ are \emph{subspaces} of the
    common Hilbert space $\HS$.  Moreover, he uses the notation $P_n$
    both for the orthogonal projection onto $\HS_n$ as map
    $\HS \to \HS$ as well as for the co-isometry (the adjoint of an
    isometry, denoted in this article by $\iota_n^*$) as map
    $\HSgen \to \HS_n$.  Moreover, the inclusion
    $\HS_n \subseteq \HSgen$ (here denoted by
    $\map{\iota_n}{\HS_n}{\HSgen}$) is not given a proper name
    in~\cite{weidmann:00,boegli:17}.

  \item
    \label{weidmann.d}
    Our interpretation of Weidmann's generalised norm resolvent
    convergence starts with \emph{isometries}
    $\map{\iota_n}{\HS_n}{\HSgen}$ instead of \emph{subspaces}
    $\HS_n \subset \HS$.  This generalisation is necessary in order to
    compare the two concepts especially in cases when there is no
    natural common parent space (see e.g.\
    \SubsecS{met.graphs}{disc.graphs}).  But this generalisation allows
    certain unwanted cases of ``convergence'': if e.g.\ $\D$ is a
    self-adjoint operator in $\HS$ and if $\map {U_n}{\HS_n}\HS$ is
    unitary for each $n \in \Nbar$, then $(\D_n)_n$ with
    $\D_n := U_n^* \D U_n$ \emph{always} converges to $\D$: choose
    $\iota_n:=U_n$ ($n \in \N$) and $\iota_\infty=\id_\HSgen$ then
    \begin{equation*}
      \iota_n R_n \iota_n^*
      =(\D-z_0)^{-1}
      =\iota_\infty R_\infty \iota_\infty^*.
    \end{equation*}
    One way of avoiding the above-mentioned ``unitary mixing'' is to
    use an additional lattice structure on the Hilbert spaces: Assume
    that the spaces $\HS_n$ ($n \in \Nbar$) and $\HS$ are
    $\Lsqrspace$-spaces.  It is then natural to assume that an
    isometry $\map {\iota_n}{\HS_n}\HS$ is
    \emph{positivity-preserving}, i.e., $f_n \ge 0$ implies
    $\iota_n f_n \ge 0$ pointwise almost everywhere.  This is the case
    in our motivating example \Ex{motivating.example1} and in some of
    our examples in \Sec{examples}.  A weaker condition is that the
    corresponding identification operators
    $J_n=\iota_\infty^* \iota_n$ (see below) are
    positivity-preserving; we see in \Sec{examples} that these
    identification operators are \emph{all} positivity-preserving in
    our examples.
  \end{enumerate}
\end{remark}

\subsection{Generalised norm resolvent convergence based on
  quasi-unitary equivalence}
\label{ssec:gnrc.que}
Independently of Weidmann's concept, the first author of the present
paper developed the notion \emph{quasi-unitary equivalence} of two
self-adjoint, unbounded and non-negative operators $\D_1$ and $\D_2$
acting in $\HS_1$ and $\HS_2$, respectively measuring a sort of
``distance''.  The setting incorporates an identification operator
$\map J {\HS_1}{\HS_2}$ in the norm difference of the resolvents
$R_1=(\D_1+1)^{-1}$ and $R_2=(\D_2+1)^{-1}$, namely
$\norm[\Lin{\HS_1,\HS_2}]{JR_1-R_2J}$.  In order to exclude trivial
cases such as $J=0$ we want that $J$ is ``close'' to a unitary
operator, measured again by a norm estimate:
\begin{definition}[QUE: quasi-unitary equivalence]
  \label{def:que}
  Let $\D_1$ and $\D_2$ be two self-adjoint operators acting in $\HS_1$
  and $\HS_2$, respectively.  For $\delta \geq 0$ we say that $\D_1$
  and $\D_2$ are \emph{$\delta$-quasi-unitary equivalent} if there
  exist a common resolvent element
  $z_0 \in \rho(\D_1) \cap \rho(\D_2)$ and a bounded operator
  $\map J {\HS_1} {\HS_2}$ such that the norm inequalities
  \begin{subequations}
      \label{eq:que}
      \begin{align}
        \label{eq:que.a}
      \norm[\Lin{\HS_1, \HS_2}] J &\leq 1+ \delta,\\
        \label{eq:que.b}
      \bignorm[\Lin {\HS_1}]{(\id_{\HS_1}-J^*J)R_1} &\leq \delta,\qquad
      \bignorm[\Lin {\HS_2}]{(\id_{\HS_2}-JJ^*)R_2} \leq \delta\\
        \label{eq:que.c}
      \norm[\Lin{\HS_1, \HS_2}]{R_2J-JR_1} &\leq \delta
    \end{align}
  \end{subequations}
  hold.  If $z_0 \in \C \setminus \R$, we require that the two
  estimates in~\eqref{eq:que.b} and that~\eqref{eq:que.c} also hold
  with the resolvent $R_j=(\D_j-z_0)^{-1}$ replaced by
  $R_j^*=(\D_j-\conj z_0)^{-1}$.  The operators $J$ and $J^*$ are
  called \emph{identification operators}, and $\delta$ is called
  \emph{distance bound} or \emph{error}.
\end{definition}
In~\cite{post:12} only non-negative operators are considered, hence
one can choose $z_0=-1$.  An extension to non-self-adjoint operators
is possible, see \Rem{non-sa}.

Next, we want to transfer this concept onto a family of self-adjoint
operators to define a convergence.  The idea is to check whether every
member of the family is quasi-unitary equivalent with the limit
operator, and that the sequence of their distance bounds $\delta_n$
converges to $0$:
\begin{definition}[QUE-convergence]
  \label{def:gnrc-que}
  Let $\D_n$ be a self-adjoint bounded or unbounded operator in a
  Hilbert space $\HS_n$ for $n \in \Nbar$.  We say that the sequence
  $(\D_n)_{n \in \N}$ \emph{converges to} $\D_\infty$ \emph{in
    generalised norm resolvent sense} (or shortly
  \emph{QUE-converges}), if there exist
  $z_0 \in \Gamma:=\bigcap_{n \in \Nbar} \rho(\D_n)$ and a sequence
  $(\delta_n)_{n \in \N}$ with $\delta_n \to 0$ as $n \to \infty$ such
  that $\D_n$ and $\D_\infty$ are $\delta_n$-quasi-unitarily
  equivalent with common resolvent element $z_0$.  We write for short
  $\D_n \gnrc \D_\infty$).  We call $(\delta_n)_n$ the
  \emph{(convergence) speed}.
\end{definition}
In particular, $\D_n$ converges to $\D_\infty$ in generalised norm
resolvent sense if
\begin{subequations}
  \label{eq:gnrc-que}
  \begin{align}
    \label{eq:que1}
    \norm[\Lin{\HS_n, \HS_\infty}] {J_n} \le 1+\delta_n &\to 1,\\
    \label{eq:que2}
    \bignorm[\Lin {\HS_n}]{(\id_{\HS_n}-J_n^*J_n)R_n}  \le \delta_n & \to 0,
    \qquad
    \bignorm[\Lin {\HS_\infty}]{(\id_{\HS_\infty}-J_nJ_n^*)R_\infty}
    \le \delta_n \to 0,\\
    \label{eq:que3}
    \norm[\Lin{\HS_n, \HS_\infty}]{R_\infty J_n-J_nR_n} \le \delta_n &\to 0
  \end{align}
\end{subequations}
as $n \to \infty$ for a family of identification operators
$(J_n)_{n \in \N}$ with $\map {J_n}{\HS_n}{\HS_\infty}$, where
$R_n=(\D_n-z_0)^{-1}$ is the resolvent in some common resolvent element
$z_0 \in \Gamma$.  If $z_0 \in \C \setminus \R$, we require
that~\eqref{eq:que1}--\eqref{eq:que3} also hold for $z_0^*$, i.e., for
$R_n$ replaced by $R_n^*$.  Then it is possible to swap the order of
the operators by taking adjoints such as
\begin{equation}
  \tag{\ref{eq:que3}'}
  \label{eq:que32}
  \norm[\Lin{\HS_\infty, \HS_n}]{J_n^*R_\infty -R_nJ_n^*}
  = \norm[\Lin{\HS_n, \HS_\infty}]{R_\infty^* J_n - J_n R_n^*}
  \le \delta_n \to 0
\end{equation}
(see also \Rem{non-sa} for non-self-adjoint operators).

\begin{example}[the motivating example continued]
  \label{ex:motivating.example2}
  We come back to \Ex{motivating.example1}.  A natural candidate for
  the identification operator is
  \begin{equation*}
    \map{J_n}{\HS_n=\Lsqr{X_n}}
    {\HS_\infty=\Lsqr {X_\infty}}, \qquad
    f_n \mapsto (f_n \restr {X_n \cap X_\infty})
    \oplus 0_{X_\infty \setminus X_n}.
  \end{equation*}
  As $J_n$ is a non-trivial partial isometry (cf.\
  \Subsec{from.gen.remarks}), we have $\norm{J_n}=1$,
  hence~\eqref{eq:que1} is trivially fulfilled with $\delta_n=0$.
  Moreover, the two estimates in~\eqref{eq:que2} are equivalent with
  the two estimates
  \begin{subequations}
    \begin{align}
      \label{eq:que.norm1}
      \bignormsqr[\Lsqr{X_n}]{(\id_{\HS_n}-J_n^*J_n)R_nf_n}
      &= \int_{X_n \setminus X_\infty} \abssqr {R_n f_n} \dd \mu
        \le \delta_n^2 \normsqr[\Lsqr{X_n}] {f_n},\\
      \label{eq:que.norm2}
      \bignormsqr[\Lsqr{X_\infty}]%
      {(\id_{\HS_\infty}-J_n J_n^*)R_\infty f_\infty}
      &= \int_{X_\infty \setminus X_n} \abssqr {R_\infty f_\infty} \dd \mu
        \le \delta_n^2 \normsqr[\Lsqr{X_\infty}] {f_\infty}
    \end{align}
    for all $f_n \in \HS_n$ resp.\ $f_\infty \in \HS_\infty$
    Finally,~\eqref{eq:que3} is equivalent here with
    \begin{equation}
      \label{eq:que.norm3}
      \normsqr[\Lsqr{X_\infty}]{(R_\infty J_n - J_n R_n)f_n}
      = \int_{X_n \cap X_\infty}
      \bigabssqr{R_\infty (f_n \restr {X_n \cap X_\infty}) - R_n f_n} \dd \mu
      \le \delta_n^2 \normsqr[\Lsqr {X_n}] {f_n}
    \end{equation}
  \end{subequations}
  for all $f \in \HS_\infty=\Lsqr X$.  We can observe already here
  that the three terms in~\eqref{eq:que.norm1}--\eqref{eq:que.norm3}
  look very similar with the three integrals in Weidmann's norm
  estimate in~\eqref{eq:norm.weid.ex}

  \myparagraph{Concrete example A: monotonely decreasing sequence.} %
  Let $X=[0,\infty)$ with Lebesgue measure,
  $X_n=[0,1] \cup [2^n,\infty)$ and $a_n(x)=x$ for $n \in \N$ and
  $X_\infty=[0,1]$, $a_\infty(x)=x$.  Here, \eqref{eq:que.norm2}
  and~\eqref{eq:que.norm3} are trivially valid with $\delta_n=0$,
  only~\eqref{eq:que.norm1} is non-trivial, and as in Weidmann's
  generalised convergence, we can choose $\delta_n=2^{-n}$.  In
  particular, we have $A_n \gnrc A_\infty$ with convergence speed
  $\delta_n=2^{-n}$.

  \myparagraph{Concrete example B: monotonely increasing sequence.} %
  Let $X=(0,1]$ with Lebesgue measure, $X_n=[2^{-n},1]$ and
  $a_n(x)=1/x$ for $n \in \N$ and $X_\infty=(0,1]$, $a_\infty(x)=1/x$.
  Here, only~\eqref{eq:que.norm2} is non-trivial and we have (as in
  Weidmann's case) $A_n \gnrc A_\infty$ with convergence speed
  $\delta_n=2^{-n}$.
\end{example}

\subsection{Main result and structure of the paper}

The goal of this paper is to relate the two notions of convergence
given in \Defs{gnrc-weid}{gnrc-que}.  As we have seen already in our
motivating example, the two notions are indeed equivalent.  Here is
the main result of our paper:
\begin{theorem}[main theorem]
  \label{thm:main}
  Let $\D_n$ be a self-adjoint operator in a Hilbert space $\HS_n$ for
  each $n \in \Nbar$.  Then $(\D_n)_n$ converges to $\D_\infty$ in
  generalised norm resolvent sense of Weidmann (cf.\ \Def{gnrc-weid})
  if and only if $(\D_n)_n$ converges to $\D_\infty$ in generalised
  norm resolvent (in the QUE-sense of \Def{gnrc-que}).
\end{theorem}

We will give the proof in four steps.
\begin{itemize}
\item The \textbf{first step} in the proof of \Thm{main} is showing
  the rather simple fact that \emph{Weidmann's convergence implies
    QUE-convergence} in \Thm{main1}.  Before showing this, we collect
  some facts about (partial) isometries in \Sec{compare}.

\item In a \textbf{second step}, we show that \emph{QUE-convergence
    implies Weidmann's convergence} (\Thm{main2a}) \emph{assuming that
    a parent space $\HS$ and isometries $\map{\iota_n}{\HS_n}\HS$
    factorising the identification operators exist}, i.e., we have
  $\map {J_n}{\HS_n}{\HS_\infty}$ with $J_n=\iota_\infty^*\iota_n$ for all
  $n \in \N$.  If
  \begin{equation}
    \label{eq:proj.commute}
    P_nP_\infty =P_\infty P_n \qquad\text{for all $n \in \N$},
  \end{equation}
  where $P_n=\iota_n \iota_n^*$ is the orthogonal projection onto the
  range of $\iota_n$, then the convergence speed has the same order in
  both cases (\Thm{main2b}), while for general identification
  operators $J_n$, there is a slight loss in the convergence speed.
  The condition~\eqref{eq:proj.commute} is equivalent with the fact
  that $P_n P_\infty$ (or $P_\infty P_n$) is an orthogonal projection
  (cf.~\cite[Satz.~2.55 (a)]{weidmann:00}.  For \emph{subspaces}
  $\HS_n \subset \HS$ ($n \in \Nbar$), the
  condition~\eqref{eq:proj.commute} is equivalent with
  \begin{equation*}
    \bigl(\HS_n \ominus (\HS_n \cap \HS_\infty)\bigr)
    \perp
    \bigl(\HS_\infty \ominus (\HS_n \cap \HS_\infty)\bigr)
  \end{equation*}
  (cf.~\cite[Aufgabe~2.22, Satz 2.55]{weidmann:00}) where
  $\HS' \ominus \HS''$ is the orthogonal complement of
  $\HS'' \subset \HS'$ in $\HS'$.  In our motivating example,
  $P_n=\1_{X_n,X}$ (multiplication with the indicator function
  $\map{\1_{X_n,X}}X{\{0,1\}}$), and the commuting
  condition~\eqref{eq:proj.commute} is fulfilled.

  Moreover, we characterise whether $J_n$ is a partial isometry in
  terms of Weidmann's data $(\map {\iota_n}{\HS_n}\HS)_{n \in \Nbar}$
  (cf.~\Thm{proj.comm.jn.iso}): namely $J_n$ is a partial isometry if
  and only if~\eqref{eq:proj.commute} holds.  The commuting property
  is therefore an \emph{invariant} among all parent spaces factorising
  the identification operators, cf.\ \Cor{proj.comm.all}.

\item As \textbf{third step}, we show in \Sec{gen.case} that \emph{a
    parent space can always be constructed} from the QUE-data
  (\Thm{dreamworld.exists}), provided that the identification
  operators $J_n$ are \emph{contractions} ($\norm {J_n} \le 1$).  The
  so-called \emph{defect operators} constructed from the
  identification operators play a prominent role in the construction
  of a parent space.  As a consequence, QUE-convergence implies
  Weidmann's convergence without the assumption that a parent space
  exists (\Cor{main2a}).  Moreover, we give further equivalent
  characterisations when the identification operators $J_n$ are
  partial isometries in terms of the defect operators
  (\Thm{j.proj.com}).

\item In a \textbf{last step} of the proof of \Thm{main}, we consider
  the general case (i.e., that $\norm{J_n}>1$ for some $n$) in
  \Subsec{change.id} and \Thm{main2}.
\end{itemize}
\Sec{examples} contains different types of examples, some where a
parent space is naturally given, and some where such a parent space is
not naturally given.  We end this introductory section with a
motivating example, more comments on existing literature and some
further comments.

\begin{remarks}[on the main theorem]
  \label{rem:main}
  \indent
  \begin{enumerate}
  \item The main result \Thm{main} remains true also for
    non-self-adjoint operators $\D_n$ and $\D_\infty$; for the
    necessary changes see \Rem{non-sa}.  For the sake of simplicity,
    the proofs are written for the self-adjoint case only.
  \item The focus in this work lies on the identification operators
    $J_n$ ($n \in \N$); the operator domains of $\D_n$ and $\D_\infty$
    are rather irrelevant for our analysis here; we only use these
    operators through their resolvents $R_n:=(\D_n - z_0)^{-1}$ and
    $R_\infty:=(\D_\infty - z_0)^{-1}$ and their adjoints.
  \item We present a method in \Sec{gen.case} how to construct a
    parent space for Weidmann's generalised norm resolvent convergence
    also in less obvious cases such as thick graphs converging to a
    metric graph (see \Subsecs{met.graphs}{met.graphs2}) or a sequence
    of discrete graphs converging to a pcf fractal (see
    \Subsec{disc.graphs}).
  \item If the projection commuting property~\eqref{eq:proj.commute}
    does not hold, then there is in general a loss in the convergence
    speed when passing from QUE-convergence to Weidmann's generalised
    norm resolvent convergence.  To avoid this loss, it seems to be
    better to use a slightly stronger estimate in the QUE-convergence,
    see \Rems{change.of.que-def}{change.of.que-def.ex}.
  \end{enumerate}
\end{remarks}

\subsection{More comments on existing literature}
\label{ssec:ex.lit}
It is impossible to mention even a small amount of relevant literature
concerning resolvent convergence on varying spaces, as it includes
e.g.\ all types of finite-dimensional approximations of any
infinite-dimensional problem.  Let us at least comment on a concept
weaker than the one considered here, namely the generalised
\emph{strong} resolvent convergence:

\subsubsection*{Generalised strong resolvent convergence and Stummel's
  discrete convergence}

Strong convergence of the resolvents (i.e., the pointwise convergence
of the operator resolvents) in Weidmann's setting means that
\begin{equation*}
  D_n f
  = \iota_n R_n \iota_n^*f - \iota_\infty R_\infty \iota_\infty^* f
\end{equation*}
converges to $0$ in $\HS$.  As already mentioned,
Stummel~\cite{stummel:71} introduced an abstract concept of (strong)
convergence of operators acting in different Banach spaces (see
also~\cite{stummel:76,vainikko:77} and references therein).  A
\emph{discrete approximation} of a Hilbert space $\HS_\infty$ by a
sequence of Hilbert spaces $\HS_n$ in the sense of Stummel is given by
a linear map $\map R {\HS_\infty}{\prod_{n \in \N}\HS_n/{\sim}}$ (not
to be confused with a resolvent) such that
$\norm[\HS_n]{u_n} \to \norm[\HS_\infty]{u_\infty}$ as $n \to \infty$
for all $u_\infty \in \HS_\infty$ and all $[(u_n)_n] \in R(u_\infty)$.
Here, $\sim$ is the equivalence relation given by
$(u_n)_n \sim (v_n)_n$ if $\norm[\HS_n]{u_n-v_n} \to 0$.  Stummel then
defines the \emph{discrete convergence} $u_n \to u_\infty$ if
$[(u_n)_n]=Ru_\infty$ (\cite[Sec.~1.1~(4)]{stummel:71}).
\begin{itemize}
\item Given Weidmann's setting,
  $R u_\infty=[(\iota_n^* \iota_\infty u_\infty)_{n \in \N}]$ defines
  a discrete approximation provided
  \begin{equation}
    \label{eq:proj.conv.strongly}
    P_n \to P_\infty \qquad \text{strongly.}
  \end{equation}
  In particular $u_n \to u_\infty$ (\emph{discrete convergence} in the
  sense of Stummel) means that
  $\norm[\HS_n]{u_n-\iota_n^*\iota_\infty u_\infty} \to 0$.
  If~\eqref{eq:proj.conv.strongly} holds, the latter is also
  equivalent with the more natural condition
  $\norm[\HS]{\iota_n u_n - \iota_\infty u_\infty}\to 0$ (``do
  everything in the parent space'').

\item Given identification operators $(J_n)_{n \in \N}$
  satisfying~\eqref{eq:que1}--\eqref{eq:que2} then one can show that a
  discrete approximation is given by
  $Ru_\infty=[(J_n^* u_\infty)_{n \in \N}]$.  In particular,
  $u_n \to u_\infty$ in the sense of Stummel if and only if
  $\norm[\HS_n]{u_n-J_n^*u_\infty} \to 0$.
\end{itemize}
A sequence $(S_n)_n$ of bounded operators on $\HS_n$ \emph{converges
  discretely} to a bounded operator $S_\infty$ on $\HS_\infty$ (in the
sense of Stummel~\cite[Sec.~1.2~(2)]{stummel:71}) if
$u_n \to u_\infty$ implies $S_n u_n \to S_\infty u_\infty$.
\begin{itemize}
\item The strong convergence of Weidmann~\eqref{eq:weidmanns_diff} is
  equivalent with Stummel's notion of discrete convergence
  $S_n \to S_\infty$ provided~\eqref{eq:proj.conv.strongly} holds.
\item If $R u_\infty:=[(J_n^* u_\infty)_n]$, then $S_n \to S_\infty$
  in the sense of Stummel is equivalent with the fact that
  \begin{equation*}
    \norm[\HS_n]{u_n-J_n^*u_\infty} \to 0
    \quadtext{implies}
    \norm[\HS_n]{S_nu_n-J_n^*S_\infty u_\infty} \to 0.
  \end{equation*}
  If $A_n \gnrc A_\infty$ in the QUE-setting, then
  \begin{equation*}
    \norm[\HS_n]{R_nu_n-J_n^*R_\infty u_\infty}
    \le \norm[\HS_n]{R_n(u_n - J_n^* u_\infty)}
    + \norm[\HS_n]{(R_nJ_n^*-J_n^*R_\infty) u_\infty)}
    \to 0,
  \end{equation*}
  as $(R_nu)_n$ is \emph{consistent} by \Lem{res.est}, i.e.,
  $(\norm[\Lin{\HS_n}]{R_n})_n$ is bounded.  In particular,
  $A_n \gnrc A_\infty$ implies the discrete convergence of the
  resolvents $R_n \to R_\infty$ in the sense of Stummel.
\end{itemize}
B\"ogli~\cite{boegli:17}  considers also non-self-adjoint operators
and is mainly interested in convergence of spectra and in particular,
the non-existence of spectral pollution.  She shows also how to
upgrade generalised strong resolvent convergence to generalised norm
resolvent convergence~\cite[Thm.~2.7 and Prp.~2.13]{boegli:17} under
some compactness assumptions on the resolvents.  Note that B\"ogli
assumes that the strong convergence of the
projections~\eqref{eq:proj.conv.strongly} holds.  If one constructs a
parent space according to \Sec{gen.case} from the QUE-data, and
if~\eqref{eq:que2} holds, then~\eqref{eq:proj.conv.strongly} holds
automatically (see \Prp{proj.conv.strongly}).

For abstract results on (mostly) strong resolvent convergence in the
context of \emph{homogenisation}, we refer to the references presented
in~\cite[Sec.~1]{khrabustovskyi-post:18}.

\subsubsection*{Concepts related with the QUE-convergence}
Another concept of convergence of operators acting in different
Hilbert spaces and related with the QUE-setting is given
in~\cite[Sec.~2.2--2.7]{kuwae-shioya:03}.  Kuwae and Shioya consider
families of Hilbert spaces and identification operators between any
two members of the family.  They define a version of generalised
\emph{strong} resolvent convergence.
In~\cite[Thm.~2.4]{kuwae-shioya:03} they prove that their strong
resolvent convergence is equivalent with Mosco-convergence of
quadratic forms.  They apply their results to convergence of families
of manifolds (where the limit is not necessarily a manifold any more).
Another abstract approach which is applied to so-called \emph{dumbbell
  domains} is given in~\cite[Sec.~4]{aclc:06}; here, the authors use a
scaled measure on the shrinking thin part of the dumbbell domain.

\subsubsection*{Domain perturbations}
Rauch and Taylor~\cite{rauch-taylor:75} embed the Hilbert spaces
$\Lsqr{X_n}$ and $\Lsqr{X_\infty}$ into the common Hilbert space
$\Lsqr{\R^d}$ by extending functions by $0$ for
$X_n,X_\infty \subset \R^d$ as in~\Ex{motivating.example1}.  Rauch and
Taylor show (what Weidmann later called) generalised strong resolvent
convergence for Dirichlet and Neumann Laplacians under some
``convergence'' conditions on $X_n$ and $X_\infty$.  From this, Rauch
and Taylor conclude in~\cite[Thm.~1.2, Thm.~1.5]{rauch-taylor:75}
(although not explicitly stated abstractly) convergence of operator
functions and convergence of the rank of spectral projections provided
the resolvent is compact.  The latter convergence implies convergence
of the spectra in the sense
of~\eqref{eq:conv.spec}, and in particular, convergence of the eigenvalues.

In~\cite[Sec.~3]{weidmann:84} Weidmann develops a preliminary version
of his generalised strong resolvent convergence based on a monotone
convergence theorem for quadratic forms shown by
Simon~\cite{simon:78}.  Weidmann applies his abstract results to
domain perturbations (\cite[Sec.~4]{weidmann:84}).
Stollmann~\cite{stollmann:95} generalises results
from~\cite{rauch-taylor:75,simon:78,weidmann:84} formulated in the
language of Dirichlet forms.

Daners considers pseudo-resolvents (cf.~\eqref{eq:pseudo.res}) given
by $\iota_n R_n(z) \iota_n^*$ as in Weidmann's approach even in Banach
spaces, where $R_n(z)=(A_n - z)^{-1}$ is a usual resolvent.  Moreover,
he shows upper semi-continuity of the spectrum,
cf.~\cite[Sec.~4]{daners:08} and references therein.  Additionally, he
gives equivalent conditions under which generalised strong and norm
resolvent convergence for Dirichlet Laplacians on $X_n \subset \R^d$
``converging'' to $X_\infty \subset \R^d$ holds, cf.~\cite[Thm.~5.2.4
and Thm.~5.2.6]{daners:08}.  For example, the strong resolvent
convergence holds provided $\Sobn{X_n} \to \Sobn {X_\infty}$ in the
sense of~\cite[Sec.~1]{mosco:69}, i.e., for any $u \in \Sobn X$ and
$n \in \N$ there exists $u_n \in \Sobn {X_n}$ such that
$\norm[\Sob {\R^d}] {u-u_n} \to 0$.  For some recent results on domain
perturbations using a concept close to our identification operators in
the QUE-setting we refer to~\cite{arrieta-lamberti:17} and the
references therein.

We plan to address \emph{generalised strong resolvent convergence} in
its different appearances in a subsequent publication.

\subsection{Some further comments and outlook}
\label{ssec:outlook}

\begin{remark}[change of common resolvent element]
  \label{rem:res-point}
  Note that $A_n \gnrcW A_\infty$ holds for \emph{some}
  $z_0 \in \Gamma := \bigcap_{n \in \Nbar} \rho(A_n)$ if and only if
  $A_n \gnrcW A_\infty$ holds for \emph{any} $z_0 \in \Gamma$
  (cf.~\cite[Satz~9.28]{weidmann:00}).

  Similarly, it can be seen that $A_n \gnrc A_\infty$ for \emph{some}
  $z_0 \in \Gamma$ and $\conj z_0 \in \Gamma$ if and only if
  $A_n \gnrc A_\infty$ for \emph{any} $z_0 \in \Gamma$.  In both
  cases, the convergence speed changes only by a factor depending on
  the common resolvent elements.
\end{remark}

\begin{remark}[extension to non-self-adjoint operators]
  \label{rem:non-sa}
  Both convergence concepts extend to non-self-adjoint (closed)
  operators $A_n$ ($n \in \Nbar$).  Weidmann's concept directly
  applies to closed operators $A_n$ (with resolvents
  $R_n:=(A_n-z_0)^{-1}$ for some
  $z_0 \in \bigcap_{n \in \Nbar} \rho(A_n)$) as done
  e.g.~in~\cite{boegli-siegl:14,boegli:17,boegli:18} (see also
  references therein); note that
  \begin{equation}
    \label{eq:dn-adj}
    \norm[\Lin{\HSgen}]
    {\iota_n R_n^* \iota_n^* - \iota_\infty R_\infty^* \iota_\infty^*}
    =\norm[\Lin{\HSgen}]
      {\iota_n R_n \iota_n^* - \iota_\infty R_\infty \iota_\infty^*}.
  \end{equation}
  For the concept of QUE-convergence in the non-self-adjoint case, we
  assume that~\eqref{eq:que2}--\eqref{eq:que3} hold for $R_n$
  \emph{and} $R_n^*=(\D_n^*-\conj z_0)^{-1}$ resp.\ $R_\infty$
  \emph{and} $R_\infty^*=(\D_\infty ^*-\conj z_0)^{-1}$.

  In particular, our main result \Thm{main} remains true with these
  modifications.
\end{remark}

\begin{remark}[a distance in Weidmann's concept]
  As in the concept of QUE-convergence, one can also define a distance
  in Weidmann's concept if there are two operators $A_1$ and $A_2$
  acting in $\HS_1$ and $\HS_2$ with resolvents $R_1=(A_1-z_0)^{-1}$
  and $R_2=(A_2-z_0)^{-1}$, respectively.  Here, one assumes that
  there is a parent Hilbert space $\HS$ and isometries
  $\map {\iota_1}{\HS_1}\HS$ and $\map {\iota_2}{\HS_2}\HS$, and the
  \emph{distance} of $R_1$ and $R_2$ is then
  \begin{equation*}
    \norm[\Lin \HS]{\iota_1 R_1 \iota_1^*-\iota_2 R_2 \iota_2^*}.
  \end{equation*}
  Of course, the parent space is not unique, and one could define the
  distance $R_1$ and $R_2$ as the infimum of all parent spaces $\HS$
  and isometries $\map {\iota_1}{\HS_1}\HS$ and
  $\map {\iota_2}{\HS_2}\HS$; a similar idea also holds for the
  QUE-concept as in \Def{que}.

  We will treat such questions in a forthcoming publication, where we
  also underline some optimality properties of the concrete parent
  space constructed in \Sec{gen.case}.
\end{remark}

\begin{remark}[extension to Banach spaces]
  Many of the concepts extend to operators acting in Banach spaces.
  Weidmann's convergence in this setting could be understood as
  \begin{equation*}
    \norm[\Lin{\mathscr X}]{
      \iota_n R_n \pi_n- \iota_\infty R_\infty \pi_\infty}
    \to 0
  \end{equation*}
  as $n \to \infty$ for bounded (or resolvents of unbounded) operators
  $R_n \in \Lin{\mathscr X_n}$ in Banach spaces $\mathscr X_n$
  ($n \in \Nbar$).  Here, $\map {\iota_n}{\mathscr X_n}{\mathscr X}$
  is an isometry (an injective partial isometry) into another Banach
  space $\mathscr X$ and $\map {\pi_n}{\mathscr X}{\mathscr X_n}$ a
  co-isometry (a surjective partial isometry); partial isometries are
  analysed e.g.\ in~\cite{mbektha:04}.

  The QUE-convergence could be generalised as follows: Assume that
  $\map{J_n}{\mathscr X_n}{\mathscr X_\infty}$ and
  $\map{J_n'}{\mathscr X_\infty}{\mathscr X_n}$ are operators
  fulfilling~\eqref{eq:que2}--\eqref{eq:que3} with
  $\HS_n$ replaced by $\mathscr X_n$ and $J_n^*$ replaced by $J_n'$.

  We will treat such convergences and their relation in a forthcoming
  publication.
\end{remark}

\section{Relation of the two concepts}
\label{sec:compare}
%

\subsection{Partial isometries, operator norms and parent spaces}
\label{ssec:from.gen.remarks}
Let us first prove some material needed for our analysis.  All our
Hilbert spaces here are assumed to be \emph{separable}.

Let $\HS$ and $\HS_0$ be two Hilbert spaces.  An \emph{isometry} is
a linear operator $\map \iota {\HS_0}\HS$ such that
$\norm[\HS]{\iota f_0}=\norm[\HS_0]{f_0}$ for all $f_0 \in \HS_0$.
Equivalently, $\iota$ is an isometry if and only if
$\iota^* \iota=\id_{\HS_0}$.  Moreover, a surjective isometry is
unitary, i.e., a Hilbert space isomorphism.

A \emph{partial isometry} is a linear operator $\map I {\HS_0} \HS$
such that
\begin{equation}
  \label{eq:def.part.iso}
  \map {I \restr {(\ker I)^\perp}}{(\ker I)^\perp}\HS
\end{equation}
is an isometry.  We call $(\ker I)^\perp=I^*(\HS)$ the \emph{initial
  space} and $I(\HS_0)=(\ker I^*)^\perp$ the \emph{final space} of
$I$.  Note that $I(\HS_0)$ is closed as isometric image of the
complete space $(\ker I)^\perp$; a similar argument holds for
$I^*(\HS)$.

\begin{lemma}[facts about partial isometries]
  \label{lem:part.iso}
  Let $\map I {\HS_0}\HS$ be an bounded linear operator. Then the
  following statements are equivalent.
  \begin{enumerate}
  \item
    \label{part.iso.a}
    $I$ is a partial isometry.
  \item
    \label{part.iso.b}
    $I^*$ is a partial isometry.
  \item
    \label{part.iso.c}
    One of the following equations is true:
    \begin{equation*}
      I=II^*I
      \quadiff
      I^*=I^*II^*
      \quadiff
      I^*I=(I^*I)^2
      \quadiff
      II^*=(II^*)^2.
    \end{equation*}
    The third resp.\ fourth assertion say that $I^*I$ resp.\ $II^*$
    are orthogonal projections onto the initial resp.\ final space of
    $I$.

  \item
    \label{part.iso.d}
    The initial space of $I$ is characterised by
    \begin{equation*}
      (\ker I)^\perp
      =I^*(\HS)=
      \set{f_0 \in \HS_0}{\norm[\HS]{If_0}=\norm[\HS_0]{f_0}}.
    \end{equation*}
  \item
    \label{part.iso.e}
    The final space of $I$ is characterised by
    \begin{equation*}
      (\ker I^*)^\perp
      =I(\HS_0)=
      \set{f\in \HS}{\norm[\HS_0]{I^*f}=\norm[\HS]f}.
    \end{equation*}
  \end{enumerate}
\end{lemma}
\begin{proof}
  \itemref{part.iso.a}$\implies$ \itemref{part.iso.b}: Let
  $g \in (\ker I^*)^\perp=I(\HS_0)$, then $g=If_0$ for a
  $f_0 \in \HS_0$. If $f_0 \in \ker I$, then $I^*g=I^*If_0 = 0$ and
  $g \in (\ker I^*)\cap (\ker I^*)^\perp=\{0\}$. If on the other hand
  $f_0 \in (\ker I)^\perp$, then
  $\norm[\HS]{I^*g}=\norm[\HS_0]{I^*I f_0}=\norm[\HS_0]{f_0} =
  \norm[\HS]{If_0}=\norm[\HS]{g}$.
  Therefore $I^*$ is a partial isometry.

  \itemref{part.iso.b}$\implies$\itemref{part.iso.a}: This follows
  from the first implication and the fact, that $(I^*)^*=I$.

  \itemref{part.iso.a}$\implies$\itemref{part.iso.c} (first
  equation: If $f_0 \in \ker I$, it is clear that:
  $If_0=0=II^*If_0$. If $f_0 \in (\ker I)^\perp$, we have
  $II^*If_0=I \id_{\HS_0}f_0$.

  \itemref{part.iso.c} (first equation) $\implies$
  \itemref{part.iso.c} (third equation) is clear by multiplication from
  the left with $I^*$.

  \itemref{part.iso.c} (third equation) $\implies$
  \itemref{part.iso.b}: Let $f_0 \in \HS_0$ then
  \begin{equation*}
    \norm[\HS]{If_0}^2
    =\iprod[\HS_0]{f_0}{I^*If}
    =\iprod[\HS_0]{f_0}{(I^*I)^2f_0}
    =\normsqr[\HS_0]{I^*If_0},
  \end{equation*}
  and therefore $I^*$ is an isometry on $I(\HS_0)=(\ker(I^*))^\perp$.

  \itemref{part.iso.b}$\implies$\itemref{part.iso.c} (second equation)
  $\implies$ (fourth equation) $\implies$\itemref{part.iso.a}
  can be shown analogously.

  \itemref{part.iso.a}$\implies$\itemref{part.iso.d} ''$\subseteq$'':
  Let $f_0 \in (\ker I)^\perp$.  Then
  $\norm[\HS] {If_0} = \norm[\HS_0]{f_0}$ by~\eqref{eq:def.part.iso}.
  ''$\supseteq$'': Let $f_0 \in \HS_0$ with
  $\norm[\HS]{If_0}=\norm[\HS_0]{f_0}$.  We have
  \begin{align*}
    \normsqr[\HS_0]{f_0-I^*If_0}
    &=\normsqr[\HS_0]{(\id_{\HS_0}-I^*I)f_0}
      =\iprod[\HS_0]{(\id_{\HS_0}-I^*I)f_0}{(\id_{\HS_0}-I^*I)f_0}\\
    &=\iprod[\HS_0]{f_0}{(\id_{\HS_0}-I^*I)f_0}
      = \normsqr[\HS_0]{f_0}-\normsqr[\HS]{If_0}=0
  \end{align*}
  Thus $f_0 \in I^*(\HS_0)$.

  \itemref{part.iso.d}$\implies$\itemref{part.iso.a} is given
  directly by \eqref{eq:def.part.iso}. Analogously we get
  \itemref{part.iso.b}$\quadiff$ \itemref{part.iso.e}.  In summary all
  statements characterise the notion of a partial isometry.
\end{proof}

An isometry is hence a partial isometry with maximal initial space, or
equivalently, an injective partial isometry.

A \emph{co-isometry} $\map \pi \HS {\HS_0}$ is the adjoint of an
isometry, i.e., $\pi^*$ is an isometry.  Equivalently, $\pi$ is a
co-isometry if and only if $\pi \pi^*=\id_\HS$.  A co-isometry is a
partial isometry with maximal final space, or equivalently, a
surjective partial isometry.

Let us now provide some simple facts about isometries and
co-isometries.
\begin{lemma}[(co-)isometries and operator norms]
  \label{lem:iso-coiso}
  Let $A \in \Lin{\HS_1,\HS_2}$ and let
  $\map{\iota_j}{\HS_j}{\wt \HS_j}$ be isometries for $j \in \{1,2\}$.
  Then we have
  \begin{equation*}
    \norm[\Lin{\HS_1,\HS_2}] A
    = \norm[\Lin{\HS_1,\wt \HS_2}] {\iota_2 A}
    = \norm[\Lin{\wt \HS_2,\HS_1}] {A^*\iota_2^*}
    = \norm[\Lin{\wt \HS_2,\wt \HS_1}] {\iota_1 A^* \iota_2^*}
  \end{equation*}
  Moreover, if $\HS_1=\HS_2$ and if $A=A^*$, then all the above
  equalities hold with $A^*$ replaced by $A$.
\end{lemma}
\begin{proof}
  The first equality follows from
  \begin{equation*}
    \norm[\Lin{\HS_1,\HS_2}] A
    = \sup_{f_1 \in \HS_1 \setminus \{0\}}
        \frac{\norm[\HS_2]{Af_1}}{\norm[\HS_1]{f_1}}
    = \sup_{f_1 \in \HS_1 \setminus \{0\}}
        \frac{\norm[\wt \HS_2]{\iota_2 Af_1}}{\norm[\HS_1]{f_1}}.
  \end{equation*}
  For the second note that
  \begin{equation*}
    \norm[\Lin{\wt \HS_2,\HS_1}] {A^* \iota_2^*}
    = \norm[\Lin{\HS_1,\wt \HS_2}] {\iota_2 A}
    = \norm[\Lin{\HS_1, \HS_2}] A
  \end{equation*}
  by taking the adjoint and using the first equality
  appropriately. The last equality is a consequence of the first two.
\end{proof}
Note that when the isometry is on the right or the co-isometry is on
the left, we have only the trivial estimate
\begin{equation}
  \label{eq:norm-iso-coiso}
  \norm[\Lin{\wt \HS_1,\wt \HS_2}]{\iota_2^*A\iota_1}
  \le \norm[\Lin{\HS_1,\HS_2}] A
\end{equation}
and the inequality can be strict (e.g.\ if $A \ne 0$, $\HS_1=\{0\}$ and
$\iota_1=0$), see also \Rem{only_dreamworld}.

Let us now formulate an important assumption in Weidmann's convergence
(see \Def{gnrc-weid}):
\begin{definition}[(minimal) parent spaces and their isometries]
  \label{def:dreamworld}
  Let $(\HS_n)_{n \in \Nbar}$ be a sequence of Hilbert spaces.
  \begin{enumerate}
  \item
    \label{dreamworld.a}
    We say that a Hilbert space $\HSgen$ is a \emph{parent space} for
    $(\HS_n)_{n \in \Nbar}$, if there are isometries
    $\map{\iota_n}{\HS_n}{\HS}$ for $n \in \Nbar$; we call
    $(\map{\iota_n}{\HS_n}{\HS})_{n \in \Nbar}$ the corresponding
    \emph{isometries} and $\map{P_n:=\iota_n \iota_n^*}\HSgen\HSgen$
    the corresponding \emph{orthogonal projections} onto the range of
    $\iota_n$.

  \item
    \label{dreamworld.b}
    We say that the corresponding isometries
    $(\map{\iota_n}{\HS_n}{\HS})_{n \in \Nbar}$ \emph{factorise the
      identification operators
      $(\map {J_n}{\HS_n}{\HS_\infty})_{n \in \N}$)} if the
    factorisation $J_n=\iota_\infty^* \iota_n$ holds for $n \in \N$
    (cf.\ \Fig{M1}).
\item
    \label{dreamworld.c}
    Given a parent space $\HS$ with isometries
    $\map {\iota_n}{\HS_n}\HS$ ($n \in \Nbar$) we call
    \begin{equation}
      \label{eq:parent.hs.min}
      \HSmin = \clolin \bigcup_{n \in \Nbar} \iota_n(\HS_n)
    \end{equation}
    the \emph{minimal parent space} associated with
    $(\map{\iota_n}{\HS_n}\HS)_{n \in \Nbar}$.  Here, $\clolin M$
    denotes the closure of the linear span of $M \subset \HS$.
  \end{enumerate}
\end{definition}
The notion \emph{minimal} parent space reflects the fact that with
$\HSmin$ as parent space instead of $\HS$, the maps
$\map {\wt \iota_n}{\HS_n}\HSmin$ are still isometries for each
$n \in \Nbar$.  Moreover, if $f \in \HSmin^\perp=\HS \ominus \HSmin$,
then $f \in (\iota_n(\HS_n))^\perp=\ker \iota_n^*$, i.e.,
$\iota_n^*f=0$ for all $n \in \Nbar$.  In particular, whatever happens
outside $\HSmin$ is not relevant for any objects involving $\iota_n$
and $\iota_n^*$ and $n \in \Nbar$.

\begin{remark}[parent spaces and identification operators]
  \label{rem:dreamworld}
  \indent
  \begin{enumerate}
  \item
    \label{dreamworld.1}
    If $\HSgen$ is a parent space with corresponding isometries
    factorising $(J_n)_n$, then each $J_n$ necessarily is a
    \emph{contraction}, i.e.,
    $\norm[\Lin{\HS_n,\HS_\infty}]{J_n} \leq 1$ for all $n \in \N$, as
    $J_n$ factorises in a co-isometry and an isometry, both with
    operator norm not greater than $1$.
  \item
    \label{dreamworld.2}
    Given a parent space, the corresponding isometries
    $(\map{\iota_n}{\HS_n}{\HS})_{n \in \Nbar}$ clearly \emph{always}
    factorise the identification operators $(J_n)_n$ given by
    \begin{equation*}
      J_n := \iota_\infty^* \iota_n
    \end{equation*}
    for $n \in \N$.  We call $(J_n)_n$ the \emph{identification
      operators associated with the parent space $\HSgen$}.
  \item
    \label{dreamworld.3}
    Note that a naive choice of a parent space would be
    \begin{equation*}
      \HS:=\HS_\infty \oplus \bigoplus_{n \in \N} \HS_n, \qquad
      \iota_\infty f_\infty=(f_\infty,0,\dots), \qquad
      \iota_n f_n = (0,\dots, 0, f_n, 0, \dots).
    \end{equation*}
    But in this case, the identification operators would be $J_n=0$.
    Moreover, Weidmann's resolvent difference
    (cf.~\eqref{eq:def.dn}) always fulfils
    \begin{equation*}
      \normsqr[\HS]{D_n f}
      =\normsqr[\HS_n]{R_nf_n}+\normsqr[\HS_\infty]{R_\infty f_\infty}
    \end{equation*}
    for $f=(f_\infty,f_1,\dots)\in \HS$, and this expression will not
    converge to $0$.  We therefore have to use more elaborated
    isometries, relating the spaces $\HS_n$ and $\HS_\infty$ in an
    appropriate way especially on those subspaces where their
    resolvents are close to each other.  We construct such a parent
    space for \emph{given} identification operators
    $\map {J_n}{\HS_n}{\HS_\infty}$ in \Subsec{concrete.iso}.

    We see in \Rem{conv.invariant} how to characterise Weidmann's
    convergence using parent spaces factorising the identification
    operators.
  \end{enumerate}
\end{remark}

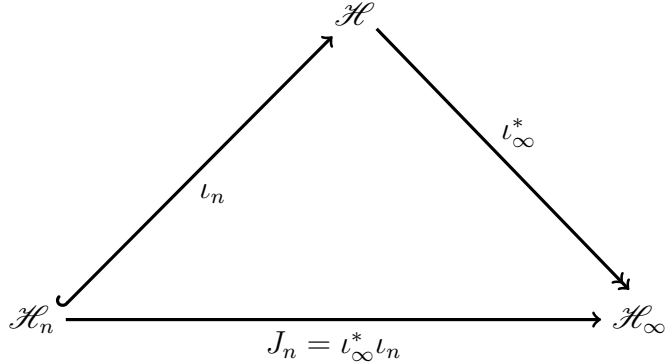
\begin{figure}[h]
  \centering
  \begin{tikzpicture}
    \node (A) at (0,0)  {$\HS_n$};
    \node (B) at (4.25,4.05) {$\HSgen$};
    \node (C) at (8,0)  {$\HS_\infty$};
    \node (A') at (0,0.25) {};
    \node (B') at (3.75,4) {};
    \node (A2) at (0.25,0) {};
    \node (B2) at (4.1,3.9) {};
    \draw[right hook->, very thick] (A2) to node [below] {\quad $\iota_n$} (B2);
    \node (C2) at (7.75,0) {};
    \node (C') at (8, 0.25) {};
    \node (B3) at (4.4, 4) {};
    \draw[->>, very thick] (B3) to  node [above] {\quad $\iota_\infty^*$} (C');
    \draw[->,  very thick] (A) to  node [below] {$J_n=\iota_\infty^*\iota_n$ }  (C);
  \end{tikzpicture}
  \caption{Factorising $J_n$ over $\HS$ via an isometry $\iota_n$ and
    a co-isometry $\iota_\infty^*$.}
  \label{fig:M1}
\end{figure}

Given a parent space $\HSgen$ with corresponding isometries
$(\map{\iota_n}{\HS_n}{\HS})_{n \in \Nbar}$ and associated
identification operators $J_n := \iota_\infty^* \iota_n$ and bounded
operators $R_n$ acting in $\HS_n$ ($n \in \Nbar$), we abbreviate their
difference as used in Weidmann's convergence by
\begin{equation}
  \label{eq:def.dn}
  \map{D_n}{\HSgen}{\HSgen}, \qquad
  D_n := \iota_n R_n \iota_n^* - \iota_\infty R_\infty \iota_\infty^*.
\end{equation}
Moreover we denote by
\begin{equation}
  \label{eq:def.pn}
  \map{P_n:= \iota_n \iota_n^*} \HSgen \HSgen
  \qquadtext{and}
  \map{P_\infty:= \iota_\infty \iota_\infty^*} \HSgen \HSgen
\end{equation}
the orthogonal projections in $\HSgen$ onto $\iota_n(\HS_n)$ resp.\
$\iota_\infty(\HS_\infty)$.  As $\iota_n$ and $\iota_\infty$ are
isometries, we clearly have
\begin{equation}
  \label{eq:isom.iota}
  \iota_n^* \iota_n = \id_{\HS_n}
  \qquadtext{and}
  \iota_\infty^* \iota_\infty = \id_{\HS_\infty}.
\end{equation}
We denote by $P_n^\perp:= \id_\HSgen-P_n$ and
$P_\infty^\perp:= \id_\HSgen-P_\infty$ the corresponding complementary
orthogonal projections.  If a parent space exists we have the
following equalities for $D_n$, $P_n$ and $P_\infty$ as defined above:
\begin{lemma}
  \label{lem:formelsammlung}
  Let $\HSgen$ be a parent space with corresponding isometries
  factorising the identification operators
  $(\map {J_n}{\HS_n}{\HS_\infty})_{n \in \N}$.  Then we have the
  following simple equalities for $P_n$, $P_\infty$ and $D_n$:
  \begin{subequations}
    \label{eq:formelsammlung}
    \begin{align}
      \label{eq:Ps_und_Iotas}
      P_n\iota_n&=\iota_n, \quad
      \iota_n^*P_n=\iota_n^*,\quad
      P_\infty\iota_\infty=\iota_\infty, \quad
      \iota_\infty^*P_\infty=\iota_\infty^*, \\
      \label{eq:Pskom_und_Iotas}
      P^\perp_n \iota_n &= 0, \quad
      \iota_n ^*P^\perp_n  = 0, \quad
      P^\perp_\infty\iota_\infty = 0, \quad
       \iota_\infty ^*P^\perp_\infty  = 0,\\
      \label{eq:dn_in_3teile}
      D_n &= P_\infty D_n P_n + P_\infty^\perp D_n P_n + P_\infty D_n P_n^\perp.
     \end{align}
   \end{subequations}
 \end{lemma}
\begin{proof}
  \eqref{eq:Ps_und_Iotas} and~\eqref{eq:Pskom_und_Iotas} follow
  directly from~\eqref{eq:isom.iota}. \eqref{eq:dn_in_3teile} is a
  direct consequence of $P_n^\perp D_n P_\infty^\perp=0$ as
  $P_\infty^\perp \iota_\infty=0$ and $\iota_n^* P_n^\perp=0$.
\end{proof}

\subsection{Weidmann's convergence implies QUE-convergence}
\label{ssec:from.wei.to.post2}

Let us now assume that we are in the situation of Weidmann's
convergence (\Def{gnrc-weid}).  In particular, $\HSgen$ is a parent
space with corresponding isometries
$(\map{\iota_n}{\HS_n}{\HS})_{n \in \Nbar}$ and associated
identification operators $(J_n)_n$ with
$J_n:= \iota_\infty^* \iota_n$.  In this case, we have the following
result:
\begin{theorem}[Weidmann's convergence implies QUE-convergence]
  \label{thm:main1}
  Let $\D_n$ be a self-adjoint operator in a Hilbert space $\HS_n$ for
  $n \in \Nbar$ such that $\D_n \gnrcW \D_\infty$ with convergence
  speed $(\delta_n)_n$.  Then $\D_n \gnrc \D_\infty$ with the same
  speed $(\delta_n)_n$.
\end{theorem}
\begin{proof}
  Clearly, the operator norm of $J_n$ is not greater than $1$,
  hence~\eqref{eq:que1} is satisfied with $\delta_n=0$.
  Moreover, we want to express $(\id_{\HS_n}-J_n^*J_n)R_n$ appearing
  in~\eqref{eq:que2} in terms of $D_n$.  We have
  \begin{align*}
    \iota_n^*P_\infty^\perp D_n \iota_n
    =\iota_n^*(\id_\HSgen-\iota_\infty^*\iota_\infty)
       \iota_n R_n \iota_n^* \iota_n
    =( \id_{\HS_n}-J_n^*J_n)R_n.
  \end{align*}
  using~\eqref{eq:Ps_und_Iotas} for the first equality
  and~\eqref{eq:isom.iota} for the second.  Similarly, we can express
  the second operator in~\eqref{eq:que2} as
  \begin{equation*}
    -\iota_\infty^*P_n^\perp D_n \iota_\infty
    =\iota_\infty^* P_n^\perp \iota_\infty R_\infty \iota_\infty^* \iota_\infty
    = (\id_{\HS_\infty} - J_n J_n^*)R_\infty.
  \end{equation*}
  Finally, for~\eqref{eq:que3} we have
  \begin{align*}
        \iota_\infty^* D_n \iota_n = J_nR_n-R_\infty J_n.
  \end{align*}
  As all operators appearing in~\eqref{eq:gnrc-que} can be factorised
  into $D_n$ and operators with norm not greater than $1$, we obtain
  the estimates required in~\eqref{eq:que1}--\eqref{eq:que3} with
  respect to the common resolvent element $z_0$.  Similarly, replacing
  $R_n$ by $R_n^*$ and using~\eqref{eq:dn-adj}, we see
  that~\eqref{eq:que1}--\eqref{eq:que3} also hold for $R_n$ replaced
  by $R_n^*$ for all $n \in \Nbar$.
  \end{proof}

\subsection{QUE-convergence implies Weidmann's convergence if a parent
  space exists}
\label{from.pos.to.wei}
To prove Weidmann's conditions using quasi-unitary equivalence, we
face two main difficulties:
\begin{itemize}
\item Can we always construct a parent space $\HSgen$ factorising
  \emph{given} identification operators
  $(\map {J_n}{\HS_n}{\HS_\infty})_{n \in \N}$?

\item If such a decomposition exists, is it possible to use
  QUE-convergence such that the operator norm of $D_n$ from Weidmann's
  concept is small?
\end{itemize}
In this subsection, we \emph{assume} that the first question is
answered affirmative, i.e., we assume that a parent space as
in \Def{dreamworld} exists.  We then answer the second question
affirmatively in this subsection in \Thm{main2a}.


Let us first express the operator norm of $D_n$ (resp.\ its three
summands in~\eqref{eq:dn_in_3teile}) in term of expressions from
QUE-convergence:
\begin{lemma}
  \label{lem:weid-from-que}
  For the norms of the three summands of $D_n$
  in~\eqref{eq:dn_in_3teile} we have
  \begin{subequations}
    \begin{align}
      \label{eq:norm.summand1}
      \norm[\Lin \HS]{P_\infty D_n P_n}
      &= \norm[\Lin{\HS_n,\HS_\infty}]{J_n R_n - R_\infty J_n}, \\
      \label{eq:norm.summand2}
      \norm[\Lin \HS]{P_\infty^\perp D_n P_n}
      &= \bigl(
        \norm[\Lin{\HS_n}]{R_n^*(\id_{\HS_n}-J_n^*J_n)R_n}
        \bigr)^{1/2},\\
      \label{eq:norm.summand3}
      \norm[\Lin \HS]{P_\infty D_n P_n^\perp}
      &= \bigl(
        \norm[\Lin{\HS_\infty}]{R_\infty^*(\id_{\HS_\infty}-J_nJ_n^*)R_\infty}
        \bigr)^{1/2}.
    \end{align}
  \end{subequations}
\end{lemma}
\begin{proof}
  To prove~\eqref{eq:norm.summand1} we calculate
  \begin{equation}
    \label{eq:term1}
    P_\infty D_n P_n
    = \iota_\infty\iota_\infty^*
    (\iota_n R_n \iota_n^*-\iota_\infty R_\infty \iota_\infty^*)
    \iota_n\iota_n^*
    =\iota_\infty(J_n R_n -R_\infty J_n)\iota_n^*.
  \end{equation}
  Using \Lem{iso-coiso} we obtain the claimed norm
  equality~\eqref{eq:norm.summand1}.
  For~\eqref{eq:norm.summand2}, we first calculate
  \begin{equation*}
    P_\infty^\perp D_n P_n
    = (\id_\HSgen - \iota_\infty \iota_\infty^*)\iota_n R_n \iota_n^*
    = (\iota_n - \iota_\infty J_n) R_n \iota_n^*
  \end{equation*}
  using again~\eqref{eq:Pskom_und_Iotas}.  Moreover, we have
  \begin{align}
    \nonumber
    \normsqr[\HSgen]{P_\infty^\perp D_n P_n f}
    &=\bigiprod[\HSgen]{\iota_n R_n^*(\iota_n^* - J_n^* \iota_\infty)
      (\iota_n-\iota_\infty J_n)R_n \iota_n^* f} f\\
    \label{eq:term2}
    &=\bigiprod[\HSgen]{\iota_n R_n^*(\id_{\HS_n}-J_n^*J_n) R_n \iota_n^* f} f
  \end{align}
  for $f \in \HSgen$, where we used
  \begin{equation*}
    (\iota_n-\iota_\infty J_n)^*(\iota_n-\iota_\infty J_n)
    = \iota_n^*\iota_n
    - J_n^*\iota_\infty^*\iota_n
    - \iota_n^*\iota_\infty J_n + J_n^* \iota_\infty^* \iota_\infty J_n
    = \id_{\HS_n} - J_n^*J_n
  \end{equation*}
  for the last step.  Taking the supremum over $f \in \HS$ with
  $\norm[\HS]f=1$ we obtain
  \begin{equation*}
    \normsqr{P_\infty^\perp D_n P_n}
    =\norm{\iota_n R_n^*(\id_{\HS_n}-J_n^*J_n) R_n \iota_n^*}
    =\norm{R_n^*(\id_{\HS_n}-J_n^*J_n) R_n}
  \end{equation*}
  using~\eqref{eq:term2} and the fact that
  $\iota_n^*R_n^*(\id_{\HS_n}-J_n^*J_n)R_n \iota_n$ is self-adjoint
  (first equality) and again \Lem{iso-coiso} (second equality), we
  arrive at the operator norm equality~\eqref{eq:norm.summand2}.

  Similarly, we show~\eqref{eq:norm.summand3}.
 \end{proof}

\begin{remark}[A problem with the (co-)isometry being at the ``wrong''
  side]
  \label{rem:only_dreamworld}
  Note that the second and third
  summand~\eqref{eq:norm.summand2}--\eqref{eq:norm.summand3} of $D_n$,
  we do not have a norm equality in terms of the one from
  QUE-convergence \emph{without} square root.  A direct calculation
  shows that we also have
  \begin{equation*}
    (\id_{\HS_n}-J_n^* J_n)R_n
    =\iota_n^*(\id_\HSgen - \iota_\infty \iota_\infty^*) \iota_nR_n
    =\iota_n^* P_\infty^\perp \iota_nR_n \iota_n^* \iota_n
    =  \iota_n^* P_\infty^\perp D_n P_n\iota_n
  \end{equation*}
  using \Lem{formelsammlung}.  Taking the operator norm yields
  \begin{equation*}
    \norm[\Lin{\HS_n}]{(\id_{\HS_n}-J_n^* J_n)R_n}
    =\norm[\Lin{\HS_n}]{\iota_n^* P_\infty^\perp D_n P_n \iota_n}
    \le \norm[\Lin\HSgen]{P_\infty^\perp D_n P_n}
  \end{equation*}
  (``$\le$''holds, but not the needed ``$\ge$'').  Actually, the
  isometry and the co-isometry are on the ``wrong'' side of the
  operator $P_\infty^\perp D_n P_n$, see \Lem{iso-coiso}.

  Similarly, for the third term~\eqref{eq:norm.summand3} we have
  \begin{equation*}
    (\id_{\HS_\infty}-J_nJ_n^*)R_\infty
    =\iota_\infty^*(\id_\HSgen - \iota_n \iota_n^*) \iota_\infty R_\infty
    =\iota_\infty^* P_n^\perp \iota_ \infty R_\infty \iota_\infty^* \iota_\infty
    =  -\iota_\infty^* P_n^\perp D_n P_\infty \iota_\infty,
  \end{equation*}
  hence we again only have the ``wrong'' estimate
  \begin{equation*}
    \norm[\Lin{\HS_\infty}]{(\id_{\HS_\infty}-J_n J_n^*)R_\infty}
    =\norm[\Lin{\HS_\infty}]{\iota_\infty^* P_n^\perp D_n P_\infty \iota_\infty}
    \le \norm[\Lin\HSgen]{P_n^\perp D_n P_\infty}.
  \end{equation*}
 \end{remark}

Before proving that QUE-convergence implies Weidmann's convergence, we
need another technical result:
\begin{lemma}
  \label{lem:res.est}
  Assume that $\D_n \gnrc \D_\infty$ converges in generalised norm
  resolvent sense in the sense of \Def{gnrc-que} with identification
  operators $J_n$ fulfilling $\norm{J_n} \le 1$ and with convergence
  speed $(\delta_n)_n$.  Then we have
  \begin{equation}
    \label{ineq:OPNormR_n}
    \norm[\Lin{\HS_n}]{R_n}
    \le \norm[\Lin{\HS_\infty}]{R_\infty} +2\delta_n.
  \end{equation}
\end{lemma}
\begin{proof}
  We have
  \begin{equation*}
    \norm{R_n}
    \le \norm{(\id_{\HS_n}-J_n^*J_n)R_n}
    + \norm{J_n^*(J_n R_n-R_\infty J_n)}
    + \norm{J_n^*R_\infty J_n}
    \le 2\delta_n + \norm{R_\infty}.
    \qedhere
  \end{equation*}
\end{proof}
We now prove our next main result:
\begin{theorem}[QUE-convergence implies Weidmann's if a parent space
  exists]
  \label{thm:main2a}
  Let $\D_n$ be a self-adjoint operator in a Hilbert space $\HS_n$ for
  $n \in \Nbar$.  If $\D_n \gnrc \D_\infty$ and if a parent space as
  in \Def{dreamworld} exists then $\D_n \gnrcW \D_\infty$
  Weidmann-converges with convergence speed
    \begin{equation}
    \label{eq:conv.speed}
    \wt \delta_n
    := \delta_n^{1/2} \cdot \Bigl(
         2\bigl(\norm[\Lin {\HS_\infty}]{R_\infty}+2\delta_n\bigr)^{1/2}
         + \delta_n^{1/2}
      \Bigr)
    \in \Err\big(\delta_n^{1/2}\big).
   \end{equation}
\end{theorem}
\begin{proof}
  Assume that $\D_n \gnrc \D_\infty$ with convergence speed
  $\delta_n$.  We estimate the norms of the three summands of $D_n$
  (cf.~\eqref{eq:dn_in_3teile}) given in \Lem{weid-from-que} and obtain
  \begin{equation*}
    \norm{P_nD_nP_\infty}
    =\norm{J_n R_n^* - R_\infty^* J_n}
    \le \delta_n
  \end{equation*}
  for the first summand.  For the norm of the second
  summand~\eqref{eq:norm.summand2}--\eqref{eq:norm.summand3} we have
  \begin{subequations}
    \begin{align}
    \nonumber
    \norm[\Lin\HSgen]{P_\infty^\perp D_n P_n}
    &\le \bigl(
    \norm[\Lin{\HS_n}]{R_n} \norm[\Lin{\HS_n}]{(\id_{\HS_n}-J_n^*J_n)R_n}
      \bigr)^{1/2}\\
    \label{eq:main2a.est1}
    &\le \bigl(
    \norm[\Lin{\HS_n}]{R_n} \cdot \delta_n
    \bigr)^{1/2}
    \le \bigl(
    (\norm[\Lin{\HS_\infty}]{R_\infty} +2\delta_n)\cdot \delta_n
    \bigr)^{1/2}
  \end{align}
  using also \Lem{res.est}.  For the third summand, we have
  \begin{align}
    \label{eq:main2a.est2}
    \norm[\Lin\HSgen]{P_\infty D_n P_n^\perp}
    \leq \bigl(
    \norm[\Lin{\HS_\infty}]{R_\infty}
    \norm[\Lin{\HS_\infty}]{(\id_{\HS_\infty}-J_nJ_n^*)R_\infty}
    \bigr)^{1/2}
    \leq \bigl(
    \norm[\Lin{\HS_\infty}]{R_\infty}\cdot \delta_n
    \bigr)^{1/2}.
  \end{align}
  \end{subequations}
  In particular, we obtain the following norm estimate used in
  Weidmann's generalised norm resolvent convergence
  \begin{align*}
    \norm[\Lin\HSgen]{D_n}
    &\le \norm[\Lin\HSgen]{P_\infty D_n P_n}
      +\norm[\Lin\HSgen]{P_\infty^\perp D_n P_n}
      + \norm[\Lin\HSgen]{P_\infty D_n P_n^\perp }\\
    & \le \delta_n
      +  \bigl(\norm[\Lin {\HS_\infty}]{R_\infty}+2\delta_n)\delta_n \bigr)^{1/2}
      + \bigl(\norm[\Lin {\HS_\infty}]{R_\infty}\delta_n\bigr)^{1/2}\\
    & \le \delta_n^{1/2} \cdot
      \Bigl( \delta_n^{1/2} +
          2\bigl(\norm[\Lin {\HS_\infty}]{R_\infty}+2\delta_n\bigr)^{1/2}
      \Bigr)
      \qedhere
  \end{align*}
\end{proof}

\begin{remark}[characterisation of Weidmann's convergence in a parent
  space]
  \label{rem:conv.invariant}
  We can interpret \Thms{main1}{main2a} in the following way: The
  convergence $\norm{D_n} \to 0$ (i.e., Weidmann's generalised norm
  resolvent convergence) is equivalent with $\D_n \gnrc \D_\infty$
  among all parent spaces factorising the identification operators.
\end{remark}

\begin{remark}[defect operators]
  If the operator $\id_{\HS_n}-J_n^* J_n$ is a projection (i.e., if
  $\id_{\HS_n}-J_n^* J_n=(\id_{\HS_n}-J_n^* J_n)^2$), then we actually
  have
  \begin{align*}
    \normsqr[\Lin \HS]{P_\infty^\perp D_n P_n}
    = \norm[\Lin{\HS_n}]{R_n^*(\id_{\HS_n}-J_n^*J_n)^2R_n}
    = \normsqr[\Lin{\HS_n}]{(\id_{\HS_n}-J_n^*J_n)R_n},
  \end{align*}
  hence there is no square root needed in.  The so-called \emph{defect
    operator } $W_n=(\id_{\HS_n}-J_n^*J_n)^{1/2}$ plays an important
  role in \Subsec{concrete.iso} in the construction of a parent space.
  We also give equivalent characterisations when $W_n$ is an
  orthogonal projection in \Thm{j.proj.com}.  A similar remark holds
  for $(\id_{\HS_\infty}-J_nJ_n^*)$.
\end{remark}
\begin{remark}[a modified version of QUE-convergence]
  \label{rem:change.of.que-def}
  We see that we loose the convergence speed in the two ``bad''
  estimates~\eqref{eq:main2a.est1}--\eqref{eq:main2a.est2}.  Probably
  it is more appropriate to change slightly the definition of
  QUE-convergence and require that
  \begin{equation}
    \label{eq:que2'}
    \tag{\ref{eq:que2}'}
    \bignorm[\Lin {\HS_n}]{R_n^*(\id_{\HS_n}-J_n^*J_n)R_n}  \le \delta_n^2 \to 0,
    \quad
    \bignorm[\Lin {\HS_\infty}]{R_\infty^*(\id_{\HS_\infty}-J_nJ_n^*)R_\infty}
    \le \delta_n^2 \to 0
  \end{equation}
  holds in \Def{gnrc-que}.  Then $\D_n \gnrc \D_\infty$ (with the modified
  definition) with convergence speed $\delta_n$ would lead to
  $\D_n \gnrcW \D_\infty$ with convergence speed $3\delta_n$, as one
  can directly see from~\eqref{eq:dn_in_3teile} and
  \Lem{weid-from-que}.  If $\norm{R_n} \le 1$
  then~\eqref{eq:que2'} implies~\eqref{eq:que2}.

  Note that the first estimate in~\eqref{eq:que2'} is
  equivalent with
  \begin{equation}
    \label{eq:que2''}
    \normsqr[\HS_n]{f_n} - \normsqr[\HS_\infty]{J_n f_n}
    \le \delta_n^2 \normsqr[\HS_n]{(\D_n+1)f_n}
  \end{equation}
  for all $f_n \in \dom \D_n$.  Note that we always have
  $0 \le \normsqr[\HS_n]{f_n} - \normsqr[\HS_\infty]{J_n f_n}$ as
  $J_n$ is a contraction.  The original first estimate
  in~\eqref{eq:que2} is equivalent with
  \begin{equation}
    \label{eq:que2'''}
    \normsqr[\HS_n]{f_n - J_n^* J_n f_n}
    \le \delta_n^2 \normsqr[\HS_n]{(\D_n+1)f_n}
  \end{equation}
  for all $f_n \in \dom \D_n$.  Note that~\eqref{eq:que2''}
  implies~\eqref{eq:que2'''} (if $\norm{R_n} \le 1$, and with
  a slightly different $\delta_n$ in the general case, see
  \Lem{res.est}).  A similar remark holds for the second estimate
  in~\eqref{eq:que2'}.

  We have already seen in~\cite[Sec.~3.2]{khrabustovskyi-post:21}
  that the stronger estimate~\eqref{eq:que2''} is also more
  appropriate than the weaker one~\eqref{eq:que2'''} when
  showing spectral convergence
  (see~\cite[Thm.~3.5]{khrabustovskyi-post:21}).
\end{remark}

\subsection{QUE-convergence implies Weidmann's convergence: better
  convergence speed}
\label{from.pos.to.wei.speed}

As we have seen, a loss in the speed of convergence occurs, when
passing from QUE- to Weidmann's convergence in \Thm{main2a}.  Under a
commutator condition on the projections, the convergence speed in
Weidmann's convergence is the same as in QUE-convergence:

\begin{theorem}[QUE-convergence implies Weidmann's: better convergence speed]
  \label{thm:main2b}
  Let $\D_n$ be a self-adjoint operator in a Hilbert space $\HS_n$ for
  $n \in \Nbar$ such that $\D_n \gnrc \D_\infty$ QUE-converges with
  convergence speed $(\delta_n)_n$.  If furthermore a parent space as
  in \Def{dreamworld} exists and if the projections
  $P_n=\iota_n\iota_n^*$ and $P_\infty=\iota_\infty \iota_\infty^*$
  commute (i.e.,~\eqref{eq:proj.commute} holds) then $(\D_n)_n$
  converges in generalised norm resolvent in the sense of Weidmann
  with convergence speed $(3\delta_n)_n$.
\end{theorem}
\begin{proof}
  We again use the decomposition of $D_n$ in three terms as
  in~\eqref{eq:dn_in_3teile}.  The first term causes no problems
  (see~\eqref{eq:norm.summand1} and~\eqref{eq:term1} and can be
  estimated by $\delta_n$.

  For the second and third summand in~\eqref{eq:dn_in_3teile} we need
  the assumption $P_nP_\infty =P_\infty P_n$.  Here, we have
  \begin{align*}
    P_\infty^\perp D_n P_n
    & = (\id_\HSgen - P_\infty) \iota_n R_n \iota_n^*
      = (\id_\HSgen - P_\infty P_n)\iota_n R_n \iota_n^*\\
    &= (\id_\HSgen - P_n P_\infty)\iota_n R_n \iota_n^*
    = \iota_n(\id_{\HS_n}-\iota_n^* \iota_\infty \iota_\infty^* \iota_n) R_n \iota_n^*
    = \iota_n(\id_{\HS_n}-J_n^*J_n) R_n \iota_n^*
  \end{align*}
  using $\iota_n=P_n\iota_n$ for the second equality and where we used
  that $P_n$ and $P_\infty$ commute for the third equality.  In
  particular, we have again an isometry on the left and a co-isometry
  on the right, hence the operator norm equality
  \begin{equation*}
    \norm{P_\infty^\perp D_n P_n}
    = \norm{(\id_{\HS_n}-J_n^* J_n)R_n}
  \end{equation*}
 (see \Lem{iso-coiso}).  Similarly, we have
 \begin{equation*}
   P_\infty D_n P_n^\perp
   = -\iota_\infty R_\infty P_n^\perp
   =-\iota_\infty R_\infty(\id_{\HS_\infty}-J_n J_n^*) \iota_\infty^*
 \end{equation*}
 and
 \begin{equation*}
    \norm{P_\infty D_n P_n^\perp}
    = \norm{R_\infty(\id_{\HS_\infty}-J_n J_n^*)}
    = \norm{(\id_{\HS_\infty}-J_n J_n^*)R_\infty^*}.
  \end{equation*}
  In particular we can estimate $\norm {D_n}$ by the three terms above
  and hence have $\norm{D_n} \le 3\delta_n$.
\end{proof}

We are now giving equivalent characterisations for the commuting
condition $P_nP_\infty=P_\infty P_n$:
\begin{theorem}[equivalent characterisation of partial isometries]
  \label{thm:proj.comm.jn.iso}
  Assume that $(\map{J_n}{\HS_n}{\HS_\infty})_{n \in \N}$ is a
  sequence of contractions, then the following assertions are
  equivalent:
  \begin{enumerate}
  \item
    \label{j.proj.comm.a}
    $J_n$ is a partial isometry.

  \item
    \label{j.proj.comm.b}
    For one parent space (for all parent spaces) with corresponding
    isometries $(\map{\iota_n}{\HS_n}\HSgen)_{n \in \Nbar}$
    factorising $(J_n)_n$ we have
    \begin{equation*}
      \iota_n(\HS_n)\cap \iota_\infty(\HS_\infty)
      = \iota_n(\HS_n'),
    \end{equation*}
    where $\HS_n'=(\ker J_n)^\perp=J_n^*(\HS_\infty)$.
  \item
    \label{j.proj.comm.c}
    For one parent space (for all parent spaces) with corresponding
    isometries $(\map{\iota_n}{\HS_n}\HSgen)_{n \in \Nbar}$
    factorising $(J_n)_n$ we have
    \begin{equation*}
      \iota_n(\HS_n)\cap \iota_\infty(\HS_\infty)
      = \iota_\infty(\HS_\infty'),
    \end{equation*}
    where $\HS_\infty'=(\ker J_n^*)^\perp=J_n(\HS_n)$.
  \item
    \label{j.proj.comm.d}
    For one parent space (for all parent spaces) with corresponding
    isometries factorising $(J_n)_n$ and with orthogonal projections
    $P_n$ and $P_\infty$ we have $P_n P_\infty=P_\infty P_n$.
  \end{enumerate}
\end{theorem}
\begin{proof}
  As \itemref{j.proj.comm.a} is formulated without reference to the
  parent space, the other assertions are true for one parent space
  resp.\ for all parent spaces for $(J_n,\HS_n,\HS_\infty)$.

  \itemref{j.proj.comm.a}$\implies$\itemref{j.proj.comm.b}
  ``$\subseteq$'': Let $f=\iota_n f_n=\iota_\infty f_\infty$.  We have
  to show that we can actually choose $f_n' \in \HS_n'$ such that
  $\iota_n f_n'=\iota_n f_n=f$.  Let
  $f_n'=J_n^*f_\infty \in J_n^*(\HS_\infty)=(\ker J_n)^\perp =
  \HS_n'$, then
  \begin{equation*}
    \iota_n f_n'
    = \iota_n J_n^* f_\infty
    = \iota_n \iota_n^* \iota_\infty f_\infty
    = \iota_n \iota_n^* \iota_n f_n
    = \iota_n f_n
    = f,
  \end{equation*}
  i.e., $f \in \iota_n(\HS_n')$.  ``$\supseteq$'': Let
  $f=\iota_n f_n' \in \iota_n(\HS_n') \subset \iota_n(\HS_n)$.  We
  have to show that $f \in \iota_\infty(\HS_\infty)$.  As the latter
  space is the final space of the isometry $\iota_\infty$,
  we use the characterisation \Lemenum{part.iso}{part.iso.e}: We have
  \begin{equation*}
    \norm[\HS_\infty] {\iota_\infty^*f}
    = \norm[\HS_\infty]{\iota_\infty^* \iota_n f_n'}
    = \norm[\HS_\infty]{J_nf_n'}
    = \norm[\HS_n]{f_n'}
    =\norm[\HSgen] {\iota_nf_n'}
    =\norm[\HSgen] f,
  \end{equation*}
  where we used for the third equation that $f_n'$ is in the initial
  space $\HS_n'=(\ker J_n)^\perp$ of the partial isometry $J_n$.  In
  particular, we have shown that $f \in \iota_\infty(\HS_\infty)$.

  \itemref{j.proj.comm.a}$\implies$\itemref{j.proj.comm.c} The proof is
  literally the same as the one
  for~\itemref{j.proj.comm.a}$\implies$\itemref{j.proj.comm.b}, just
  interchange $(\cdot)_\infty$ and $J_n^*$ with $(\cdot)_n$ and $J_n$,
  respectively.

   \itemref{j.proj.comm.b}$\implies$\itemref{j.proj.comm.d}: We first
  observe that $\iota_n(\HS_n'') \perp \iota_\infty(\HS_\infty)$,
  where $\HS_n''=(\HS_n')^\perp=\ker J_n$: Let $f=\iota_n f_n''$ with
  $J_n f_n''=0$ and
  $g = \iota_\infty g_\infty \in \iota_\infty(\HS_\infty)$.  Then
  \begin{equation*}
    \iprod[\HSgen] f g
    = \iprod[\HSgen]{\iota_n f_n''}{\iota_\infty g_\infty}
    = \iprod[\HS_\infty]{\iota_\infty^*\iota_n f_n''}{g_\infty}
    = \iprod[\HS_\infty]{J_nf_n''}{g_\infty}
    =0.
  \end{equation*}
  Now we have
  \begin{align*}
    \iota_n(\HS_n)
    = \iota_n(\HS_n') \oplus \iota_n(\HS_n'')
    &= \bigl(\iota_n(\HS_n) \cap \iota_\infty(\HS_\infty)\bigr)
    \oplus \iota_n(\HS_n'')\\
    \text{and}\quad
    \iota_\infty(\HS_\infty)
    &= \bigl(\iota_n(\HS_n) \cap \iota_\infty(\HS_\infty)\bigr)
    \oplus \hat \HS^\perp,
  \end{align*}
  where $\hat \HS^\perp$ is defined by the last line.  As
  $\hat \HS^\perp \subset \iota_\infty(\HS_\infty)$ we have
  $\hat \HS^\perp \perp \iota_n(\HS_n'')$.
  From~\cite[Aufgabe~2.22]{weidmann:00} we conclude that
  $P_nP_\infty=P_\infty P_n$.

  \itemref{j.proj.comm.c}$\implies$\itemref{j.proj.comm.d}: Again, the
  proof is literally the same as the one
  for~\itemref{j.proj.comm.b}$\implies$\itemref{j.proj.comm.d}, just
  interchange $(\cdot)_\infty$ and $J_n^*$ with $(\cdot)_n$ and $J_n$,
  respectively.

  \itemref{j.proj.comm.d}$\implies$\itemref{j.proj.comm.a}: We have
  \begin{equation*}
    J_n J_n^* J_n
    = \iota_\infty^* \iota_n \iota_n^* \iota_\infty \iota_\infty^* \iota_n
    = \iota_\infty^* P_n P_\infty \iota_n
    = \iota_\infty^* P_\infty P_n \iota_n
    = \iota_\infty^* \iota_n
    =J_n.
  \end{equation*}
  By \Lem{part.iso}, $J_n$ is a partial isometry.
\end{proof}

One consequence of \Thm{proj.comm.jn.iso} is that the commuting
property of the projections is an invariant, i.e., only depending on
the quasi-unitary setting given by
$(\map{J_n}{\HS_n}{\HS_\infty})_{n \in \N}$:
\begin{corollary}[projection commuting is invariant]
  \label{cor:proj.comm.all}
  If there is a parent space with isometries factorising $(J_n)_n$
  such that $P_n P_\infty=P_\infty P_n$ then this is true for
  \emph{all} parent spaces.
\end{corollary}

We immediately conclude from \Thms{main2b}{proj.comm.jn.iso}:
\begin{corollary}[$J_n$ is a partial isometry]
  \label{cor:proj.comm.jn.p-iso}
  Assume that $\D_n \gnrc \D_\infty$ QUE-converges with identification
  operators $(J_n)_n$ and convergence speed $(\delta_n)_n$ and that a
  parent space with isometries factorising $(J_n)_n$ exists.

  If $J_n$ are partial isometries for all $n \in \N$ then
  $\D_n \gnrcW \D_\infty$ with speed $(3\delta_n)_n$.
\end{corollary}

Two special cases often appearing in applications deserve to be
mentioned.  Although the claims follow from the fact that a
(co-)isometry is a partial isometry, we give explicit proofs here (as
they are rather simple):
\begin{corollary}[$J_n$ is an isometry]
  \label{cor:proj.comm.jn.iso}
  Assume that $\D_n \gnrc \D_\infty$ QUE-converges with speed
  $(\delta_n)_n$.  If the identification operators $J_n$ are
  isometries for all $n \in \N$ then $\HSgen=\HS_\infty$ is a parent
  space and $\D_n \gnrcW \D_\infty$
  with speed $(2\delta_n)_n$.
\end{corollary}
\begin{remark*}
  Note that if $J_n$ is an isometry, then $\D_n \gnrcW \D_\infty$ just
  means that
  \begin{equation*}
    \norm[\Lin \HSgen]{J_n R_n J_n^* - R_\infty} \to 0.
  \end{equation*}
\end{remark*}
\begin{proof}
  If $J_n$ is an isometry, then clearly $\HSgen=\HS_\infty$ is a
  parent space with isometries
  $(\map{\iota_n}{\HS_n}{\HSgen})_{n \in \N}$ and $\iota_n=J_n$.
  Moreover, $\iota_\infty=\id_{\HS_\infty}$ is unitary and we can add
  $\iota_\infty$ on the right and $\iota_\infty^*$ on the left hand
  side.  In particular, we have
  \begin{align*}
    \norm[\Lin{\HSgen}]{D_n}
    &= \norm[\Lin{\HS_\infty}]
      {\iota_\infty^*(\iota_n R_n \iota_n^*
      - \iota_\infty R_\infty \iota_\infty^*)\iota_\infty}\\
    &= \norm[\Lin{\HS_\infty}]{J_nR_nJ_n^* -R_\infty}\\
    &\leq \norm[\Lin{\HS_\infty}]{J_n(R_nJ_n^* - J_n^*R_\infty)}
      +  \norm[\Lin{\HS_\infty}]{(J_nJ_n^*-\id_{\HS_\infty})R_\infty}
      \leq 2\delta_n.
      \qedhere
  \end{align*}
\end{proof}

\begin{corollary}[$J_n$ is a co-isometry]
  \label{cor:proj.comm.jn.co-iso}
  Assume that $\D_n \gnrc \D_\infty$ QUE-converges with speed
  $(\delta_n)_n$ and that a parent space exists.  If $J_n$ are
  co-isometries for all $n \in \N$ then $\D_n \gnrcW \D_\infty$ with
  speed $(2\delta_n)_n$.  Moreover,
  $\iota_\infty(\HS_\infty) \subset \iota_n(\HS_n)$ for all
  $n \in \N$.
\end{corollary}
\begin{proof}
  If $J_n^*$ is an isometry then $\HS_\infty'=\HS_\infty$ and
  $\iota_\infty(\HS_\infty) \subset \iota_n(\HS_n)$ by
  \Thmenum{proj.comm.jn.iso}{j.proj.comm.c}, and therefore $\iota_n$
  can be added on the left and $\iota_n^*$ on the right hand side.  In
  particular, we have
  \begin{align*}
    \norm[\Lin{\HSgen}]{D_n}
    &= \norm[\Lin{\HS_n}]
      {\iota_n^*(\iota_n R_n \iota_n^*
      - \iota_\infty R_\infty \iota_\infty^*)\iota_n}\\
    &= \norm[\Lin{\HS_\infty}]{R_n -J_n^*R_\infty J_n}\\
    &\leq \norm[\Lin{\HS_n}]{(\id-J_n^*J_n)R_n}
      +  \norm[\Lin{\HS_n}]{J_n^*(J_nR_n - R_\infty J_n)}
    \leq 2\delta_n.
      \qedhere
  \end{align*}
\end{proof}

\section{From QUE-convergence to Weidmann's convergence: the general case}
\label{sec:gen.case}
%
\subsection{Defect operators and existence of a parent space}
\label{ssec:concrete.iso}
The main result in this subsection is to prove the existence of a
parent space $\HSgen$ with isometries
$(\map{\iota_n}{\HS_n}\HSgen)_{n \in \Nbar}$ factorising a given
sequence of identification operators
$(\map{J_n}{\HS_n}{\HS_\infty})_{n \in \N}$, see \Def{dreamworld}.  As
$J_n=\iota_\infty^* \iota_n$, we necessarily have
\begin{equation}
  \label{eq:jn.le1}
  \norm[\Lin{\HS_n,\HS_\infty}]{J_n} \leq 1
  \qquad\text{for all $n \in\N$,}
\end{equation}
hence we assume~\eqref{eq:jn.le1} throughout this subsection.

Our main result in this subsection is as follows:
\begin{theorem}[a parent space exists]
  \label{thm:dreamworld.exists}
  If $\map {J_n} {\HS_n} {\HS_\infty}$ are contractions
  ($\norm{J_n} \le 1$) for all $n \in \N$, then there is a parent
  space $\HSgen$ and isometries
  $(\map{\iota_n}{\HS_n}\HSgen)_{n \in \Nbar}$ factorising the
  identification operators $(J_n)_{n \in \N}$, i.e.,
  $J_n=\iota_\infty^*\iota_n$.
\end{theorem}
Together with \Thm{main2a} we immediately conclude:
\begin{corollary}[QUE-convergence implies Weidmann's one, the case of
  contractions]
  \label{cor:main2a}
  Let $\D_n$ be a self-adjoint operator in a Hilbert space $\HS_n$ for
  $n \in \Nbar$.  If $\D_n \gnrc \D_\infty$ with identification
  operators $J_n$ fulfilling $\norm{J_n} \le 1$ and with convergence
  speed $(\delta_n)_n$. Then $\D_n \gnrcW \D_\infty$
  Weidmann-converges with convergence speed
  $\wt \delta_n \in \Err(\delta_n^{1/2})$ given
  in~\eqref{eq:conv.speed}.
\end{corollary}

The proof of \Thm{dreamworld.exists} follows from the statement in
\Lem{concrete.iso} below.  Before, we need some preparations and start
with an important ingredient in the isometry $\iota_n$ already used
in~\cite[Sec.~I.3]{szfbk:10}:
\begin{lemma}
  \label{lem:defect.ops}
  Assume~\eqref{eq:jn.le1}.
  \begin{enumerate}
  \item
    \label{defect.ops.a}
    The so-called \emph{defect operators}
    \begin{equation}
      \label{eq:def.wn}
      W_n:= (\id_{\HS_n}- J_n^* J_n)^{1/2} \qquadtext{and}
      W_{\infty,n}:= (\id_{\HS_\infty}-J_nJ^*_n)^{1/2}
    \end{equation}
    are well-defined and self-adjoint with spectrum in $[0,1]$.
    Moreover, the spectra of $W_n$ and $W_{\infty,n}$ agree including
    multiplicity, except for the value $1$, i.e.,
    \begin{equation*}
      \spec{W_n} \setminus \{1\} = \spec{W_{\infty,n}} \setminus \{1\}.
    \end{equation*}
  \item
    \label{defect.ops.b}
    We have
    \begin{equation}
      \label{eq:intertwining}
      J_nW_n=W_{\infty,n} J_n
      \qquadtext{and}
      W_nJ_n^*=J_n^*W_{\infty,n}.
    \end{equation}

  \item
    \label{defect.ops.c}
    We have
    \begin{subequations}
      \begin{align}
        \label{eq:hilfsglnorm}
        \normsqr[\HS_\infty]{J_n f_n} + \normsqr[\HS_n]{W_n f_n}
        &= \normsqr[\HS_n]{f_n}\qquad\text{and}\\
        \label{eq:hilfsglnorm2}
        \normsqr[\HS_n]{J_n^* f_\infty} + \normsqr[\HS_\infty]{W_{\infty,n} f_\infty}
        &= \normsqr[\HS_\infty]{f_\infty}.
      \end{align}
    \end{subequations}
  \end{enumerate}
\end{lemma}
\begin{proof}
  \itemref{defect.ops.a}~As $\norm{J_n} \le 1$ we have
  $0 \le \id_{\HS_n} - J_n^*J_n$, hence the square root $W_n$ is
  well-defined.  Moreover, we have
  $\id_{\HS_n} - J_n^*J_n\le \id_{\HS_n}$, hence the spectrum of
  $\id_{\HS_n} - J_n^*J_n$ and therefore of $W_n$ lies in $[0,1]$.
  The claim for $W_{\infty,n}$ follows similarly.  The assertion on
  the spectra of $W_n$ and $W_{\infty,n}$ follows from the fact that
  $\id_{\HS_n}-W_n^2=J_n^*J_n$ and
  $\id_{\HS_\infty}-W_{\infty,n}^2=J_nJ_n^*$ and the fact that the
  spectra of $J_n^*J_n$ and $J_nJ_n^*$ agree except for the value $0$
  (see e.g.~\cite[Sec.~1.2]{post:09c})

  \itemref{defect.ops.b}~For~\eqref{eq:intertwining}, we have
  \begin{equation*}
    J_n(W_n)^2
    =J_n(\id_{\HS_n}-J_n^*J_n)
    = J_n-J_nJ_n^*J_n
    =(\id_{\HS_\infty}-J_nJ_n^*)J_n
    = (W_{\infty,n})^2 J_n,
  \end{equation*}
  and this can be extended to $J_n p((W_n)^2)=p((W_{\infty,n})^2)J_n$
  for any polynomial $p$ defined on $\R$.  As we can approximate
  $t\mapsto\sqrt{t}$ uniformly on $\spec{W_n}$ resp.\
  $\spec{W_{\infty,n}} \subset [0,1]$ by a polynomial, and as uniform
  convergence turns into operator convergence for the operator
  functions, we conclude $J_nW_n=W_{\infty,n} J_n$.  The second
  equality in~\eqref{eq:intertwining} follows by taking adjoints.

  \itemref{defect.ops.c}~We have
  \begin{equation*}
    \normsqr[\HS_\infty]{J_nf_n} + \normsqr[\HS_n]{W_n f_n}
    =\iprod[\HS_n]{f_n}{J_n^*J_nf_n}
    + \iprod[\HS_n]{f_n}{(\id_{\HS_n}-J_n^*J_n)f_n}
    =\normsqr[\HS_n]{f_n}
  \end{equation*}
  and similarly for~\eqref{eq:hilfsglnorm2}.
\end{proof}

Using the definition of the defect operators, we immediately conclude
from \Lem{weid-from-que}:
\begin{corollary}
  \label{cor:defect.ops}
  If $\HS$ is a parent space with corresponding isometries factorising
  $J_n$ and orthogonal projections $P_n$, we have
  \begin{subequations}
    \begin{align}
      \nonumber
      \normsqr[\Lin \HS]{P_\infty^\perp D_n P_n}
      = \normsqr[\Lin{\HS_n}]{W_n R_n}
      &= \norm[\Lin{\HS_n}]{R_n^*(\id_{\HS_n}-J_n^*J_n)R_n}\\
      \label{eq:defect.op.est1}
      &\le \norm[\Lin{\HS_n}] {R_n}
        \norm[\Lin{\HS_n}]{(\id_{\HS_n}-J_n^*J_n)R_n}
        \quad\text{and}\\
      \nonumber
      \normsqr[\Lin \HS]{P_\infty D_n P_n^\perp}
      = \normsqr[\Lin{\HS_\infty}]{W_{\infty,n} R_\infty}
      &=\norm[\Lin{\HS_\infty}]{R_\infty^*(\id_{\HS_\infty}-J_nJ_n^*)R_\infty}\\
      \label{eq:defect.op.est2}
      &\le \norm[\Lin{\HS_\infty}] {R_\infty}
        \norm[\Lin{\HS_\infty}]{(\id_{\HS_\infty}-J_nJ_n^*)R_\infty}.
    \end{align}
  \end{subequations}
\end{corollary}
\begin{proof}
  We have
  \begin{equation*}
    \normsqr[\Lin{\HS_n}]{W_n R_n}
    = \norm[\Lin{\HS_n}]{R_n^* W_n^2 R_n}
    = \norm[\Lin{\HS_n}]{R_n^* (\id_{\HS_n}-J_n^*J_n)R_n},
  \end{equation*}
  hence the first equality follows from~\eqref{eq:norm.summand2}.  A
  similar argument holds for~\eqref{eq:defect.op.est2}.
\end{proof}

Let us now define the parent space:
\begin{definition}[parent space associated with identification
  operators]
  \label{def:parent.hs}
  Let $(\map{J_n}{\HS_n}{\HS_\infty})_{n \in \N}$ be a sequence of
  contractions.  We call
  \begin{subequations}
    \label{eq:parent.hs}
    \begin{align}
      \label{eq:parent.hs.sum}
      \HSgen = \HS_\infty \oplus \bigoplus_{n=1}^\infty \HS_n,
      \qquad f=(f_\infty, f_1, f_2, \dots) \in \HSgen\\
      \label{eq:parent.hs.norm}
      \text{with}\quad
      \normsqr[\HSgen] f
      = \normsqr[\HS_\infty]{f_\infty}
      + \sum_{n=1}^\infty \normsqr[\HS_n]{f_n}
      < \infty
    \end{align}
  \end{subequations}
  the \emph{associated parent space}.  Moreover, we set
  \begin{subequations}
    \begin{equation}
      \label{eq:concrete.iso.n}
      \map {\iota_n}{\HS_n} \HSgen,
      \quad \iota_nf_n = (J_nf_n,0,\dots,0, W_n f_n,0,\dots)
    \end{equation}
    for $n \in \N$, where the entry $W_n f_n$ is at the $n$-th
    component.  Finally, we set
    \begin{equation}
      \label{eq:concrete.iso.infty}
      \map{\iota_\infty}{\HS_\infty}\HSgen,
      \quad \iota_\infty f_\infty=(f_\infty,0,\dots).
    \end{equation}
  \end{subequations}
  We call $(\iota_n)_{n \in \Nbar}$ the \emph{associated isometries}.
\end{definition}
The maps $\iota_n$ are actually isometries factorising $J_n$:
\begin{lemma}[proof of \Thm{dreamworld.exists}]
  \label{lem:concrete.iso}
  \noindent
  \begin{enumerate}
  \item
    \label{concrete.iso.a}
    $\iota_n$ and $\iota_\infty$ are isometries.
  \item
    \label{concrete.iso.b}
    The adjoints $\map{\iota_n^*}\HSgen{\HS_n}$ and
    $\map{\iota_\infty^*}\HSgen{\HS_\infty}$ act as
    \begin{equation*}
      \iota_n^*g=J_n^*g_\infty+ W_n g_n \qquadtext{and}
      \quad \iota_\infty^*g=g_\infty,
    \end{equation*}
    respectively, where $g=(g_\infty,g_1,g_2,\dots) \in \HSgen$.
  \item
    \label{concrete.iso.c}
    We have $J_n =\iota_{\infty}^*\iota_n$ and
    $J_n^*=\iota_n^*\iota_{\infty}$.
  \end{enumerate}
\end{lemma}
\begin{proof}
  \itemref{concrete.iso.a}~$\iota_n$ is an isometry, as
  \begin{align*}
    \normsqr[\HSgen]{\iota_nf_n}
    =\normsqr[\HS_\infty]{J_nf_n}
      + \sum_{k=1}^\infty \normsqr[\HS_k]{(\iota_nf_n)_k}
    =\normsqr[\HS_\infty]{J_nf_n}
    + \normsqr[\HS_n]{W_n f_n}
    =\normsqr[\HS_n]{f_n}
  \end{align*}
  using~\eqref{eq:hilfsglnorm}.
  Moreover, that $\iota_\infty$ is an isometry is obvious.

  \itemref{concrete.iso.b}~We have
  \begin{align*}
    \iprod[\HSgen]{\iota_nf_n}{g}
    = \iprod[\HS_\infty]{J_nf_n}{g_\infty}
      + \sum_{k=1}^\infty \iprod[\HS_k]{(\iota_n f_n)_k}{g_k}
    &= \iprod[\HS_\infty]{J_nf_n}{g_\infty}
      + \iprod[\HS_n]{W_n f_n}{g_n}\\
    &=\iprod[\HS_n]{f_n}{J_n^*g_\infty+W_n g_n}
  \end{align*}
  for $f_n \in \HS_n$ and $g \in \HSgen$, as $W_n^*=W_n$.  The claim
  $\iota_\infty^* g=g_\infty$ is obvious.

  \itemref{concrete.iso.c}~Obviously, we have
  $\iota_{\infty}^*\iota_nf_n=(\iota_n f_n)_\infty=J_nf_n$ and hence
  the second claim by taking adjoints.
\end{proof}

\begin{remark}
  \label{rem:dn.direct}
  We can express Weidmann's resolvent difference here as
  \begin{align}
    \nonumber
    D_n f
    &= \iota_n R_n \iota_n^*f - \iota_\infty R_\infty \iota_\infty^*f\\
    \label{eq:dn.direct}
    &=\bigl((J_n R_n J_n^* - R_\infty)f_\infty
      + J_n R_n W_n f_n, 0, \dots, 0,
      W_n R_n(J_n^* f_\infty + W_n f_n),0,\dots\bigr)
  \end{align}
  for all $f=(f_\infty,f_1,\dots,0,f_n,0,\dots) \in \HS$.
  Moreover, a direct proof of \Cor{main2a} is possible using
  \Cor{defect.ops} and the QUE-convergence.
\end{remark}

\subsection{Another equivalent characterisation of commuting
  projections}
\label{ssec:eq.char.comm.proj}
For further purposes, let us now calculate the projections
$P_n=\iota_n \iota_n^*$ and $P_\infty=\iota_\infty \iota_\infty^*$ in
this concrete situation: From~\Lemenum{concrete.iso}{concrete.iso.b}
we conclude
\begin{subequations}
  \begin{equation}
    \label{eq:pn}
    P_n f = \bigl(J_n J_n^* f_\infty+J_n W_nf_n,0,\dots,0,
    W_nJ_n^* f_\infty+(\id_{\HS_n}-J_n^*J_n)f_n,0,\dots\bigr)
  \end{equation}
  where the entry with $W_n$ is at the $n$-th position.
  Moreover,
  \begin{equation}
    \label{eq:pinfty}
    P_\infty f = (f_\infty,0,\dots),
  \end{equation}
  where, as usual, $f=(f_\infty,f_1,f_2,\dots) \in \HS$.
\end{subequations}

Let us now calculate the commutator $P_n P_\infty - P_\infty P_n$ needed
as assumption in \Thm{main2b} in the concrete case here:
\begin{lemma}
  \label{lem:proj.commut}
  We have
  \begin{subequations}
    \label{eq:proj.commut}
    \begin{align}
      \label{eq:proj.commut1}
      (P_nP_\infty-P_\infty P_n)f
      &=\bigl(
        -J_n W_n f_n,
        0, \dots, 0,
        W_n J_n^* f_\infty,
        0, \dots
        \bigr)\\
      \label{eq:proj.commut2}
      &=\bigl(
        -W_{\infty,n} J_n f_n,
        0, \dots, 0,
        J_n^* W_{\infty,n} f_\infty,
        0, \dots
        \bigr)
    \end{align}
  \end{subequations}
  for all $f=(f_\infty,f_1,f_2,\dots) \in \HS$.
\end{lemma}
\begin{proof}
  We have
  \begin{equation*}
    P_nP_\infty f
    = P_n(f_\infty,0,0,\dots)
    = \big(J_nJ_n^*f_\infty,0,\dots,0,W_nJ_n^*f_\infty,0,\dots)
  \end{equation*}
  and
  \begin{align*}
    P_\infty P_n f
    = (P_n f)_\infty
    =\bigl(J_nJ_n^*f_\infty+J_n W_nf_n,0, \dots \bigr);
  \end{align*}
  hence~\eqref{eq:proj.commut1} follows.
  Equality~\eqref{eq:proj.commut2} is a consequence
  of~\eqref{eq:intertwining}.
\end{proof}

We now continue with \Thm{proj.comm.jn.iso} within the setting of the
concrete parent space constructed above:
\begin{theorem}[equivalent characterisation of partial isometries]
  \label{thm:j.proj.com}
  Let $n \in \N$, then the following assertions are equivalent:
  \begin{enumerate}
  \item
    \label{j.proj.com.a}
    the identification operator $J_n$ is a partial isometry;

  \item
    \label{j.proj.com.b}
    the defect operator $W_n$ is an orthogonal projection (onto
    $\ker J_n$);

  \item
    \label{j.proj.com.b'}
    $\spec {W_n} \subset \{0,1\}$;

  \item
    \label{j.proj.com.c}
    $J_n W_n=0$;

  \item
    \label{j.proj.com.a'}
    the identification operator $J_n^*$ is a partial isometry;

  \item
    \label{j.proj.com.e}
    the defect operator $W_{\infty,n}$ is an orthogonal projection
    (onto $\ker J_n^*$);
  \item
    \label{j.proj.com.e'}
    $\spec {W_{\infty,n}} \subset \{0,1\}$;
  \item
    \label{j.proj.com.d}
    $W_{\infty,n}J_n=0$;
  \item
    \label{j.proj.com.f}
    the projections commute, i.e., $P_nP_\infty=P_\infty P_n$.
  \end{enumerate}
\end{theorem}
\begin{proof}
  \itemref{j.proj.com.a}$\implies$\itemref{j.proj.com.b}:
  We have
  \begin{equation*}
    W_n^4
    =(\id_{\HS_n}-J_n^*J_n)^2
    =\id_{\HS_n} - 2 J_n^*J_n +J_n^*J_nJ_n^*J_n
    =\id_{\HS_n} - J_n^*J_n
    =W_n^2
  \end{equation*}
  as $J_n=J_nJ_n^*J_n$.  In particular, $W_n^2$ is idempotent and
  (clearly) self-adjoint, hence an orthogonal projection with
  \begin{equation*}
    W_n^2 f_n=f_n
    \quadiff
    J_n^* J_n f_n=0
    \quadiff
    f_n \in \ker J_n,
  \end{equation*}
  i.e., onto $\ker J_n$.  For an orthogonal projection $P$ , we also
  have $P^{1/2}=P$, and hence $W_n$ itself is an orthogonal
  projection.

  \itemref{j.proj.com.b}$\iff$\itemref{j.proj.com.b'} and
  \itemref{j.proj.com.e}$\iff$\itemref{j.proj.com.e'} follow from the
  fact that a self-adjoint operator $P$ is an orthogonal projection
  if and only if $\spec P \subset \{0,1\}$.

  \itemref{j.proj.com.b}$\implies$\itemref{j.proj.com.c} is obvious.

  \itemref{j.proj.com.a}$\iff$\itemref{j.proj.com.a'} follows from
  \Lemenum{part.iso}{part.iso.a}$\iff$\itemref{part.iso.b}.

  \itemref{j.proj.com.a'}$\implies$\itemref{j.proj.com.e}%
  $\implies$\itemref{j.proj.com.d} follow as above (use
  $(W_{\infty,n}J_n)^*=J_n^*W_{\infty,n}$).

  \itemref{j.proj.com.c}$\lor$\itemref{j.proj.com.d}$\implies$%
  \itemref{j.proj.com.f} follows from \Lem{proj.commut}.

  \itemref{j.proj.com.f}$\implies$\itemref{j.proj.com.a} has already
  been proven in
  \Thmenum{proj.comm.jn.iso}{j.proj.comm.d}$\implies$\itemref{j.proj.comm.a}.
\end{proof}
In the special case of (co-)isometries we have:
\begin{corollary}[(co-)isometries]
  \label{cor:j.proj.com}
   $J_n$ is an isometry
  if and only if $W_n=0$; similarly, $J_n$ is a co-isometry if and
  only if $W_{\infty,n}=0$.
\end{corollary}

Let us finally state another property of our parent space constructed
here:
\begin{proposition}
  \label{prp:proj.conv.strongly}
  Assume that $\norm{J_n}\le 1$ for all $n \in \N$ and
  that~\eqref{eq:que2} holds then there is a parent space such that
  $P_n \strongto P_\infty$ strongly, i.e.,
  $\norm[\HSgen]{P_n f - P_\infty f} \to 0$ for all $f \in \HSgen$.
\end{proposition}
\begin{proof}
  Let $\HSgen$ be the parent space as constructed in
  \Subsec{concrete.iso}.  For $f \in \HSgen$ with
  $f=(f_\infty,f_1,f_2,\dots)$ such that $f_n\ne 0$ only for a finite
  number of elements, say $f_n=0$ for all $n>n_0$ and some
  $n_0 \in \N$.  Moreover, we assume that $f_\infty=R_\infty g_\infty$
  for some $g_\infty \in \HS_\infty$.  Then we have
  \begin{align}
    \nonumber
    \normsqr[\HS]{(P_n-P_\infty)f}
    &= \normsqr[\HS_\infty]{-(f_\infty-J_nJ_n^*f_\infty) + J_nW_n f_n}
      + \normsqr[\HS_n]{(f_n - J_n^*J_n f_n) + J_n^*W_{\infty,n} f_\infty}\\
    \label{eq:proof.proj.strong}
    &= \normsqr[\HS_\infty]{f_\infty-J_nJ_n^*f_\infty}
      + \normsqr[\HS_n]{J_n^*W_{\infty,n} f_\infty}\\
    \nonumber
    &\le \normsqr[\HS_\infty]{(\id_{\HS_\infty} -J_nJ_n^*)R_\infty g_\infty}
      + \normsqr[\HS_\infty]{W_{\infty,n} R_\infty g_\infty}
  \end{align}
  for $n>n_0$ by~\eqref{eq:intertwining}
  and~\eqref{eq:pn}--\eqref{eq:pinfty}.  Now,
  \begin{equation*}
    \normsqr[\Lin{\HS_\infty}]{W_{\infty,n} R_\infty}
    \le \norm[\Lin{\HS_\infty}]{R_\infty} \delta_n \to 0
  \end{equation*}
  as $n \to \infty$ using~\eqref{eq:defect.op.est2}
  and~\eqref{eq:que2}.  In particular, we conclude that
  $\normsqr[\HS]{(P_n-P_\infty)f} \to 0$ as $n \to \infty$.  Since the
  set of $f$ of the above form are dense in $\HSgen$, the result
  follows.
\end{proof}

A simple consequence is the following:
\begin{corollary}
  \label{cor:proj.conv.strongly}
  The commutator of $P_n$ and $P_\infty$ converges strongly to $0$.
\end{corollary}
\begin{proof}
  We have
  \begin{align*}
    \norm[\HS]{(P_nP_\infty -P_\infty P_n)f}
    &=\norm[\HS]{(P_nP_\infty- P_n + P_\infty-P_\infty P_n + P_n- P_\infty)f}\\
    &\leq \norm[\HS]{(P_n- P_\infty)f}
      + \norm[\HS]{(P_\infty- P_n)f}
      +\norm[\HS]{(P_\infty- P_n)f}
  \end{align*}
  for all $f \in \HSgen$, and the result follows from
  \Prp{proj.conv.strongly}.
\end{proof}
\begin{remarks}[convergence of the projections]
  \label{rem:proj.conv.strongly}
  \indent
  \begin{enumerate}
  \item We have seen the relevance of $P_n \strongto P_\infty$ already
    in Stummel's concept discussed near~\eqref{eq:proj.conv.strongly}.
  \item
    \label{proj.conv.strongly.a}
    \myparagraph{Strong convergence not an invariant.}  %
    Note that $P_n \strongto P_\infty$ is not an invariant of a parent
    space, i.e., it may be true in one parent space, but not in
    another one.  In \Prp{proj.conv.strongly} we have shown
    $P_n \strongto P_\infty$ for the concrete parent space of
    \Subsec{concrete.iso}, but it may not be true in others (see
    \Ex{weidmanns-excercise}).
  \item
    \label{proj.conv.strongly.b}
    \myparagraph{No norm convergence of the projections.}  %
    We cannot expect that $\norm{P_n-P_\infty} \to 0$ in
    operator norm: for $f=(f_\infty,0,\dots)$ we conclude
    \begin{align*}
      \normsqr[\HS]{(P_n-P_\infty)f}
      \ge \normsqr[\HS_\infty]{W_{\infty,n}^2 f_\infty}
    \end{align*}
    from~\eqref{eq:proof.proj.strong}.  If $J_n$ is a partial isometry
    then $W_{\infty,n}$ is an orthogonal projection onto $\ker J_n^*$,
    and $\ker J_n^* \ne \{0\}$ if $J_n$ is not a co-isometry.  In
    particular, $\norm{P_n-P_\infty} \ge \norm{W_{\infty,n}} =1$ in
    this case.  If $J_n$ is a co-isometry (and not an isometry), then
    $W_n$ is an orthogonal projection onto $\ker J_n \ne \{0\}$, and
    we have
    \begin{align*}
      \normsqr[\HS]{(P_n-P_\infty)f}
      \ge \normsqr[\HS_\infty]{W_n^2 f_n}
    \end{align*}
    for $f=(0,0,\dots,f_n,0,\dots) \in \HSgen$ again
    by~\eqref{eq:proof.proj.strong}.  As before, we have
    $\norm{P_n-P_\infty} \ge \norm{W_n} =1$.  In particular, we have
    shown that if $J_n$ is a partial isometry (and not unitary), then
    $\norm{P_n - P_\infty}=1$.  If $J_n$ is unitary then
    $P_n=P_\infty$.

    Nevertheless, the commutator $P_n P_\infty - P_\infty P_n$ (if not
    already $0$) might converge to $0$ in operator norm in some
    examples as we will see in \Subsec{met.graphs2}.
  \end{enumerate}
\end{remarks}

\subsection{A minimal parent space}
\label{ssec:min.parent.space}

Let us now see how the \emph{minimal} parent space (cf.\
\Defenum{dreamworld}{dreamworld.c}) in the concrete construction of
\Def{parent.hs} looks like:
\begin{lemma}
  \label{lem:hs.min}
  Assume that $W_n(\HS_n)$ is closed\footnote{This is equivalent with
    the fact that $0$ is isolated in $\spec {W_n}$}.  then we have
  \begin{equation*}
    \HSmin =
    \HS_\infty \oplus \bigoplus_{n \in \N} W_n(\HS_n),
  \end{equation*}
  i.e., we can replace $\HS_n$ by the orthogonal complement in $\HS_n$
  of the space where $J_n$ is an isometry.  In particular, if $J_n$ is
  an isometry, then $\ran W_n=\{0\}$.
\end{lemma}
\begin{proof}
  ``$\subseteq$'': clearly $\iota_\infty(\HS_\infty) \subset \HSmin$,
  and as the $n$-th component of $\iota_n f_n$ is $W_nf_n$, hence an
  element of $\clo{\ran W_n}$.

  ``$\supseteq$'': Let $f=(f_\infty,0,\dots)$, then
  $f=\iota_\infty f_\infty \in \HSmin$.  Moreover, if
  $f=(0,\dots,0,f_n,0,\dots) \in \HS_\infty \oplus \bigoplus_{n \in
    \N} W_n(\HS_n)$ with $f_n=W_n g_n$ and $g_n \in \HS_n$ then
  $f=\iota_n g_n - \iota_\infty J_n g_n \in \HSmin$.%
\end{proof}

If $J_n$ is a partial isometry, we conclude from
\Subsec{eq.char.comm.proj}:
\begin{proposition}
  \label{prp:jn.part.iso.concrete}
  Assume that $J_n$ is a partial isometry for each $n \in \N$.  Then
  \begin{equation*}
    \HSmin
    = \HS_\infty \oplus \bigoplus_{n \in \N} \ker J_n,
    \quad
    \iota_nf_n=(J_n f_n,0,\dots,0,f_n'',0,\dots),
    \quad
    \iota_\infty f_\infty=(f_\infty,0,\dots),
  \end{equation*}
  where $f_n''=f_n-J_n^*J_n f_n$ is the orthogonal projection of $f_n$
  onto $\ker J_n$.
\end{proposition}
\begin{proof}
  If $J_n$ is a partial isometry then $W_n=\id_{\HS_n}-J_n^*J_n$ is
  the orthogonal projection onto $\ker J_n$ by
  \Thmenum{j.proj.com}{j.proj.com.b}.  In particular,
  $W_n(\HS_n)=\ker J_n$ is closed and \Lem{hs.min} applies.
\end{proof}

\subsection{A change of identification operators}
\label{ssec:change.id}
As a final step in the proof of our main result \Thm{main}, we want to
get rid of the condition $\norm {J_n} \le 1$ used in \Cor{main2a} and
also in \Subsec{concrete.iso}.  Therefore, we rescale the
identification operators so that their norm is less than or equal to
$1$, but the convergence is maintained.  To do so, we have a closer
look on the effect of a slight change of the identification operators
$J_n$.
\begin{definition}[$\delta$-equivalence of identification operators]
  \label{def:que-equivalence}
  The identification operators $J$ and
  $\hat J \in \mathfrak{L}(\HS_1,\HS_2)$ are called
  \emph{$\delta$-equivalent} for $\delta >0$, if
  \begin{equation}
    \label{eq:def_delta_eq}
  \norm[\mathfrak{L}(\HS_1,\HS_2)]{J-\hat{J}} \leq \delta.
\end{equation}
\end{definition}
If this $\delta$-equivalence among two pairs is given, then
quasi-unitary equivalence gets transferred.

The following two propositions show that we can always assume
$\norm[\mathfrak{L}(\HS_n,\HS_\infty)]{J_n} \leq 1$ without loss of
generality.
\begin{proposition}
  \label{prp:eq_dann_que}
  Let $\delta>0$ and $\delta'>0$.  If $\D_1$ and $\D_2$ are
  $\delta$-quasi-unitary equivalent w.r.t.\ $J$, and if $J$ and
  $\hat J$ are $\delta'$-equivalent, then $\D_1$ and $\D_2$ are
  $\hat \delta$-quasi-unitary equivalent w.r.t.\ $(\hat{J},\hat{J}^*)$
  with
  \begin{equation*}
    \hat \delta
    = 
    \delta + \delta'\max\{\norm{R_1},\norm{R_2}\}
    \bigl(2(1+\delta)+\delta'\bigr)
    \in \delta+\Err(\delta').
  \end{equation*}
\end{proposition}
\begin{proof}
  For the proof we have to face the four inequalities given in
  \Def{que}: for the first, we have
  \begin{equation}
    \norm[\mathfrak{L}(\HS_1,\HS_2)]{\hat{J}}
    \le  \norm[\Lin{\HS_1,\HS_2}] J
    +    \norm[\Lin{\HS_1,\HS_2}]{\hat J-J}
    \le 1+\delta+\delta'.
  \end{equation}
  Similarly, we obtain
  \begin{align*}
    \bignorm[\Lin{\HS_1}]{(\id_{\HS_1}-\hat{J}^*\hat{J}){R_1} }
    &= \norm[\Lin{\HS_1}]{\big(\id_{\HS_1}
      -(J+\hat{J}-J)^*(J+\hat{J}-J) \big)R_1}\\
    &= \norm[\Lin{\HS_1}]{\big(\id_{\HS_1}-J^*J
      +J^*(\hat{J}-J) + (\hat{J}-J)^*J
      + (\hat{J}-J)^*(\hat{J}-J)  \big)R_1 }\\
    &\le \norm[\Lin{\HS_1}]{\big(\id_{\HS_1}-J^*J  \big)R_1 }
      + \norm[\Lin{\HS_1}]{J^*(\hat{J}-J)R_1 }\\
    & \hspace*{0.1\textwidth}{}
      + \norm[\Lin{\HS_1}]{(\hat{J}-J)^*J R_1 }
      + \norm[\Lin{\HS_1}]{(\hat{J}-J)^*(\hat{J}-J)R_1}\\
    &\le \delta + 2(1+\delta)\delta' \norm[\Lin{\HS_1}]{R_1}
      + (\delta')^2 \norm[\Lin{\HS_1}]{R_1},
  \end{align*}
  and similarly for
  $\norm[\Lin{\HS_2}]{(\id_{\HS_2}-\hat{J}\hat{J}^*){R_2}}$.  For the
  last inequality we can insert $J$ as well and obtain
  \begin{align*}
    \norm[\Lin{\HS_1, \HS_2}]{\hat J R_1-R_2 \hat J}
    &\le \norm[\Lin{\HS_2}]{(\hat J-J){R_1}}
      + \norm[\Lin{\HS_1 ,\HS_2}]{J R_1 -R_2 J}
      +\norm[\Lin{\HS_1}]{R_2(\hat J-J)}\\
    &\le \delta +\delta' \big(\norm[\Lin{\HS_1}] {R_1}
      +\norm[\Lin{\HS_2}]{R_2}\big).
  \end{align*}
  Therefore we can conclude the claim by choosing $\hat \delta$ as the
  maximum of those three terms.
\end{proof}
Finally, we show that if $\norm{J_n}>1$ then we can replace $J_n$ by
an equivalent identification operator with norm equal to $1$:
\begin{proposition}
  \label{prp:eq-que}
  Let $\D_1, \D_2$ be $\delta$-quasi-unitary equivalent with
  identification operator $J$.  If $\norm[\Lin{\HS_1,\HS_2}]{J} \ge 1$
  then $\D_1$ and $\D_2$ are $\hat \delta$-quasi-unitary equivalent
  with identification operator $\hat J$, where
  \begin{equation*}
    \hat J= \frac 1{\norm J} J
    \qquadtext{and}
    \hat \delta
    = \delta\bigl(
    1 + \max\{\norm{R_1},\norm{R_2}\}(2+3\delta)\bigr).
  \end{equation*}
\end{proposition}
\begin{proof}
  By \Prp{eq_dann_que} it is enough to show that $J$ and
  $\hat J$  are
  $\delta$-equivalent: We have
  \begin{equation*}
    \norm{J-\hat J}
    = \Bignorm{J-\frac{1}{\norm J}J}
    = \norm{J}-1
    \le \delta
  \end{equation*}
  using $\norm J \ge 1$.  In particular, from \Prp{eq_dann_que} with
  $\delta'=\delta$, we conclude the result.
\end{proof}

\begin{corollary}
  \label{cor:eq-que}
  Assume that $\D_n \gnrc \D_\infty$ with convergence speed
  $(\delta_n)_n$ and with identification operators $J_n$ not
  necessarily being contractions ($\norm{J_n}>1$ may happen).  Then
  $\D_n \gnrc \D_\infty$ with identification operators $\wt J_n$ with
  $\norm{\wt J_n} \le 1$ and convergence speed
  \begin{equation}
    \label{eq:conv.speed'}
    \hat \delta_n
    \le \delta_n\bigl(
    1 + (\norm{R_\infty} +2\delta_n)(2+3\delta_n)\bigr)
    \in \Err(\delta_n).
  \end{equation}
\end{corollary}
\begin{proof}
  We apply \Prp{eq-que} with $\D_1=\D_n$ and $\D_2=\D_\infty$.  From
  \Lem{res.est} we obtain $\norm{R_n} \le \norm{R_\infty}+2\delta_n$.
\end{proof}

We now have proven the last step of our main result using \Cor{eq-que}
and \Cor{main2a}:
\begin{theorem}[QUE-convergence implies Weidmann's convergence]
  \label{thm:main2}
  Let $\D_n$ be a self-adjoint operator in a Hilbert space $\HS_n$ for
  $n \in \Nbar$ such that $\D_n \gnrc \D_\infty$ QUE-converges with
  speed $(\delta_n)_n$.  Then $\D_n \gnrcW \D_\infty$ Weidmann-converges with
  speed
  \begin{equation}
    \hat \delta_n
    \le \wt \delta_n\bigl(
    1 + (\norm{R_\infty} +2\wt\delta_n)(2+3\wt \delta_n)\bigr)
    \in \Err(\sqrt{\delta_n}),
  \end{equation}
  where $\wt \delta_n \in \Err(\sqrt{\delta_n})$ is given
  in~\eqref{eq:conv.speed}.
\end{theorem}

\section{Examples}
\label{sec:examples}
%

In all examples one encounters convergence of Laplace-like operators
on varying spaces $X_n$ with limit space $X_\infty$, therefore we
assume here that $X_n$ is a measure space with measure $\mu_n$ for
each $n \in \Nbar$.  In particular, we have
\begin{equation*}
  \HS_n := \Lsqr{X_n}
  \qquad\text{for $n \in \Nbar=\N \cup \{\infty\}$.}
\end{equation*}

Note that Weidmann's convergence \emph{always} implies the
QUE-convergence with the same convergence speed by \Thm{main1} for
\emph{any} parent space and isometries factorising the identification
operators (not only the parent space associated with the
identification operators of \Def{parent.hs}).

\subsection{Examples with natural common parent space}
\label{ssec:ex.nat.parent.space}
For the first example class assume that there is a common measure
space $X$ with measure $\mu$, that $X_n \subset X$ is a measurable
subset and that $\mu_n$ is the corresponding restriction of $\mu$ to
$X_n$ for $n \in \Nbar$.  In particular, we can embed elements of
$\HS_n = \Lsqr{X_n}$ naturally into $\HS :=\Lsqr X$, by extending it by
$0$.  Denote this extension of an element $f_n \in \Lsqr{X_n}$ to
$\Lsqr X$ by $f_n \oplus 0_{X \setminus X_n}$.

\subsubsection*{Weidmann's convergence}
We then have the natural isometries
\begin{align*}
  \map{\iota_n&}{\HS_n}{\HS},
  & \iota_n f_n&=f_n \oplus 0_{X \setminus X_n}
  &(n \in \Nbar).
\end{align*}
Identifying $f_n$ with $\iota_n f_n$, we are actually in the situation
originally considered by Weidmann, namely that $\HS_n$ is a
\emph{subspace} of the parent space $\HS$.  The adjoint $\iota_n^*$ is
the restriction of an equivalent class of functions $f$ from $\Lsqr X$
to $X_n$, denoted by $\iota_n^* f= f\restr{X_n}$.  Moreover, the norm
of the lifted resolvent difference is given
in~\eqref{eq:norm.weid.ex}.

The corresponding orthogonal projections are
  \begin{align*}
    P_nf=\iota_n \iota_n^*f
        =f \restr{X_n} \oplus 0_{X \setminus X_n}
    \quadtext{and}
    P_\infty f=\iota_\infty \iota_\infty^*f
        =f\restr{X_\infty} \oplus 0_{X \setminus X_\infty}.
  \end{align*}
In particular, the projections can be written as $P_n=\1_{X_n,X}$ and
$P_\infty=\1_{X_\infty,X}$ using $\map{\1_{M,X}} X {\{0,1\}}$ for the
indicator function ($\1_{M,X}(x)=1$ if $x \in M \cap X$ and
$\1_{M,X}(x)=0$ if $x \in X \setminus M$) as well as for the
corresponding multiplication operator in $\Lsqr{X}$.  Moreover,
the projections commute as we have
\begin{equation}
  \label{eq:proj.comm.ex}
  P_nP_\infty
  =\1_{X_n \cap X_\infty, X}
  =P_\infty P_n
  \qquad\text{for all $n \in \N$.}
\end{equation}

Let us present another example with multiplication operators as in our
motivating example~\Ex{motivating.example1}
showing the possibility that the limit space suddenly shrinks:
\begin{example}[suddenly shrinking limit space, non-converging
  projections]
  \label{ex:weidmanns-excercise}
  Let
  \begin{equation*}
    X=X_n=[0,1/2] \cup [2,\infty)
    \qquadtext{and}
    X_\infty=[0,1/2]
  \end{equation*}
  with Lebesgue measure.  Moreover, let $A_n$ be the multiplication
  operator multiplying with the function $a_n(x)=x^n$ and
  $A_\infty=0$.  Here, Weidmann's generalised norm resolvent
  convergence holds, as
  \begin{align*}
    \bignormsqr[\Lsqr X]%
    {(\iota_n R_n \iota_n^* - \iota_\infty R_\infty \iota_\infty^*)f}
    &= \int_{X_\infty}\Bigabssqr{\Bigl(\frac1{1+x^n}-1\Bigr) f(x)} \dd x
    + \int_{X \setminus X_\infty}
    \Bigabssqr{\frac1{1+x^n} \cdot f(x)} \dd x\\
    &= \int_{[0,1/2]} \Bigabssqr{\frac{x^n}{1+x^n} \cdot f(x)} \dd x
    + \int_{[2,\infty)}
      \Bigabssqr{\frac1{1+x^n} \cdot f(x)} \dd x\\
    &\le \frac 1{(2^n)^2}
       \int_{[0,1/2] \cup [2,\infty)} \abssqr{f(x)} \dd x
      = \frac 1{(2^n)^2} \normsqr[\Lsqr X] f.
  \end{align*}
  Note that the limit space suddenly shrinks, and $P_n=\id_{\HS}$ does
  not converge to $P_\infty=\1_{[0,1/2]}$ in any topology
  (cf.\ \Remenum{proj.conv.strongly}{proj.conv.strongly.a}).
\end{example}

\subsubsection*{QUE-convergence}
The identification operators obtained from Weidmann's isometries are
  \begin{align*}
    \label{eq:j.ex}
    J_nf_n = \iota_\infty^*\iota_nf_n
        =(f_n \oplus 0_{X \setminus X_n})\restr{X_\infty}
    \quadtext{and}
    J_n^*f_\infty = \iota_n^*\iota_\infty f_\infty
        =(f_\infty \oplus 0_{X \setminus X_\infty}) \restr{X_n}.
  \end{align*}

From \Thms{main1}{main2b} together with~\eqref{eq:proj.comm.ex} we
conclude the following characterisation of norm resolvent convergence
in this case for multiplication operators.  Actually, a direct proof
follows easily from the motivating
examples~\Exs{motivating.example1}{motivating.example2} and especially
from~\eqref{eq:norm.weid.ex}:
\begin{proposition}
  \label{prp:eq.char.gnrc.mult}
  If $A_n$ is the multiplication operator with the measurable (not
  necessarily bounded) function $\map {a_n}{X_n}\R$ then the following
  are equivalent:
  \begin{enumerate}
  \item
    \label{eq.char.gnrc.mult.a}
    $A_n \gnrcW A_\infty$ with convergence speed of order $\delta_n \to 0$;

  \item
    \label{eq.char.gnrc.mult.b}
    $A_n \gnrc A_\infty$ with convergence speed of order $\delta'_n \to 0$;

  \item
    \label{eq.char.gnrc.mult.c}
    We have
    \begin{gather*}
      \norm[\Linfty {X_\infty \setminus X_n}]{(a_\infty-\im)^{-1}}
      \le \delta'_n, \qquad
      \norm[\Linfty {X_n \setminus X_\infty}]{(a_n-\im)^{-1}}
      \le \delta'_n
      \qquad\text{and}\\
      \norm[\Linfty {X_n \cap X_\infty}]{(a_n-\im)^{-1}-(a_\infty-\im)^{-1}}
      \le \delta'_n
    \end{gather*}
    and $\delta_n \to 0$.
  \end{enumerate}
  For
  $\itemref{eq.char.gnrc.mult.a} \implies
  (\itemref{eq.char.gnrc.mult.b} \vee \itemref{eq.char.gnrc.mult.c})$
  we can choose $\delta'_n=\delta_n$; for
  $(\itemref{eq.char.gnrc.mult.b} \vee \eqref{eq.char.gnrc.mult.c})
  \implies \itemref{eq.char.gnrc.mult.a}$ we can choose
  $\delta_n=3\delta'_n$.
\end{proposition}

\begin{remark}[abstract minimal parent space not always optimal]
  \label{rem:minimal.parent.space.not.opt}
  Let us now calculate the abstract (minimal) parent space associated
  with the identification operators $J_n$ (see \Def{parent.hs} and
  \Lem{hs.min}).  The defect operators are
  \begin{align*}
    W_n
    = (\id_{\HS_n}-J_n^* J_n)^{1/2}
    &  = \1_{X_n \setminus X_\infty,X_n}
      \colon \Lsqr{X_n} \to \Lsqr{X_n}
      \quad\text{and}\\
    W_{\infty,n}
    = (\id_{\HS_\infty}-J_n J_n^*)^{1/2}
    &  = \1_{X_\infty \setminus X_n,X_\infty}
      \colon \Lsqr{X_\infty} \to \Lsqr{X_\infty}.
  \end{align*}
  Since
  $\ran W_n = \1_{X_n \setminus X_\infty,X_n}(\Lsqr{X_n})=\Lsqr{X_n
    \setminus
    X_\infty}$, 
  the minimal parent space $\HSmin$ is
  \begin{equation}
    \label{eq:hsmin.ex}
    \HSmin
    = \Lsqr{X_\infty} \oplus
    \bigoplus_{n \in \N} \Lsqr{X_n \setminus X_\infty}
  \end{equation}
  by \Lem{hs.min}.  Moreover, the corresponding isometries
  $\map{\iota_{n,\min}}{\HS_n}\HSmin$ ($n \in \Nbar$) are
  \begin{equation}
    \label{eq:iso.ex}
    \iota_{n,\min} f_n
    =\bigl((f_n \oplus 0_{X\setminus X_n})\restr{X_\infty},0,\dots,0,
    f_n\restr{X_n \setminus X_\infty}
    0,\dots \bigr)
    \quadtext{and}
    \iota_{\infty,\min} f_\infty=(f_\infty,0,\dots).
  \end{equation}
  We can see here, that the abstract minimal parent space is not
  always as simple as the obvious choice $\HS=\Lsqr X$.

  Nevertheless, this abstract construction (including the artificially
  added sequence of spaces $\Lsqr{X_n \setminus X_\infty}$) leads to
  the strong convergence of the corresponding projections, i.e.,
  \begin{equation*}
    P_{n,\min} f
    = (\1_{X_n \cap X_\infty,X_\infty} f_\infty,0,\dots,
    \1_{X_n \setminus X_\infty,X_n} f_n,0,\dots)
    \to P_{\infty,\min} f = (f_\infty,0,\dots)
  \end{equation*}
  for all $f=(f_\infty,f_1,f_2,\dots) \in \HSmin$.
\end{remark}
\begin{example}[\Ex{weidmanns-excercise} revisted]
  \label{ex:weidmanns-excercise2}
  The QUE-convergence holds here with
  \begin{align*}
    \map{J_n}{\HS_n=\Lsqr{[0,1/2]\cup[2,\infty)}}
    {\HS_\infty=\Lsqr{[0,1/2]}}, \quad
    J_n f_n = f_n \restr{[0,1/2]}.
  \end{align*}
  Then $J_n^*f_\infty=f_\infty \oplus 0_{[2,\infty)}$,
  $J_n^*J_n= \1_{[0,1/2]}$ and $J_nJ_n^*=\id_{\HS_\infty}$ (for
  simplicity, we just write $\1_{[0,1/2]}$ instead of the above
  introduced $\1_{[0,1/2],X_\infty}$).  The defect operators are
  \begin{align*}
    W_n=\1_{[2,\infty)}
    \qquadtext{and}
    W_{\infty,n}
    =0.
  \end{align*}
  The minimal Hilbert space is here
  \begin{align*}
    \HSmin
    =\Lsqr{[0,1/2]} \oplus \bigoplus_{n \in \N} \Lsqr{[2,\infty]}.
  \end{align*}
  Here, the strong convergence $P_{n,\min} \strongto P_{\infty,\min}$ (cf.\
  \Prp{proj.conv.strongly}) holds, as we have
  \begin{equation*}
    P_{n,\min} f
    = (f_\infty,0,\dots,
    \1_{[2,\infty)} f_n,0,\dots)
    \to P_{\infty,\min} f = (f_\infty,0,\dots)
  \end{equation*}
  for all $f=(f_\infty,f_1,f_2,\dots) \in \HSmin$.  It does not hold in
  the natural parent space $\HS=\Lsqr{X}$, as we have seen in
  \Ex{weidmanns-excercise}.
\end{example}

\subsubsection*{The identification operator $J_n$ is an isometry.}
If $X_n \subset X_\infty$ holds for all $n \in \N$, then we can choose
\begin{equation*}
  \HSmin=\HS_\infty=\Lsqr {X_\infty}
\end{equation*}
as minimal parent space and $\map{\iota_n}{\HS_n}{\HS_\infty}$ as
$\iota_n f_n=f_n \oplus 0_{X_\infty \setminus X_n}$ for $n \in \N$ and
$\iota_\infty=\id_{\HS_\infty}$.  Moreover, Weidmann's generalised
resolvent convergence means that
\begin{equation}
  \label{eq:gnrc.weid.ex}
  J_n R_n J_n^* -R_\infty
  = \iota_n R_n \iota_n^* - R_\infty
  = D_n
  = R_n \oplus 0_{\Lsqr{X_\infty \setminus X_n}} - R_\infty
\end{equation}
converges in operator norm in $\Lsqr{X_\infty}$ to $0$.  Here,
$0_{\Lsqr{X_\infty \setminus X_n}}$ is the $0$-operator on
$\Lsqr{X_\infty \setminus X_n}$.

In the QUE-setting, $J_n=\iota_n$ is an isometry, i.e.,
$J_n^*J_n=\id_{\HS_n}$ and for the QUE-convergence we need only the
second estimate in~\eqref{eq:que2}.  It is equivalent with
\begin{equation}
  \label{eq:que2-ex}
  \normsqr[\Lsqr {X_\infty}]{(\id_{\HS_\infty}-J_n J_n^*)g}
  = \int_{X_\infty \setminus X_n} \abssqr {g} \dd \mu
  = \normsqr[\Lsqr{X_\infty \setminus X_n}] {g}
  \le \delta_n^2 \normsqr[\Lsqr {X_\infty}] {(A_\infty+1)g}
\end{equation}
for all $g \in \dom A_\infty$ (substitute $g=R_\infty f$).  Moreover,
the last estimate~\eqref{eq:que3} is equivalent with the operator norm
in~\eqref{eq:gnrc.weid.ex} being not greater than $\delta_n$.

\subsubsection*{Laplacians with obstacles}
We assume here that $X$ is an open subset of $\R^d$ or a Riemannian
manifold of bounded geometry, see~\cite[Sec.~3.2]{anne-post:21} for
details.  In previous works, the first author has established
QUE-convergence of Laplace-like operators.  We will show here that
equivalently, Weidmann's convergence also holds with the same
convergence speed.

In almost all examples where QUE-convergence has been established, one
shows the stronger estimate
\begin{equation}
  \label{eq:norm.que1.ex}
  \normsqr[\Lsqr{X \setminus X_n}] {g}
  \le \delta_n^2 \normsqr[\Lsqr X] {(A_\infty+1)^{1/2}g}
  = \delta_n^2 (\qf a_\infty(g)+\normsqr[\Lsqr X]{g})
\end{equation}
for all $g \in \dom \qf a_\infty$.  Here, $\qf a_\infty$ is the
quadratic form associated with $A_\infty$ (see
e.g.~\cite[Sec.~3.1]{post:12}).  The stronger estimate implies the
weaker one~\eqref{eq:que2-ex} since $\norm{(A_\infty+1)^{1/2}g} \le
\norm {(A_\infty+1)g}$ for all $g \in \dom A_\infty$.  If $A_\infty$
is the Laplace operator with Neumann or Dirichlet boundary conditions
on $X_\infty \subset \R^d$ then $\dom \qf a_\infty \subset \Sob
{X_\infty}$ and $\qf a_\infty(g)=\normsqr[\Lsqr
  {X_\infty}]{\abs{\nabla g}}$.  Moreover,~\eqref{eq:norm.que1.ex} is
equivalent with
\begin{equation*}
  \int_{X_\infty \setminus X_n} \abssqr{g} \dd \mu
  \le  \delta_n^2 \int_{X_\infty} \bigl(\abssqr{\nabla g}
  + \abssqr {g}\bigr) \dd \mu
\end{equation*}
for $g \in \Sob {X_\infty}$.  We named this property
\emph{$(X_\infty \setminus X_n,X_\infty)$-non-concentrating}
in~\cite[Def.~3.7]{anne-post:21}). 

\begin{examples}[Laplace operators and fading obstacles]
  \label{ex:dir.neu.holes}
  We first assume that $X=X_\infty$ and that
  $X_n= X_\infty \setminus B_n$ where $B_n \subset X$ are given by
  \begin{equation*}
    B_n = \bigcup_{p \in I_n} \clo B_{1/n}(p),
    \quadtext{where}
    \clo B_r(p):=\set{x \in X}{ d(x,p)\le r}
  \end{equation*}
  and $I_n$ is a discrete set such that for all $p,q \in I_n$ and
  $p \ne q$, we have $d(p,q) \ge 2/n^\alpha$ for some
  $\alpha \in (0,1)$.  Note that the union of balls is disjoint.  In
  both cases below, the limit operator $\D_\infty$ does not see the
  obstacles any more (hence we called them \emph{fading}).

  \begin{enumerate}
  \item \myparagraph{Neumann Laplacians on fading obstacles.}  Let
    $\qf a_n$ be the quadratic form with
    $\qf a_n(g)=\normsqr[\Lsqr {X_n}]{\abs{\nabla g}}$ and with domain
    $\dom \qf a_n=\Sob{X_n}$.  Then the associated operator $\D_n$ is
    (minus) the Laplacian on $X_n$ ($\D_n f_n= \nabla^* \nabla f_n$)
    with Neumann boundary conditions (the normal derivative of $f_n$
    on $\bd X_n$ vanishes) for $n \in \Nbar$.  We have shown
    in~\cite[Thm.~4.7]{anne-post:21} that $\D_n \gnrc \D_\infty$ with
    convergence speed $\delta_n\in \Err(1/n^{1-\alpha})$ if $d \ge 3$
    resp.\ $\delta_n \in \Err((\log n)^{1/2}/n^{1-\alpha})$ if $d=2$
    for any $\alpha \in (0,1)$.  We refer to~\cite{anne-post:21} for
    more details.
  \item \myparagraph{Dirichlet Laplacians on fading obstacles.}  Let
    $\qf a_n$ be the quadratic form as above with domain $\dom \qf
    a_n=\Sobn{X_n}$ (functions vanishing on $\bd B_n$).  Then the
    associated operator $\D_n$ is (minus) the Laplacian on $X_n$, but
    now with Dirichlet boundary conditions ($f_n$ vanishes on $\bd
    X_n$) for $n \in \Nbar$.  Assume here (for simplicity) that $d=3$.
    If $\alpha \in (0,1/3)$ then $\D_n \gnrc \D_\infty$ with
    convergence speed $\delta_n \in \Err(1/n^{(1/3-\alpha)/2})$
    (\cite[Cor.~5.7]{anne-post:21}).  The critical parameter
    $\alpha=1/3$ is the \emph{capacity}, see the discussion
    in~\cite[Rem.~5.8]{anne-post:21}), and
    also~\cite{khrabustovskyi-post:18} for a treatment of the case
    $\alpha=1/3$.
  \end{enumerate}
  From \Thms{main1}{main2b} we conclude that in both cases, the
  established QUE-convergence is equivalent with
  $\D_n\gnrcW \D_\infty$ (i.e., the resolvents $R_n$ of $\D_n$
  fulfil~\eqref{eq:gnrc.weid.ex}) with the same convergence speed.
\end{examples}

\subsubsection*{The identification operator $J_n$ is a co-isometry.}
Assume now that $X_\infty$ is eventually contained in $X_n$.  For
simplicity, we assume that $X_\infty \subset X_n$ for all
$n \in \N$.  In this case, obviously $J_n^*$ is an isometry, i.e.,
$J_n$ a co-isometry, and we have
\begin{equation*}
  J_nf_n=f_n \restr{X_\infty}
  \qquadtext{and}
  J_n^* f_\infty = f_\infty \oplus 0_{X_n \setminus X_\infty}
\end{equation*}
A parent space is either $X$, or the (possibly smaller) set
$\bigcup_{n \in \N} X_n \subset X$, or, using the abstract
construction, it is as in~\eqref{eq:hsmin.ex}.  We continue
\Ex{dir.neu.holes}:
\begin{figure}[h]
  \centering
  \begin{picture}(0,0)%
    \includegraphics{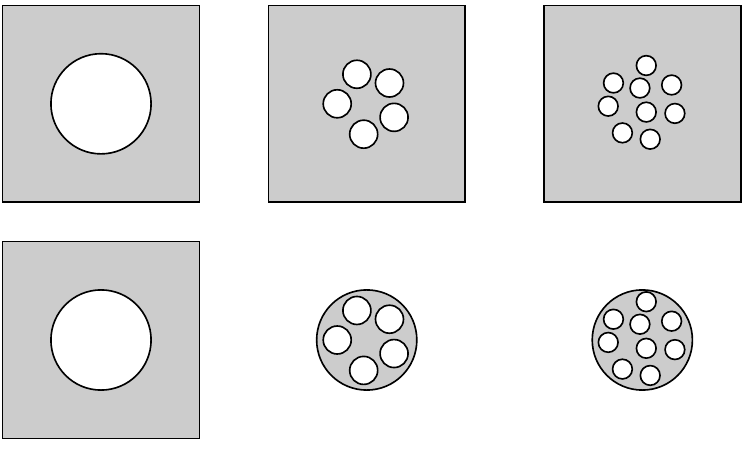}%
  \end{picture}%
  \setlength{\unitlength}{4144sp}%
  \begin{picture}(3399,2147)(439,-1746)
    \put(541,-446){\makebox(0,0)[lb]{$X_\infty$}}
    \put(1701,-446){\makebox(0,0)[lb]{$X_{n_1}$}}
    \put(2979,-446){\makebox(0,0)[lb]{$X_{n_2}$}}
    \put(507,-1582){\makebox(0,0)[lb]{$X_\infty$}}
    \put(1676,-1582){\makebox(0,0)[lb]{$X_{n_1}\setminus X_\infty$}}
    \put(3001,-1582){\makebox(0,0)[lb]{$X_{n_2}\setminus X_\infty$}}
  \end{picture}%

  \caption{The Dirichlet solidifying case.  \emph{Top row:} The
    parent space $X$ is the entire square, $X_\infty=X \setminus S$
    is the limit space, where $S$ is the centred ball.  Moreover,
    $X_{n_j}=X \setminus B_{n_j}$ are the approximating spaces for
    $n_1 < n_2$, where the obstacles $B_{n_j}$ (balls of radius
    $1/n_j$) are taken out.  \emph{Bottom row:} Three components of
    the abstract minimal parent space $\HSmin$, see
    \Rem{minimal.parent.space.not.opt}.}
  \label{fig:dir-sol}
\end{figure}
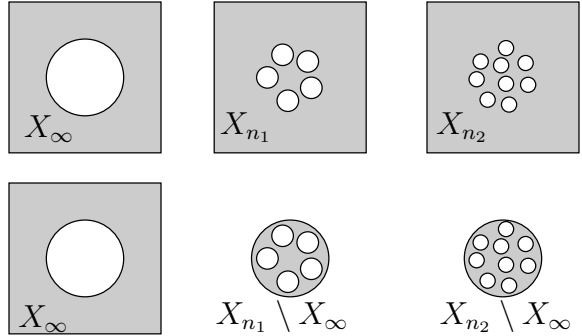
\begin{enumerate}
  \setcounter{enumi}{2}
\item \myparagraph{Dirichlet Laplacians on solidifying obstacles.}  As
  before, let $\D_n$ be (minus) the Laplacian on $X_n$ with Dirichlet
  boundary conditions for $n \in \Nbar$.  Assume here (again for
  simplicity) that $d=3$ (see \Fig{dir-sol} for $d=2$).  Moreover,
  assume that there is a compact subset $S \subset X$ with smooth
  boundary such that $B_n \subset S$ and that $S$ is included in the
  $1/n^\alpha$-neighbourhood of $I_n$.  Finally, assume that there is
  $N \in \N$ such that
  \begin{equation*}
    \card{\set{p \in I_n}
      {B_{1/n^\alpha}(p) \cap B_{1/n^\alpha}(q)\ne \emptyset}}
    \le N
  \end{equation*}
  for $\alpha \in (1/2,1)$.  The limit operator $\D_\infty$ is (minus)
  the Dirichlet Laplacian on $X \setminus S$.  We have shown
  in~\cite[Cor.~6.18]{anne-post:21} that $\D_n \gnrc \D_\infty$ with
  convergence speed $\delta_n \in \Err(1/n^{(1/3-\alpha)/4})$.  We can
  understand the result in the sense that the obstacles $B_n$
  ``solidify'' to the set $S$, leading to the Dirichlet condition ``on
  $S$''.  Unfortunately, the case $\alpha \in (1/3,1/2]$ closer to
  the critical case $\alpha=1/3$ cannot be treated by the results
  of~\cite{anne-post:21}.
\end{enumerate}

Again we conclude from \Thms{main1}{main2b} that the established
QUE-convergence is equivalent with $\D_n\gnrcW \D_\infty$ with the
same convergence speed.  Note that the parent space has to be larger
than the limit space $\Lsqr{X_\infty}$ here as $X_\infty \subset
X_n$. Nevertheless, there are natural candidates for a parent space
such as $\HS=\Lsqr{X}$.

\subsection{No common natural parent space I: Graph-like spaces
  converging to metric graphs, a simplified model}
\label{ssec:met.graphs}
The remaining examples we present here are cases, where the
approximating and limit space are of different nature; namely in the
limit there is a change in dimension.  We begin with the QUE-setting
and then construct the abstract minimal parent space, which is not
naturally given.

The next example was actually the first example for which the concept
of QUE-convergence was established (cf.~\cite{post:06}).

\subsubsection*{A metric graph as the limit space}
Let $X_\infty$ be a (for simplicity) compact metric graph, i.e., a
topological graph where each edge $e \in E$ is assigned a length
$\ell_e \in (0,\infty)$.  Moreover, $X_\infty$ carries a natural
measure: the Lebesgue measure on each line segment.  We decompose the
metric graph into its line segments $X_{\infty,e} \cong [0,\ell_e]$,
i.e., we have
\begin{equation*}
  X_\infty = \bigcup_{e \in E}  X_{\infty,e}
\end{equation*}
(see \Fig{emb-gl-space} left).  We use $x \in [0,\ell_e]$ also as
coordinate of $X_{\infty,e}$, suppressing the formally necessary
isometry $\map{\phi_{\infty,e}}{X_{\infty,e}}{[0,\ell_e]}$ in the
notation.  Moreover, we refer to the coordinate $x$ as the
\emph{longitudinal} direction.

\subsubsection*{A graph-like space converging to a metric graph}
For simplicity, we assume that $X_\infty$ is embedded in $\R^d$ and
that the edges are straight line segments in $X_\infty$.  Let now
$X_n$ be the $1/n$-neighbourhood of $X_\infty$ in $\R^d$ together with
its natural Lebesgue measure (see \Fig{emb-gl-space}).  In particular,
there is a number $a \in (0,\infty)$ such that we can decompose $X_n$
into so-called \emph{shortened edge neighbourhoods} $X_{n,e}$ and
\emph{vertex neighbourhoods} $X_{n,v}$.  Each shortened edge
neighbourhood $X_{n,e}$ is isometric with
$[a/n,\ell_e-a/n] \times B_n$, where
$B_n=\set{y \in \R^{d-1}}{\abs y \le 1/n}$ is called the
\emph{transversal} direction.  Moreover, the ``area'' $\abs {B_n}$ of
$B_n$ is $\abs{B_n}=\abs{B_1}n^{-(d-1)}$.  Each vertex neighbourhood
$X_{n,v}$ is $(1/n)$-homothetic with a building block $X_{1,v}$.

\subsubsection*{The ``bumpy'' graph-like space as simplification}
In order to simplify the arguments, we consider the \emph{(full) edge
  neighbourhood} $\check X_{n,e}$ isometric with $[0,\ell_e] \times B_n$
instead of the shortened edge neighbourhood $X_{n,e}$.  We consider
the resulting space $\check X_n$ as being glued together from the
so-called \emph{vertex neighbourhoods} $X_{n,v}$ ($v \in V$) and the
full edge neighbourhoods $\check X_{n,e}$ ($e \in E$) in such a way that
\begin{equation*}
  \check X_n = \bigcup_{v \in V} X_{n,v}
  \cup \bigcup_{e \in E}  \check X_{n,e}
\end{equation*}
(see \Fig{bumpy-gl-space}).  Moreover, we assume that
$X_{n,v}\cap \check X_{n,e} =\emptyset$ if $e$ and $v$ are not
incident and that $X_{n,v} \cap \check X_{n,e}$ is $(d-1)$-dimensional
and isometric with $B_n$ if $e$ and $v$ are incident.
\begin{figure}[h]
  \centering
  \begin{picture}(0,0)
    \includegraphics[scale=0.8]{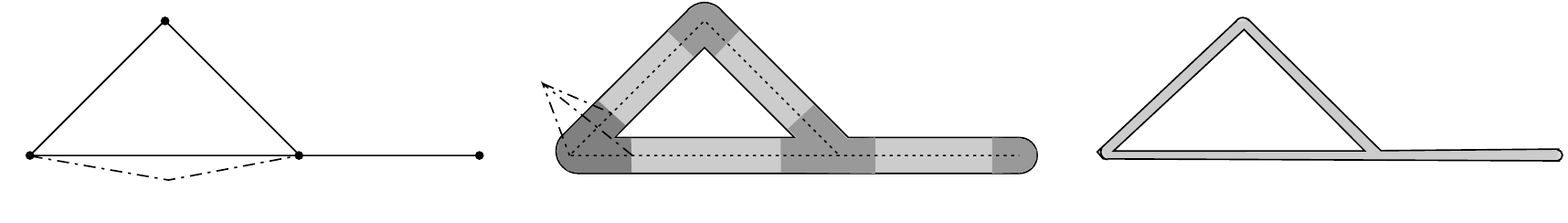}%
  \end{picture}%
  \setlength{\unitlength}{3315sp}
  \begin{picture}(7827,1108)(301,-395)
    \put(296,-381){$X_\infty$}
    \put(316, 24){$v$}
    \put(991, 24){$X_{\infty,e}$}
    \put(2881,400){$a/n$}
    \put(2581,104){$X_{n,v}$}
    \put(3691,-381){$X_{n,e}$}
    \put(991,-381){$\ell_e$}
    \put(2781,-381){$X_{n}$}
    \put(5851,-381){$X_{n'}$}
  \end{picture}%
  \caption{\emph{Left:} A metric graph $X_\infty$ with four edges and
    four vertices.  The edge $e$ is identified with the interval
    $X_{\infty,e}=[0,\ell_e]$ (fat). \emph{Middle:}
    The embedded graph-like space as subset of $\R^2$ ($d=2$),   the
    edge neighbourhood $X_{n,e}$ (light grey) is isometric with
    $[0,\ell_e-2a/n] \times [-1/n,1/n]$, and the vertex
    neighbourhood $X_{n,v}$ (dark grey) is $(1/n)$-homothetic with
    a fixed set  $X_{1,v}$.  \emph{Right:} An embedded graph-like
    space $X_{n'}$ for $n'>n$.}
   \label{fig:emb-gl-space}
\end{figure}
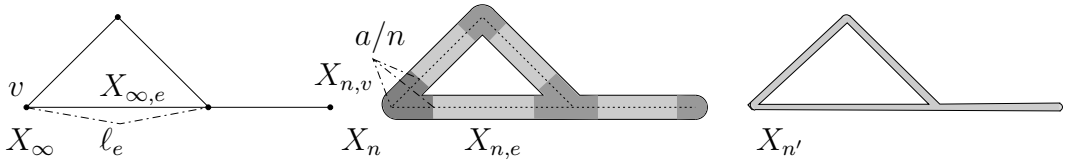

\begin{figure}[h]
  \centering
  \begin{picture}(0,0)%
    \includegraphics{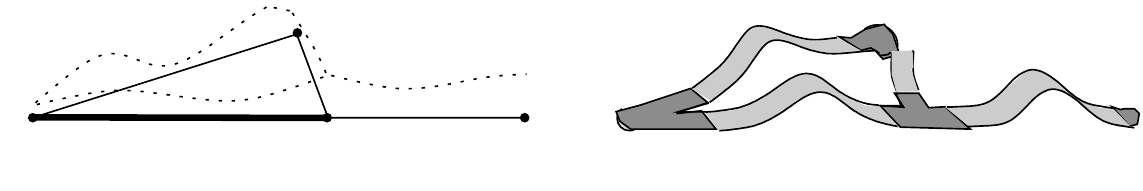}%
  \end{picture}%
  \setlength{\unitlength}{4144sp}%
  \begin{picture}(5221,859)(301,-2240)
    \put(316,-1771){\makebox(0,0)[lb]{$v$}}
    \put(991,-2261){\makebox(0,0)[lb]{$X_{\infty,e}$}}
    \put(296,-2276){\makebox(0,0)[lb]{$X_\infty$}}
    \put(2881,-1781){\makebox(0,0)[lb]{$X_{n,v}$}}
    \put(3691,-2276){\makebox(0,0)[lb]{$\check X_{n,e}$}}
    \put(2961,-2276){\makebox(0,0)[lb]{$\check X_n$}}
  \end{picture}%
  \caption{\emph{Left:} A metric graph $X_\infty$, the dotted lines
    represent a function $f_\infty$ on $X_\infty$.  \emph{Right:} A
    bumpy graph-like space.  Here, the edge neighbourhood $\check
    X_{n,e}$ is isometric with $[0,\ell_e] \times [-1/n,1/n]$ (for
    $d=2$), and the vertex neighbourhood $X_{n,v}$ is the same as in
    the embedded case..}
  \label{fig:bumpy-gl-space}
\end{figure}
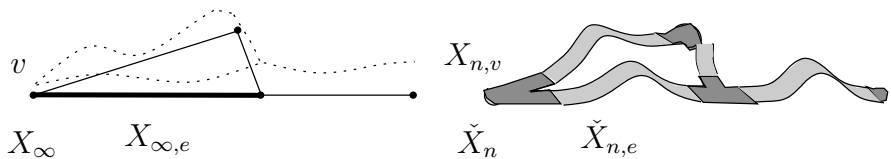
If the original graph $X_\infty$ is embedded in $\R^2$, one might
think of $\check X_n$ as a paper model, where the edge neighbourhoods
have a ``bump'' in direction of a third dimension, in order to have
space for the extra longitudinal length (see \Fig{bumpy-gl-space}).
Therefore, we call $\check X_n$ the \emph{bumpy graph-like space}.  We
will treat the original space $X_n$ with the slightly shortened edge
neighbourhoods $X_{n,e}$ as a perturbation of the bumpy graph-like
space $\check X_n$ in \Subsec{met.graphs2}, see
also~\cite[Prp.~5.3.7]{post:12}.

As the components of the decomposition are disjoint up to sets of ($d$-dimensional)
measure $0$, we have
\begin{align*}
  \check \HS_n
  := \Lsqr{\check X_n}
  =\bigoplus_{v \in V} \Lsqr{X_{n,v}} \oplus \bigoplus_{e \in E} \Lsqr{\check X_{n,e}}
  \quadtext{and}
  \HS_\infty
  := \Lsqr{X_\infty}
  =\bigoplus_{e \in E} \Lsqr{X_{\infty,e}}.
\end{align*}
According to this decomposition, we write
\begin{equation*}
  f_{n,v} := f_n \restr{X_{n,v}}, \quad
  f_{n,e} := f_n \restr{\check X_{n,e}}  \quadtext{and}
  f_{\infty,e} := f_\infty \restr{X_{\infty,e}}.
\end{equation*}
for $f_n \in \check \HS_n$ and $f_\infty \in \HS_\infty$ and use the notation
$(\cdot)_{n,v}$, $(\cdot)_{n,e}$ and $(\cdot)_{\infty,e}$ as the
corresponding restriction operators, respectively.

\subsubsection*{The QUE-setting and the identification operators}
We now define the identification operators $\check J_n$ as follows:
\begin{align*}
  \map{\check J_n}{\check \HS_n=\Lsqr{\check X_n}}
  {\HS_\infty=\Lsqr{X_\infty}}, \qquad 
  (\check J_n f_n)_{\infty,e}(x)=\abs{B_n}^{-1/2} \int_{B_n}
  f_{n,e}(x,y)\dd y.
\end{align*}
Note that $(\check J_n f_n)_{\infty,e}$ is (up to a scaling) the
transversal average of $f_{n,e}$.  A straightforward calculation gives
the adjoint $\check J_n^* g_\infty$ on the components as
\begin{equation*}
  (\check J_n^*g_\infty)_{n,v}=0 \qquadtext{and}
  (\check J_n^*g_\infty)_{n,e}(x,y)
  =  \abs{B_n}^{-1/2} g_\infty(x).
\end{equation*}
Moreover, it is easily seen that $\check J_n^*$ is an isometry.

It was shown in~\cite{post:06} (see also~\cite{post:12}) that (minus)
the Neumann Laplacian on $\check X_n$ denoted here by $\check \D_n$
QUE-converges to (minus) the standard Laplacian $\D_\infty$ on
$X_\infty$, i.e., $\check \D_n \gnrc \D_\infty$ with convergence speed
$\delta_n \in \Err(n^{-1/2})$.  The standard Laplacian in $X_\infty$
is the operator acting as $(\D f_\infty)_e=-f_e''$ with functions
continuous at a vertex and $\sum_{e \sim v} f_e'(v)=0$, where
$f_e'(v)$ denotes the derivative along $X_{\infty,e}$ towards the
point in $X_\infty$ corresponding to $v$.
\subsubsection*{Weidman's setting: the corresponding defect operators,
  the minimal parent space and its isometries}
As $\check J_n^*$ is an isometry, its defect operator
fulfils
$\check W_{\infty,n}=(\id_{\HS_\infty}-\check J_n\check
J_n^*)^{1/2}=0$.
It follows that
$\check W_n=(\id_{\check \HS_n}-\check J_n^*\check J_n)^{1/2}$ is an
orthogonal projection onto $\ker \check J_n$ (see
\Thmenum{j.proj.com}{j.proj.com.b}); hence we can leave out the square
root.  In particular,
\begin{equation}
  \label{eq:gl-space.wn}
  \check W_n
  = \id_{\check \HS_n}-\check J_n^*\check J_n
  =\bigoplus_{v \in V} \id_{\Lsqr{X_{v,n}}} \oplus
  \bigoplus_{e \in E} \check P_{n,e}^\perp,
\end{equation}
where
\begin{align*}
  (\check P_{n,e}^\perp f_{n,e}(x,y)
  = f_{n,e}(x,y)- \frac 1{\abs{B_n}} \int_{B_n} f_{n,e}(x,y') \dd y'
\end{align*}
is the projection onto the orthogonal complement of
$\Lsqr{[0,\ell_e]}\otimes \C \1_{B_n}$ in
$\Lsqr{\check X_{n,e}} \cong \Lsqr{[0,\ell_e]}\otimes \Lsqr{B_n}$;
i.e., the projection onto the space of functions with transversal
average $0$.  Moreover,
\begin{equation*}
  \check W_n(\check \HS_n)
  = \ker \check J_n
  \cong \bigoplus_{v \in V} \Lsqr{X_{n,v}} \oplus
  \bigoplus_{e \in E} \Lsqr{[0,\ell_e]}
  \otimes (\Lsqr{B_n} \ominus \C \1_{B_n})
\end{equation*}
is a closed subspace of $\check \HS_n$.  In particular, we have
\begin{equation*}
  \check \HSmin = \Lsqr{X_\infty}
   \oplus \bigoplus_{n \in \N}
   \Bigl(\bigoplus_{v \in V} \Lsqr{X_{n,v}} \oplus
   \bigoplus_{e \in E} \Lsqr{[0,\ell_e]}
   \otimes (\Lsqr{B_n} \ominus \C \1_{B_n})
  \Bigr)
\end{equation*}
\begin{figure}[h]
  \centering
  \begin{picture}(0,0)%
    \includegraphics[scale=0.8]{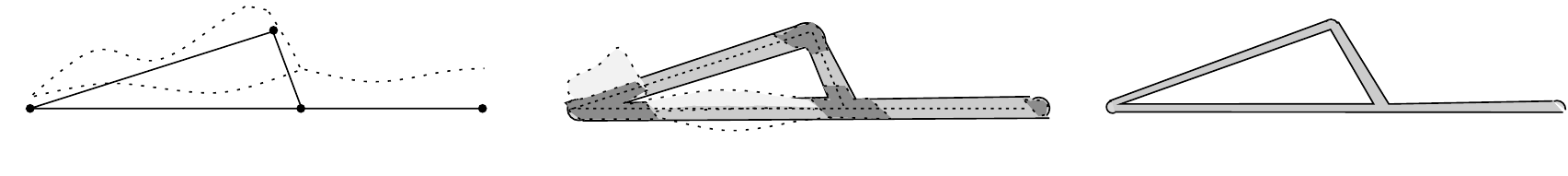}%
  \end{picture}%
  \setlength{\unitlength}{3315sp}%
  \begin{picture}(7792,859)(301,-2240)
    \put(316,-1771){$v$}%
    \put(991,-2276){$f_{\infty,e}$}%
    \put(496,-2276){$X_\infty$}%
    \put(2881,-1681){$f_{n,v}$}%
    \put(3691,-2276){$f_{n,e}$}%
    \put(3061,-2276){$\check X_n$}%
    \put(5851,-2276){$\check X_{n'}$}%
  \end{picture}%
  \caption{Three components of the minimal parent space (for
    simplicity not drawn as bumpy space): left the space
    $\Lsqr{X_\infty}$, in the middle the $n$-th component (only the
    functions on the vertex neighbourhoods $X_{n,v}$ are needed and
    the functions on $\check X_{n,e}$ with transversal average $0$ on
    $\check X_{n,e}$).  On the right, the underlying space of the
    $n'$-th component for some $n'>n$.}
  \label{fig:gl-min.ps}
\end{figure}
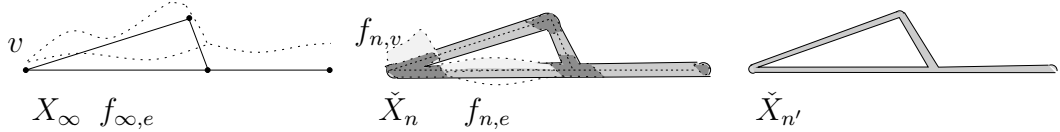
as abstract minimal parent space (see \Lem{hs.min}) with isometries
$\map{\check \iota_\infty}{\Lsqr{X_\infty}}{\check\HSmin}$ given by
$\check \iota_\infty f_\infty=(f_\infty,0,\dots)$ and $\map{\check
  \iota_n}{\Lsqr{\check X_n}}{\check \HSmin}$ acting as
\begin{equation*}
  \check \iota_n f_n
  =\bigr(\check J_n f_n,0,\dots,0,
  (f_{n,v})_{v \in V} \oplus (\check P_{n,e}^\perp f_{n,e})_{e \in E)},0,
  \dots \bigr)
\end{equation*}
for $f_n \in \check \HS_n$ (see \Fig{gl-min.ps}).  Note that $\check
\iota_n f_n$ here is \emph{not} positivity-preserving, as for any
function $f_{n,e}(x,y)=g_e(x) h_n(y)$ we have
\begin{align*}
  P_{n,e}^\perp f_{n,e}(x,y)=g_e(x)\Bigl(h_n(y)-
  \frac 1{\abs{B_n}} \int_{B_n} h_n(y')\dd y'\Bigr).
\end{align*}
If $g_e \ge 0$ and if $h_n \ge 0$ is not constant, then $f_{n,e}\ge
0$, but $\check P_{n,e}^\perp f_{n,e}$ must change sign for different
$y$ (and $x$ with $g_e(x)>0$).  Nevertheless, note that $\check
J_n=\check \iota_\infty^*\check \iota_n$ is positivity preserving.

The projections
$\check P_\infty=\check \iota_\infty \check \iota_\infty^*$ and
$\check P_n=\check \iota_n \check \iota_n^*$ are given by
\begin{equation*}
  \check P_\infty f=(f_\infty,0,\dots)
  \qquadtext{and}
  \check P_n f
  =\bigr(f_\infty,0,\dots,0,
  (f_{n,v})_{v \in V} \oplus (\check P_{n,e}^\perp f_{n,e})_{e \in E)},0,
  \dots \bigr)
\end{equation*}
for $f=(f_\infty,f_1,f_2,\dots) \in \HSgen$.

As the projections commute ($J_n^*$ is an isometry), the convergence
speed is preserved: denote by $\delta_n$ the convergence speed of the
QUE-convergence $\D_n \gnrc \D_\infty$ shown e.g.\ in
in~\cite{post:06,post:12}.  We then conclude from
\Cor{proj.comm.jn.co-iso}:
\begin{corollary}[Weidmann's convergence for graph-like spaces]
  We have $\check \D_n \gnrcW \D_\infty$ with convergence speed
  $2\delta_n \in \Err(n^{-1/2})$.
\end{corollary}

\subsection{No common natural parent space II: Graph-like spaces
  converging to metric graphs, the original model}
\label{ssec:met.graphs2}

Let us now show an example where we need to rescale the identification
operators.  We keep the notation from the previous subsection.
\subsubsection*{The original (embedded) graph-like space}
Let us now consider the original graph-like space $X_n$ with the
slightly shortened edge neighbourhoods $X_{n,e}$ isometric with
$[a/n,\ell_e-a/n]\times B_n$ (see \Fig{emb-gl-space}).  We only
present the differences with the abstract model in
\Subsec{met.graphs}: we set $\HS_n:=\Lsqr{X_n}$ and
\begin{gather*}
  \map{J_n}{\HS_n=\Lsqr{X_n}}{\HS_\infty=\Lsqr{X_\infty}}, \qquad
  (J_n f_n)_{\infty,e}(x)
  =\abs{B_n}^{-1/2} \int_{B_n} f_{n,e}(\Phi_{n,e}(x),y)\dd y,
\end{gather*}
where
\begin{equation*}
  \map{\Phi_{n,e}}{[0,\ell_e]}{\Bigl[\frac an,\ell_e-\frac an\Bigr]},
  \qquad
  \Phi_{n,e}(x)=\frac an +\Bigl(1-\frac {2a}{n\ell_e}\Bigr) x.
\end{equation*}
Again, an easy computation shows that
\begin{equation*}
  (J_n^*g_\infty)_{n,v}=0 \qquadtext{and}
  (J_n^*g_\infty)_{n,e}(\hat x,y)
  =  \abs{B_n}^{-1/2} \Bigl(1-\frac {2a}{n\ell_e}\Bigr)^{-1}
  g_\infty(\Phi_{n,e}^{-1}(\hat x)).
\end{equation*}
Moreover, we have here
\begin{align*}
  \normsqr[\HS_n]{J_n^* g_\infty}
  &= \frac1{\abs{B_n}} \sum_{e \in E}
    \Bigl(1-\frac {2a}{n\ell_e}\Bigr)^{-2}
  \int_{X_{n,e}} \abssqr{g_{\infty,e}(\Phi_{n,e}^{-1}(\hat x))}
  \dd y \dd \hat x\\
  &= \sum_{e \in E} \Bigl(1-\frac {2a}{n\ell_e}\Bigr)^{-1}
  \int_0^{\ell_e}  \abssqr{g_{\infty,e}(x)} \dd x
  =\normsqr[\HS_\infty]{K_{\infty,n} g_\infty},
\end{align*}
where $K_{\infty,n}$ is the operator multiplying with
\begin{equation*}
  \kappa_{n,e}
  = \Bigl(1-\frac {2a}{n\ell_e}\Bigr)^{-1/2}>1.
\end{equation*}
In particular,
$\norm{J_n}=\norm{J_n^*}
=\norm{K_{\infty,n}}=(1-2a/(n\ell_0))^{-1/2}>1$,
where $\ell_0$ is the maximum of all edge lengths $\ell_e$.

\subsubsection*{QUE- convergence in the embedded case: a
  change of identification operators}
It was again shown in~\cite{post:06} (see also~\cite{post:12}) that
$\D_n$ ((minus) the Neumann Laplacian on $X_n$)
QUE-converges to (minus) the standard Laplacian $\D_\infty$ on the
metric graph $X_\infty$, i.e., that $\D_n \gnrc \D_\infty$ with speed
$\delta_n \in \Err(n^{-1/2})$.

As here $\norm{J_n}>1$, we have to change the identification
operator as in \Subsec{change.id} to
\begin{equation*}
  \hat J_n
  := \frac 1{\norm{J_n}} J_n
  = \Bigl(1-\frac {2a}{n\ell_0}\Bigr)^{1/2} J_n.
\end{equation*}
From \Cor{eq-que} we conclude that $\D_n \gnrc \D_\infty$ with
identification operators $(\hat J_n)_n$ and convergence speed still of
order in $\Err(n^{-1/2})$.

\subsubsection*{Weidmann's convergence  in the embedded case: a
  change of identification operators}
Let us finally comment on Weidmann's convergence: we have
\begin{equation*}
  (\hat J_n \hat J_n^* f_\infty)_e(x)
  =  \Bigl(1-\frac {2a}{n\ell_0}\Bigr)
  \Bigl(1-\frac {2a}{n\ell_e}\Bigr)^{-1}
  f_{\infty,e}(x),
\end{equation*}
hence $J_n^*$ is only an isometry if $\ell_e=\ell_0$ for all $e \in E$
(i.e., if $X_\infty$ is an \emph{equilateral} metric graph).  In
particular, the defect operator associated with $\hat J_n^*$ is given
by
\begin{equation*}
  (\hat W_{\infty,n} f_\infty)_e
  =  \hat w_{n,e} \cdot f_{\infty,e},
  \quadtext{where}
  \hat w_{n,e}
  :=\Big(1-\frac{1-2a/(n\ell_0)}{1-2a/(n\ell_e)} \Big)^{1/2}
  \in \Err\Bigl(\frac1{n^{1/2}}\Bigr),
\end{equation*}
and we have $\norm[\Lin{\HS_\infty}]{\hat W_{\infty,n}} \to 0$ as $n
\to \infty$.  For the defect operator associated with $\hat J_n$, we
have
\begin{equation*}
  \hat W_n = (\id_{\HS_n}-\hat J_n^*\hat J_n)^{1/2}
  =\bigoplus_{v \in V} \id_{\Lsqr{X_{n,v}}} \oplus
  \bigoplus_{e \in E}
  (\hat w_{n,e}
  P_{n,e} \oplus P_{n,e}^\perp)\Bigr),
\end{equation*}
where $P_{n,e}$ is the orthogonal projection onto
$\Lsqr{[a/n,\ell_e-a/n]} \otimes \C \1_{B_n} \subset \Lsqr{X_{n,e}}$
and $P_{n,e}^\perp$ its complement.  The abstract minimal parent space
as constructed in \Subsec{concrete.iso} is given by
\begin{equation*}
  \hat \HSmin = \Lsqr{X_\infty}
   \oplus \bigoplus_{n \in \N}
   \Bigl(\bigoplus_{v \in V} \Lsqr{X_{n,v}} \oplus
   \bigoplus_{e \in E} \Lsqr{X_{n,e}}
  \Bigr)
\end{equation*}
with isometries
$\map{\hat \iota_\infty}{\Lsqr{X_\infty}}{\hat\HSmin}$ given by
$\hat \iota_\infty f_\infty=(f_\infty,0,\dots)$ and
$\map{\hat \iota_n}{\Lsqr{\hat X_n}}{\hat \HSmin}$ acting as
\begin{equation*}
  \hat \iota_n f_n
  =\Bigr(\hat J_n f_n,0,\dots,0,
  (f_{n,v})_{v \in V}
  \oplus \bigl((\hat w_{n,e}P_{n,e} \oplus \hat P_{n,e}^\perp) f_{n,e}
  \bigr)_{e \in E},0,
  \dots \Bigr)
\end{equation*}
for $f_n \in \hat \HS_n$.  Again, $\hat \iota_n f_n$ is not
positivity-preserving (but $\hat J_n$ is).  The corresponding
orthogonal projections $\hat P_n$ and $\hat P_\infty$ of the abstract
parent space fulfil
\begin{align*}
  (\hat P_n\hat P_\infty-\hat P_\infty \hat P_n)f
  &=\bigl(
    -\hat W_{\infty,n} \hat J_n f_n,
    0, \dots, 0,
    \hat J_n^* \hat W_{\infty,n} f_\infty,
    0, \dots
    \bigr),
\end{align*}
i.e., they do not commute.  In particular, we conclude from
\Thm{main2} that the convergence speed in Weidmann's convergence will
be slower:
\begin{corollary}[Weidmann's convergence for embedded graph-like
  spaces]
  We have $\check \D_n \gnrcW \D_\infty$ with convergence speed
  of order $\sqrt{n^{-1/4}}$.
\end{corollary}

\begin{remarks}
  \label{rem:change.of.que-def.ex}
  \indent
  \begin{enumerate}
  \item One could turn $\hat J_n^*$ into an isometry also for
    non-equilateral graphs by using the identification operator $\wt
    J_n := K_{\infty,n}^{-1/2} J_n = J_n K_{\infty,n}^{-1/2}$ and
    hence get the same convergence speed, but we wanted to illustrate
    the effect of non-commuting projections here.

  \item Note that we have $\norm[\Lin \HSgen]{\hat P_n\hat
    P_\infty-\hat P_\infty \hat P_n} \to 0$ as $\norm{\hat
    W_{\infty,n}} \to 0$ in operator norm and as $\norm{\hat J_n}=1$,
    cf.\ also \Cor{proj.conv.strongly}, but the operator norm
    $\norm[\Lin \HSgen]{\hat P_n - \hat P_\infty}$ of the projection
    difference does not converge to $0$.

  \item As already mentioned in \Rem{change.of.que-def}, it seems to
    be better to use~\eqref{eq:que2''} instead of~\eqref{eq:que2}
    (with $J_n$ replaced by $\hat J_n$).  In this concrete example,
    one can show 
    that~\eqref{eq:que2''} holds with $\delta_n \in \Err(1/n^{1/2})$,
    hence following the argument in \Rem{change.of.que-def} we would
    still end up with $\D_n \gnrcW \D_\infty$, but now with
    convergence speed of order $\Err(1/n^{1/2})$.
  \item
  \label{change.of.que-def.ex.d}
    Actually, when proving the QUE-convergence
    in~\cite{post:06,post:12}, we have shown that
    \begin{equation*}
      \norm{(\id_{\HS_n}-\hat J_n^*\hat J_n)R_n^{1/2}}=\Err(1/n^{1/2}).
    \end{equation*}
    It is not clear to us if
    $\norm{(\id_{\HS_n}-\hat J_n^*\hat J_n)R_n}$ would give a better
    estimate here.  In the next example (see \Subsec{disc.graphs}),
    using $R_n$ instead of $R_n^{1/2}$ in the operator norm estimate
    gives a better result, see \Rem{res.power.makes.difference}.  It
    could be helpful to analyse the resolvent difference
    $\hat D_n=\hat\iota_n \hat R_n \hat\iota_n^*-\hat\iota_\infty \hat
    R_\infty \hat\iota_\infty^*$ (see~\eqref{eq:dn.direct} for a
    formula) directly in this example.
  \end{enumerate}
\end{remarks}

\subsection{No common natural parent space III: Discrete graphs converging to metric measure spaces}
\label{ssec:disc.graphs}
The second example class in which the approximating and limit space
are of different nature is taken
from~\cite{post-simmer:18,post-simmer:21b}.  Again, we will not go
into much details, as here, only the measure spaces $X_n$ and
$X_\infty$ matter.  Assume that $X_\infty$ is a measure space with
(for simplicity) finite measure $\mu_\infty$; e.g.\ a post-critically
finite fractal such as the Sierpi\'nski gasket.  We treat here a very
simple example of a post-critically finite ``fractal'': the unit
interval $X_\infty=[0,1]$ with Lebesgue measure $\mu_\infty$ together
with its usual (Neumann) Laplacian on $[0,1]$,
see~\cite[Sec.~5.3]{post-simmer:18}
or~[Sec.~2.1]\cite{post-simmer:19}.

  We set
$\HS_\infty=\Lsqr{X_\infty,\mu_\infty}$.
The limit operator $\D_\infty$ is
(minus) the Neumann Laplacian on $X_\infty=[0,1]$, i.e., we have
\begin{equation*}
  \D_\infty f_\infty=-f_\infty''
  \quadtext{with}
  f_\infty \in \Sob[2]{[0,1]}, \quad
  f_\infty'(0)=0, f_\infty'(1)=0.
\end{equation*}

\subsubsection*{The discrete approximation space}
In our unit interval example we set
\begin{equation*}
  X_n = \set{v2^{-n}}{v\in\{0,\dots,2^n\}};
\end{equation*}
in particular, $X_0=\{0,1\} \subset X_n \subset X_\infty =[0,1]$.  We
call $X_0 \subset X_\infty$ the \emph{boundary} of $X_\infty$.
The approximating operator $\D_n$ is
\begin{equation*}
  \D_n f_n = 2 \cdot 4^n \Delta_{G_n},
\end{equation*}
where $\Delta_{G_n}$ is the standard discrete Laplacian on the path
graph $G_n$ with $2^n+1$ vertices, i.e., $\Delta_{G_n}$ acts as
\begin{equation*}
  (\Delta_{G_n} f_n)(v)
  =
  \begin{cases}
    f(v)-\frac12\bigl(f(v_-)+f(v_+)\bigr), & v \in X_n \setminus X_0,\\
    f(v)-f(v_+), & v=0 \in X_0,\\
    f(v)-f(v_-), & v=1 \in X_0.
  \end{cases}
\end{equation*}
where $v_-$ is the left and $v_+$ the right neighbour of $v \in X_n$.

To each vertex $v \in X_n$ we associate a measurable function
$\map{\psi_{n,v}}{X_\infty}[0,1]$ such that $(\psi_{n,v})_{v \in X_n}$
is a partition of unity, i.e., such that
\begin{equation}
  \label{eq:part.uni}
  \sum_{v \in X_n} \psi_{n,v}(x)=1
  \qquad\text{holds for ($\mu_\infty$-almost) all $x \in X_\infty$.}
\end{equation}
\begin{figure}[h]
  \centering
  \begin{picture}(0,0)%
    \includegraphics{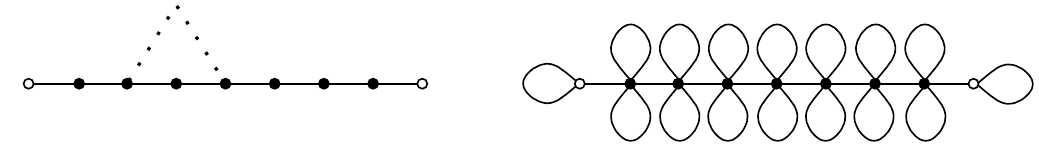}%
  \end{picture}%
  \setlength{\unitlength}{4144sp}%
  \begin{picture}(4734,730)(320,86)
    \put(2581,205){$G_n^\circ$}%
    \put(1335, 605){$\psi_{n,v}$}%
    \put(1100,205){$v \in X_n$}%
    \put(335,605){$X_\infty$}%
    \put(335,205){$G_n$}%
  \end{picture}%
  \caption{\emph{Left:} The space $X_\infty=[0,1]$ with its
    approximating graph $G_n$ on the vertex set $X_n$ (here $n=3$),
    the boundary vertices $0$ and $1$ are outlined; the dotted line
    represents the function $\psi_{n,v}$ for $v=3$.  \emph{Right:} The
    graph $G_n^\circ$ with the added loops.}
  \label{fig:unit-int}
\end{figure}
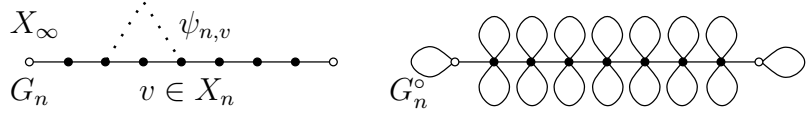
We choose $\psi_{n,v}$ to be piecewise affine linear (harmonic) with
$\psi_{n,v}(x)=1$ for $x=v$ and $\psi_{n,v}(x)=0$ for all other
vertices $x \in X_n \setminus \{v\}$ (see \Fig{unit-int} left).  Such
functions are continuous and hence belong automatically to the
corresponding form domain $\Sob {[0,1]}$ of the Laplacian.  For more
complicated fractal spaces $X_\infty$
(cf.\ e.g.~\cite{post-simmer:18,post-simmer:21b}) we have chosen
adopted versions of so-called (piecewise) ``harmonic'' functions
$\psi_{n,v}$ having value $1$ at the vertex $v$ (viewed as a point in
$X_\infty$) and $0$ at all other vertices in $V \setminus \{v\}$.

We define $\mu_n(v):=\int_{X_\infty} \psi_{n,v} \dd \mu_\infty$ and
let $\HS_n=\lsqr{X_n,\mu_n}$ be the corresponding weighted discrete
space with norm given by
$\normsqr[\lsqr{X_n,\mu_n}] {f_n}=\sum_{v \in X_n}\abssqr{f_n(v)}
\mu_n(v)$.  In our unit interval example we have
\begin{equation*}
  \mu_n(v)= \int_0^1 \psi_{n,v}(x) \dd x
  =
  \begin{cases}
    \int_0^{2^{-n}} s/2^{-n} \dd s=2^{-n}/2, & v \in X_0,\\
    2\int_0^{2^{-n}} s/2^{-n} \dd s=2^{-n}, & v \in X_n \setminus X_0,
  \end{cases}
\end{equation*}
for a boundary resp.\ inner vertex.  We may think of $(X_n,\mu_n)$ as
a \emph{discretisation} of the space $(X_\infty,\mu_\infty)$, and
typically, $X_n$ is a subset of $X_\infty$ converging in some
sense to $X_\infty$.

\subsubsection*{The identification operators and the QUE-convergence}
We now define the identification operator
$\map{J_n}{\HS_n}{\HS_\infty}$ by
\begin{equation}
  \label{eq:def.j.graph}
  J_n f_n := \sum_{v \in X_n} f_n(v) \psi_{n,v}.
\end{equation}
An easy computation shows that
$\normsqr[\Lsqr{X_\infty,\mu_\infty}]{J_n f_n} \le
\normsqr[\lsqr{X_n,\mu_n}] {f_n}$.
using
\begin{equation*}
  \sum_{w \in X_n} \iprod[\HS_\infty]{\psi_{n,v}}{\psi_{n,w}}
  =\int_{X_n} \psi_{n,v} \dd \mu_\infty=\mu_n(v)
\end{equation*}
by~\eqref{eq:part.uni} (this justifies the definition of the discrete
weight $\mu_n$).  In particular, $J_n$ is a contraction.  The adjoint
$\map{J_n^*}{\HS_\infty}{\HS_n}$ is given by
\begin{equation*}
  (J_n^* f_\infty)(w)
  := \frac1{\mu_n(w)} \iprod[\HS_\infty]{f_\infty}{\psi_{n,w}}.
\end{equation*}

Since $\D_n \ge 0$ for all $n \in \Nbar$, we choose $z_0=-1$ as common
resolvent element.  In particular, we have shown
$\D_n \gnrc \D_\infty$ with convergence speed
$\delta_n=((1+\sqrt 2))2^{-n} \in \Err(2^{-n})$
(cf.~\cite[Sec.~5,3]{post-simmer:18}); for the optimality of this
estimate see \Rem{opt.unit.int}.

If $X_\infty$ is the Sierpi\'nski gasket and $X_n$ an appropriate
discrete weighted graph, then the corresponding scaled discrete
Laplacian $\D_n$ QUE-converges towards the Laplacian $\D_\infty$ on
the Sierpi\'nski gasket with convergence speed
$\delta_n \in \Err(5^{-n/2})$, for more examples and details we
refer to ~\cite{post-simmer:18,post-simmer:21b,post-simmer:19}.

\begin{remark}[an orthogonal partition of unity]
  Another choice of a partition of unity $(\psi_{n,v})_{v \in X_n}$ is
  to \emph{discretise} the measure space $(X_\infty,\mu_\infty)$ into
  a finite number of elements $(X_{\infty,n,v})_{v \in X_n}$, each
  with measure $\mu_n(v)=\mu_\infty(X_{\infty,n,v})>0$, and we set
  $\psi_{n,v}:=\1_{X_{\infty,n,v}}$.  Then $(\psi_{n,v})_{v \in X_n}$
  is an orthonormal family, and $J_n$ is an isometry, and we are in
  the setting of \Cor{proj.comm.jn.iso}.  Nevertheless, the
  QUE-convergence $A_n$ towards $A_\infty$ has not been shown in this
  setting.
\end{remark}

\subsubsection*{The corresponding defect operator and a relation with
  a discrete Laplacian}
For the defect operator associated with $J_n$ we calculate
\begin{align}
  \nonumber
  ((\id_{\HS_n}-J_n^*J_n)f_n)(w)
  &= f(w)-\frac1{\mu_n(w)}\sum_{v \in X_n} f(v)
    \iprod[\HS_\infty]{\psi_{n,v}}{\psi_{n,w}}\\
  \label{eq:wn.graph}
  &= \frac1{\mu_n(w)}\sum_{v \in X_n} \bigl(f(w)-f(v)\bigr)
  \iprod[\HS_\infty]{\psi_{n,v}}{\psi_{n,w}}.
\end{align}
We can interpret $\id_{\HS_n}-J_n^*J_n$ as the Laplacian of a discrete
weighted graph $G_n^\circ=(X_n,\mu_n,\gamma_n)$ (the notation
$G_n^\circ$ becomes clear in a moment (see \Fig{unit-int} right),
where $\map{\mu_n}{X_n}{(0,\infty)}$ is a \emph{vertex weight} and
$\map{\gamma_n}{X_n \times X_n}{(0,\infty)}$ with
\begin{equation*}
  \gamma_n(v,w)
  =\gamma_n(w,v)
  =\iprod[\HS_\infty]{\psi_{n,v}}{\psi_{n,w}}
\end{equation*}
is an \emph{edge weight}.  The associated graph $G_n^\circ$ has an
edge between two vertices $v$ and $w$ if $\gamma_n(v,w)>0$.

Since
\begin{equation*}
  \sum_{w \in X_n} \gamma_n(v,w) = \mu_n(v),
\end{equation*}
the weight is \emph{normalised} (in the terminology
of~\cite[Sec.~2.2]{fclp:18}).  For such graphs, the corresponding
Laplacian $\Delta_{G_n^\circ}$ acts as in~\eqref{eq:wn.graph}.   Since we
\emph{always} have $\gamma_n(v,v)>0$, the graph has \emph{loops} at
each vertex $v \in X_n$ (a \emph{look} is an edge starting and ending
at the same vertex).  We use the symbol $(\cdot)^\circ$ to indicate
the existence of loops at each vertex.

In the unit interval example we have
\begin{align*}
  \gamma_n(v,v)
  &=
  \begin{cases}
    \int_0^{2^{-n}} (s/2^{-n})^2 \dd s=2^{-n}/3, & v \in X_0,\\
    2\int_0^{2^{-n}} (s/2^{-n})^2 \dd s=2^{-n}\cdot(2/3),
    & v \in X_n \setminus X_0
    \quad\text{and}
  \end{cases}\\
  \gamma_n(v,w) &= \int_0^{2^{-n}} s/2^{-n} \cdot (1-s/2^{-n}) \dd s
                  =2^{-n}/6 \quadtext{if} \abs{v-w}=2^{-n}.
\end{align*}
The corresponding matrix is the one associated with the standard
Laplacian being a path graph with $2^n+1$ vertices, one loop attached
at each of the two boundary vertices $v \in \{0,1\}$, and two loops
attached at each inner vertex $v \in X_n \setminus \{0,1\}$ (see
\Fig{unit-int} left).  One can see that this standard Laplacian is
$1/3$ times the standard Laplacian $\Delta_{G_n}$ of the associated
\emph{simple} path graph $G_n$ with $2^n+1$ vertices (and no loops).
The spectrum of $\Delta_{G_n}$ is given by the $2^n+1$ numbers
$1-\cos(k\pi/2^n)$ ($k=0,\dots,2^n$), i.e., the spectrum of
$\Delta_{G_n^\circ}$ is
\begin{align}
  \nonumber
  \spec{\Delta_{G_n^\circ}}
  &= \frac 13 \spec{\Delta_{G_n}}\\
  \label{eq:delta.circ.spec}
  &= \Bigl\{
  0=\frac13-\frac13\cos\Bigl(\frac{0\pi}{2^n}\Bigr),
  \dots,
  \frac13-\frac13\cos\Bigl(\frac{k\pi}{2^n}\Bigr), \dots,
  \frac23=\frac13-\frac13\cos\Bigl(\frac{2^n\pi}{2^n}\Bigr)
  \Bigr\}.
\end{align}
In our unit interval example, the matrix representation of
$\Delta_{G_n}^\circ$ has $1-\gamma_n(v,v)/\mu_n(v)=1-2/3=1/3$ on the
diagonal.  For two different adjacent vertices $v, w \in X_n$
($v \ne w$), we have
\begin{equation*}
  -\frac{\gamma_n(v,w)}{\sqrt{\mu_n(v)\mu_n(w)}}
  = \begin{cases}
  -1/6, & v,w \in X_n \setminus X_0,\\
  -1/(3\sqrt 2)& \text{$v$ or $w$ is in $X_0=\bd X_\infty$}
  \end{cases}
\end{equation*}
(see e.g.~\cite[Sec.~2.4]{fclp:18}).  Usually, the spectrum of a
Laplacian with \emph{normalised} weight is contained in $[0,2]$.
Here, the spectrum of $\Delta_{G_n^\circ}$ is in $[0,1]$; this is due
to the fact that there are loops at each vertex.  Note that $0$ is
always in the spectrum of $\Delta_{G_n^\circ}$ (with constant
eigenvector $\1_{X_n}$).

The defect operator then is
\begin{equation}
  \label{eq:wn.graph2}
  W_n := (\id_{X_n} - J_n^*J_n)^{1/2}
  = \Delta_{G_n^\circ}^{1/2},
\end{equation}
its spectrum is given by the square roots
of~\eqref{eq:delta.circ.spec}.  In particular, the kernel is the same
as the kernel of $\Delta_{G_n^\circ}$; and we have
\begin{align}
  \label{eq:ran.wn}
  \ran W_n
  = (\ker W_n)^\perp
  =\lsqr{X_n,\mu_n} \ominus \C \1_{X_n}.
\end{align}
\subsubsection*{The second defect operator}
The other defect operator
$W_{\infty,n}=(\id_{\HS_\infty}-J_nJ_n^*)^{1/2}$ is given by
\begin{equation}
  \label{eq:w-infty.graph}
  W_{\infty,n}^2f_\infty
  =(\id_{\HS_\infty}-J_nJ_n^*)f_\infty
  = \sum_{v \in X_n} \Bigl(f_\infty
  - \frac1{\mu_n(v)} \iprod[\HS_\infty]{f_\infty}{\psi_{n,v}}
  \Bigr)\psi_{n,v}
\end{equation}
as $f_\infty=\sum_{v \in X_v} f\psi_{n,v}$ pointwise
$\mu_\infty$-almost surely again due to~\eqref{eq:part.uni}.  Note
that in our unit interval example, we can think of $J_nJ_n^*f_\infty$
as a linear spline approximation of $f_\infty$, hence
$W_{\infty,n}^2 f_\infty$ measures the error made by this
approximation.

Moreover, from \Lemenum{defect.ops}{defect.ops.a} we conclude
\begin{align*}
  \spec{W_{\infty,n}}
  &=\spec{W_n}  \cup \{1\}
    =\spec{\Delta_{G_n^\circ}^{1/2}}  \cup \{1\}.
\end{align*}
Note that on
$\HS_\infty \ominus J_n(\HS_n)=\set{\psi_{n,v}}{v \in X_n}^\perp$, the
defect operator $W_{\infty,n}$ acts as identity, hence the extra
eigenvalue $1$.

\subsubsection*{The minimal parent space and corresponding isometries}
The minimal abstract parent space is given by
\begin{equation*}
  \HSmin = \Lsqr{X_\infty}
   \oplus \bigoplus_{n \in \N}
   \bigl(\lsqr{X_n,\mu_n} \ominus \C \1_{X_n}\bigr)
\end{equation*}
(cf.\ \Lem{hs.min} and~\eqref{eq:ran.wn}).  Moreover,
$\map{\iota_\infty}{\Lsqr{X_\infty}}\HSmin$ and
$\map{\iota_n}{\Lsqr{X_n}}\HSmin$ are the corresponding isometries
given by $\iota_\infty f_\infty=(f_\infty,0,\dots)$ and
\begin{equation*}
  \iota_n f_n
  =\Bigr(\sum_{v \in X_n} f_n(v) \psi_{n,v},0,\dots,0,
  \Delta_{G_n^\circ}^{1/2}f_n ,0,\dots\Bigr),
\end{equation*}
respectively.  Again, $\iota_n$ is not positivity-preserving (but
$J_n=\iota_\infty^* \iota_n$ is).  Moreover, we have
\begin{align*}
  (P_nP_\infty-P_\infty P_n)f
  &=\Bigl(
    -\sum_{v \in X_n} (\Delta_{G_n^\circ}^{1/2} f_n)(v) \psi_{n,v},
    0, \dots, 0,
    \Delta_{G_n^\circ}^{1/2}g_n,
    0, \dots
    \Bigr)\\
  \text{with}\quad
    g_n&=J_n^* f_\infty
    =\Bigl(\frac 1{\mu_n(v)}\iprod[\HS_\infty]{f_\infty}{\psi_{n,v}}\Bigr)_{v \in X_n}
\end{align*}
by \Lem{proj.commut}.

\subsubsection*{Weidmann's convergence}
As the projections here do not commute, we conclude from the
QUE-convergence with convergence speed of order $\Err(2^{-n})$ and
\Thm{main2a}:
\begin{corollary}[Weidmann's convergence for discrete graphs
  converging to the unit interval]
  \label{cor:main.ex.int}
  We have $\D_n \gnrcW \D_\infty$ with convergence speed of
  order $\delta_n^{1/2} \in \Err(2^{-n/2})$.
\end{corollary}
As a consequence of QUE- or Weidmann's convergence, the eigenvalues
converge: The spectrum of $\D_n$ and $\D_\infty$ are given by
\begin{align*}
  \spec{\D_n}
  &= \spec{2 \cdot 4^n \Delta_{G_n}}
  = \Bigset
  {2 \cdot 4^n \Bigl(1-\cos\Bigl(\frac{k\pi}{2^n}\Bigr)\Bigr)}
  {k = 0, \dots, 2^n} \quad\text{and}\\
  \spec {\D_\infty}&= \bigset {k^2\pi^2}{k = 0, \dots, 2^n},
\end{align*}
and an eigenvalue $2 \cdot 4^n(1-\cos(k\pi/2^n))= k^2\pi^2+\Err(k^4
4^{-n})$ of $\D_n$ converges to an eigenvalue $k^2\pi^2$ of
$\D_\infty$ for $k \in \{0,\dots,2^n\}$ fixed as $n \to \infty$.

\subsubsection*{Better convergence speed for Weidmann's convergence?}
Let us have a closer look at the convergence speed.  In the proof of \Thm{main2a},  the loss of
convergence speed $\delta_n^{1/2}$ is due to the fact that we only
estimate
\begin{equation}
  \label{eq:no-loss-of-conv-speed}
  \norm{P_\infty^\perp D_n P_n} \le \norm{W_n^2R_n}^{1/2}
  \quadtext{instead of}
  \norm{P_\infty^\perp D_n P_n}
  =\norm{R_n^* W_n^2R_n}^{1/2}
  =\norm{W_nR_n}
\end{equation}
in~\eqref{eq:main2a.est1}, and similarly for $P_\infty D_n P_n^\perp$
in~\eqref{eq:main2a.est2}.  Here, we have the special situation that
$W_n$ and $R_n$ are functions of the same operator $\Delta_{G_n}$,
namely $W_n=w_n(\Delta_{G_n})$ and $R_n=r_n(\Delta_{G_n})$ with
\begin{equation*}
  w_n(\lambda)=\sqrt {\lambda/3}
  \quadtext{and}
  r_n(\lambda)=\frac 1{1+ 2 \cdot 4^n\lambda}.
\end{equation*}
In particular, we can easily calculate the order of several operator
norms:
\begin{lemma}
  \label{lem:norm.est.ex}
  We have
  \begin{subequations}
  \begin{align}
    \label{eq:norm.est.ex}
    \norm{P_\infty^\perp D_n P_n}
    &= \norm{W_nR_n}
      \in \frac{2^{-n}}{\sqrt 3(1+\pi^2)} + \Err(4^{-n})
      \subset \Err(2^{-n}),\\
    \label{eq:norm.est2.ex}
        \norm{(\id_{\HS_n}-J_n^*J_n)R_n}
    &= \norm{W_n^2 R_n}
    = \Bigl(\frac2 {3(1+4^{n+1})}\Bigr)^{1/2}
    \in \Err(4^{-n}),\\
    \label{eq:est.qf}
    \norm{(\id_{\HS_n}-J_n^*J_n)R_n^{1/2}}
    &= \norm{W_n^2 R_n^{1/2}}
    = 2/(3\sqrt{1+4^{n+1}})
    \in \Err(2^{-n}).
  \end{align}
\end{subequations}
\end{lemma}
\begin{proof}
  For the first asymptotic expansion, we have set
  $\eta_n(\lambda) := w_n(\lambda) r_n(\lambda)$, then
  $\norm{W_nR_n} = \sup \eta_n(\spec {\Delta_{G_n}})$.  The function
  $\eta_n$ has its maximal value $1/(\sqrt 3 \cdot 2^{n+1})$ at
  $\lambda=1/(2\cdot 4^n)$ as the function $\eta_n$ is monotonously
  decreasing for $\lambda>1/(2\cdot 4^n)$.  The first non-zero
  eigenvalue $\lambda_1(G_n)$ of $\Delta_{G_n}$ is larger than
  $1/(2 \cdot 4^n)$ (we have
  $\lambda_1(G_n)=1-\cos(\pi/2^n) \in \pi^2/(2\cdot 4^n) +
  \Err(16^{-n})$), the supremum $\sup \eta_n(\spec {\Delta_{G_n}})$ is
  achieved at $\lambda_1(G_n)$, so we have
  $\sup \eta_n(\spec {\Delta_{G_n}}) \in \eta_n(\pi^2/(2\cdot 4^n)) +
  \Err(4^{-n})$.

  For the second asymptotic expansion, we set
  $\xi_n(\lambda)=w_n(\lambda)^2r_n(\lambda)$, then then
  $\norm{W_n^2R_n} = \sup \xi_n(\spec {\Delta_{G_n}})$.  The function
  $\xi_n$ is monotonously increasing, hence the supremum is actually
  achieved at $\lambda=2$, i.e.,
  $\sup \xi_n(\spec {\Delta_{G_n}})=\xi_n(2)$.

  The last asymptotic expansion can be seen similarly: the function
  $\lambda \mapsto \lambda/(3\sqrt{1+2\cdot 4^n \lambda})$ is
  monotonously increasing, and hence achieves its maximum on
  $\spec{\Delta_{G_n}}$ at $\lambda=2$).
\end{proof}

\begin{remark}[no loss of convergence speed in the ``bad''
  estimate~\eqref{eq:main2a.est1} in \Thm{main2a}]
  \label{rem:estimate.not.better}
  \indent

  \noindent
  The cause of the square root in the convergence speed lies in the
  two ``bad'' estimates
  ~\eqref{eq:main2a.est1}--~\eqref{eq:main2a.est2} in the proof of
  \Thm{main2a}.  Nevertheless, using \Lem{norm.est.ex}, we see that
  in~\eqref{eq:no-loss-of-conv-speed} in this example, we do not loose
  convergence speed as both terms are of order $\Err(2^{-n})$.
\end{remark}
\begin{remark}[optimality of estimate in the (modified) QUE-setting
  for the unit interval example]
  \label{rem:opt.unit.int}
  \indent

  \noindent
  In~\cite{post-simmer:18,post-simmer:19} we have actually shown the
  norm estimate
  \begin{equation*}
    \norm[\HS_n]{(\id_{\HS_n}-J_n^*J_n)R_n^{1/2}}
    \in \Err(2^{-n})
  \end{equation*}
  in the unit interval example.  In~\eqref{eq:est.qf} we have seen
  that $\norm{(\id_{\HS_n}-J_n^*J_n)R_n^{1/2}} \in \Err(2^{-n})$, and
  the estimate is sharp.

  Since we have also shown in~\cite{post-simmer:18,post-simmer:19}
  that the remaining estimates are of the same order, i.e.,
  \begin{equation*}
    \norm{(\id_{\HS_\infty}-J_nJ_n^*)R_\infty^{1/2}}
    \in \Err(2^{-n})
    \quadtext{and}
    \norm{R_\infty J_n - J_n R_n} \in \Err(2^{-n}),
  \end{equation*}
  the modified QUE-estimates of~\cite{post-simmer:18,post-simmer:19}
  (with the resolvent powers $R_n^{1/2}$ and $R_\infty^{1/2}$ instead
  of $R_n$ and $R_\infty$ as in~\eqref{eq:que2}) are \emph{sharp}.
\end{remark}

\begin{remark}[resolvent power makes a difference]
  \label{rem:res.power.makes.difference}
  It makes a difference if one uses only the resolvent $R_n^{1/2}$
  instead of $R_n$ in~\eqref{eq:que2}, namely we have
  \begin{align*}
    \norm{(\id_{\HS_n}-J_n^*J_n)R_n}
    &=\norm{W_n^2R_n}
      \in \Err(4^{-n}),
      \quad\text{and}\\
    \norm{(\id_{\HS_n}-J_n^*J_n)R_n^{1/2}}
    &=\norm{W_n^2R_n^{1/2}}
      \in \Err(2^{-n})
  \end{align*}
  from~\eqref{eq:norm.est2.ex}--\eqref{eq:est.qf}.  We believe that
  better estimates can be done also for other post-critically finite
  fractals as considered in~\cite{post-simmer:18,post-simmer:21b}.  We
  are not aware if this also makes a difference in the graph-like
  spaces example, see
  \Remenum{change.of.que-def.ex}{change.of.que-def.ex.d}.

  In order to get the better convergence speed $2^{-n}$ in Weidmann's
  convergence (and not $2^{-n/2}$ as in \Cor{main.ex.int}) one should
  also analyse the second ``bad'' term~\eqref{eq:norm.summand3},
  namely
  \begin{equation*}
    \norm{P_\infty D_n P_n^\perp}
    = \norm{R_\infty^* W_{\infty,n}^2 R_\infty}^{1/2}.
  \end{equation*}
  We have shown in~\cite{post-simmer:18,post-simmer:21b} that
  $\norm[\Lin{\HS_\infty}]{W_{\infty,n}^2 R_\infty^{1/2}}$ is of order
  $2^{-n}$ in the unit interval example.  We believe that (at least in
  the unit interval example) one can show that
  $\norm[\Lin{\HS_\infty}]{W_{\infty,n}^2 R_\infty}$ is of order
  $4^{-n}$, i.e., that there is a constant $C>0$ such that
  \begin{equation}
    \label{eq:est.op}
    \norm[\Lsqr {X_\infty}]{f_\infty-J_n J_n^* f_\infty} \le C 4^{-n}
    \norm[\Lsqr{X_\infty}]{-f_\infty''+f_\infty}
  \end{equation}
  for all $n \in \N$ and all $f_\infty \in \dom \D_\infty$.  Recall
  that $J_nJ_n^*f_\infty$ is an approximation of $f_\infty$ with a
  linear spline.  We will not go into details here.

  If~\eqref{eq:est.op} holds, then we are able to show that $\D_n
  \gnrcW \D_\infty$ with convergence speed of order $2^n$ as in the
  QUE-convergence.  We will treat related questions in a subsequent
  publication.

\end{remark}

\newpage
%
%


\providecommand{\bysame}{\leavevmode\hbox to3em{\hrulefill}\thinspace}
\providecommand{\MR}{\relax\ifhmode\unskip\space\fi MR }
\providecommand{\MRhref}[2]{%
  \href{http://www.ams.org/mathscinet-getitem?mr=#1}{#2}
}
\providecommand{\href}[2]{#2}

\end{document}